\documentclass[11pt,letterpaper]{amsart}
\usepackage[T1]{fontenc}
\usepackage[utf8]{inputenc}
\usepackage[english]{babel}
\usepackage{amsmath}
\usepackage{amsfonts}
\usepackage{amssymb}
\usepackage{amsthm}
\usepackage{braket}
\usepackage{bm}
\usepackage{tikz}
\usetikzlibrary{shapes,snakes}
\usepackage{graphicx}
\usepackage{marginnote}
\usepackage{enumitem}
\usepackage{mathtools}
\usepackage[margin=1.5in]{geometry}

\newcommand{\bigdual}[2]{\mathinner{\bigl\langle{#1}\bigl|{#2}\bigr.\bigr\rangle}}

\newcommand{\dual}[2]{\braket{#1{|}#2}}

\usepackage{color}
\usepackage[colorlinks=true, pdfstartview=FitV, linkcolor=darkblue, citecolor=darkred, urlcolor=cyan]{hyperref}
\definecolor{darkred}{RGB}{203,65,84}
\definecolor{darkblue}{RGB}{70,130,180}
\definecolor{brown}{RGB}{139,69,19}

\usepackage[stretch=16,shrink=16,step=2,kerning=true,protrusion=true,final]{microtype}

\newtheorem{theorem}{Theorem}[section] 
\newtheorem{lemma}[theorem]{Lemma}
\newtheorem{proposition}[theorem]{Proposition}
\newtheorem{corollary}[theorem]{Corollary}

\newcounter{foo}
\newtheorem{theo}[foo]{Theorem}
\newtheorem{coro}[foo]{Corollary}

\theoremstyle{definition}
\newtheorem{definition}[theorem]{Definition}
\theoremstyle{remark}

\newtheorem{remark}[theorem]{Remark}
\newtheorem{remarks}[theorem]{Remarks}
\newtheorem{examples}[theorem]{Examples}
\newtheorem{example}[theorem]{Example}

\newcommand{\grf}{\pi_1(\Sigma)}

\newcommand{\sld}{\mathsf{SL}_2(\mathbb R)}

\newcommand{\gl}{\mathsf{GL}}
\newcommand{\GL}{\mathsf{GL}}
\newcommand{\pgl}{\mathsf{PGL}}
\newcommand{\PGL}{\mathsf{PGL}}
\newcommand{\psld}{\mathsf{PSL}_2(\mathbb R)}
\newcommand{\slm}{\mathsf{SL}_m(\mathbb R)}

\newcommand{\skd}{\mk{sl}_2}
\newcommand{\ad}{\operatorname{ad}}
\newcommand{\Ad}{\operatorname{Ad}}

\renewcommand{\hom}{\operatorname{Hom}}

\newcommand{\Rp}{{\mathbf{ P}}^1(\mathbb R)}

\renewcommand{\d}{{\rm d}}

\newcommand{\sbt}{\,\begin{picture}(-1,1)(-1,-1)\circle*{2}\end{picture}\ }
\renewcommand{\dot}[1]{\overset{\sbt}{#1}}

\newcommand{\ms}{\mathsf}
\newcommand{\mk}{\mathfrak}
\newcommand{\mG}{\ms G}
\newcommand{\mK}{\ms K}
\newcommand{\mP}{\ms P}
\newcommand{\mU}{\ms U}
\newcommand{\mS}{\ms S}
\newcommand{\mL}{\ms L}
\newcommand{\mN}{\ms N}
\newcommand{\mB}{\ms B}
\newcommand{\mH}{\ms H}

\newcommand{\mW}{\ms W}
\newcommand{\mV}{\ms V}

\newcommand{\mM}{\ms M}

\newcommand{\cF}{{\mathcal F}}
\newcommand{\cG}{\mathcal G}
\newcommand{\cGG}{\mathcal I}
\newcommand{\cL}{\mathcal L}

\newcommand{\cS}{\mathcal S}

\newcommand{\bP}{\mathbf P}

\newcommand{\T}{\ms T}

\newcommand{\defeq}{\coloneqq}

\renewcommand{\leq}{\leqslant}
\renewcommand{\geq}{\geqslant}

\renewcommand{\epsilon}{\varepsilon}
\renewcommand{\phi}{\varphi}

\newcommand{\mGG}{\mathcal G}

\newcommand{\seq}[1]{ \{{#1}_m\}_{m\in\mathbb N}}

\newcommand{\mappingnew}[4] { 
    \begin{array}{rcl}
      #1 &\longrightarrow& #2\\
      #3 &\longmapsto& #4
    \end{array}}

\newcommand{\bPhi}{\boldsymbol{\Phi}}

\newcommand{\bb}{{\rm b}}
\newcommand{\pp}{{\rm p}}

\title[Positivity and the Collar Lemma]{Positivity, cross-ratios, photons, and the Collar Lemma}

\author[J.~Beyrer  \and  O.~Guichard  \and  F.~Labourie \and   B.~Pozzetti
\and  A.~Wienhard]{Jonas Beyrer \and Olivier Guichard \and François Labourie
  \and  \\ Beatrice  Pozzetti \and  Anna Wienhard}
\AtBeginDocument{%
   \def\MR#1{}
}

\begin{document}
\thanks{F.L.\ acknowledges funding by the European Research Council under ERC-Advanced grant 101095722. B.P.\ acknowledges funding by the DFG, 427903332 (Emmy Noether). A.W.\ was supported by the European Research Council under ERC-Advanced Grant 101018839. A.W.\ also thanks the Hector Fellow Academy  for support. This work was supported by the Deutsche Forschungsgemeinschaft under Germany’s Excellence Strategy EXC- 2181/1 - 390900948 (the Heidelberg STRUCTURES Cluster of Excellence). }
\begin{abstract}
	We prove that $\Theta$-positive representations of fundamental groups of surfaces (possibly cusped or of infinite type) satisfy a collar lemma, and their associated cross-ratios are positive. As a consequence we deduce that $\Theta$-positive representations form closed subsets of the representation variety.
	Along the way we
systematically study a class of curves inside general flag manifolds that we call photons. Using their
interplay with transversality in the flag manifold, we construct photon projections, which we use to relate
the cross-ratios on the photons to the cross-ratios on the flag manifolds.
\end{abstract}

\maketitle

\setcounter{tocdepth}{2}
\tableofcontents{}

\section*{Introduction}
The use of cross-ratios in hyperbolic dynamics was initiated by Jean-Pierre Otal \cite{Otal:1990th} and notably used by Ursula Hamenstädt \cite{Hamenstadt:1997} and François Ledrappier \cite{Ledrappier:1995}. For the purpose of this introduction, let us recall 
that they consider real valued functions of generic quadruples of points in
the boundary at infinity $\partial_\infty\Gamma$ of a hyperbolic group~$\Gamma$ satisfying certain additive (or multiplicative) cocycle identities. 

These additive functions arise as logarithms of what we will call cross-ratios in this paper. Given a cross-ratio~$\bb$, and a non-trivial element~$\gamma$ in~$\Gamma$, the {\em period} of~$\gamma$ is
$
\pp(\gamma)\defeq\bb(\gamma^+,\gamma^-,x, \gamma(x))$,
where~$\gamma^-$ and~$\gamma^+$ are respectively the repelling and attracting fixed points of~$\gamma$ in $\partial_\infty\Gamma$ and $x$~is any  element of $\partial_\infty\Gamma$ not fixed by~$\gamma$. In the context of plane hyperbolic geometry, when $\Gamma$~is the fundamental group of a closed hyperbolic surface~$S$, $\partial_\infty\Gamma$ is identified with the projective line over $\mathbb R$ and the period of the projective cross-ratio of an element $\gamma$ in $\Gamma$ is the exponential of the length of the associated closed geodesic on $S$.
For Anosov representations cross-ratios have been introduced by François Labourie in \cite{Labourie:2005}  (see Section~\ref{sec:sympl-reint}); they have since become a standard tool.

In this paper we concern ourselves with special classes of discrete subgroups of semisimple Lie groups, the images of the so-called positive representations \cite{Guichard:2021aa}.
It has been recognized that for special families of Lie groups some classes of representations of
fundamental groups of surfaces have a unique behavior: when the surface is closed, they
form entire connected components of the space of homomorphisms; as a corollary
one finds components consisting entirely of discrete and faithful representations. 
For $\psld$, positive representations are precisely the holonomy representations of hyperbolic structures. 

More generally, for split real Lie groups, Hitchin representations
give rise to connected components consisting entirely of discrete and faithful
representations, see \cite{Fock:2006a, Labourie:2006}, and the same is true for maximal representations in groups of
Hermitian type
\cite{Burger:2010ty, Burger:2005}. Even though these spaces of representations
were introduced and investigated using very different techniques, Olivier Guichard and
Anna Wienhard unveiled a common structure ---called  
{\em $\Theta$-positivity}---
that underlies  them all \cite{GuichardWienhard_pos} and generalizes the total
positivity à la Lusztig \cite{Lusztig:1994}, which played a central role in
work of Vladimir Fock and Alexander Goncharov \cite{Fock:2006a}. Lie groups admitting a positive
structure relative to~$\Theta$ 
exist beyond the above-mentioned examples, see \cite{GuichardWienhard_pos},
and Guichard, Labourie, and Wienhard 
started the study of $\Theta$-positive representations of surface groups in
\cite{Guichard:2021aa}. The notion of  $\Theta$-Positivity is defined relative to a subset
$\Theta$ of simple roots (or equivalently with respect to the choice of a
conjugacy class of parabolic groups), and it is shown in
\cite{Guichard:2021aa}  that, for closed surface groups, $\Theta$-positive
representations are in particular Anosov with respect to $\Theta$. For the
purpose of this  introduction, we don't recall the precise definition of
$\Theta$-positivity, but just recall that it can be described by a subset of
quadruples, called {\em positive quadruples}, in the flag manifold
$\cF_\Theta$ defined by~$\Theta$, and positive representations
preserve
a cyclically ordered subset of~$\cF_\Theta$ for which 
ordered triples and quadruples are positive (see Section~\ref{sec:def-pos}).

In this paper we explore cross-ratios on general flag manifolds and in particular cross-ratios associated to
$\Theta$-positive representations. We show that these cross-ratios are
positive, namely the cross-ratio of a cyclically oriented quadruple is bigger than~$1$, so its logarithm is positive. 
We further prove a
Collar Lemma in the spirit of hyperbolic geometry.  We use these results in combination
with a result from \cite{Guichard:2021aa}, to show that the space of
$\Theta$-positive representations of the fundamental group of a closed surface
is open and closed in the space of homomorphisms, thus establishing a
conjecture of Guichard, Labourie and Wienhard (cf.\ \cite{Guichard:2021aa, GuichardWienhard_pos, Wienhard_Guichard_ECM, 
	Wienhard_ICM}), extending previous results of the authors in
\cite{Guichard:2021aa} and \cite{BeyrerPozzetti}. Along the way we introduce
new objects ---called {\em photons}--- that might be useful for the study of
surface group representations in wider settings. 
\medskip

Let us describe the results in more detail. Recall that the projective cross-ratio on the projective space associates to a pair $(x,y)$ of lines and a pair $(X,Y)$ of hyperplanes  with suitable transversality properties the real number 
\[
\bb(x,y,X,Y)\defeq\frac{\dual{\bar x}{\bar X}\dual{\bar y}{\bar  Y}}{\dual{\bar y}{\bar X}\dual{\bar x}{\bar Y}}\ ,
\]
where $\bar x$ and  $\bar y$ are (any) non-zero elements in $x$ and $y$ respectively, and $\bar X$ and $\bar Y$ are any linear forms whose kernel are $X$ and $Y$ respectively. 

Let now $\Theta$ be a subset of the set $\Delta$ of simple roots of
$\mathsf G$. Let $\cF_\Theta$ be the associated  flag
manifold. 
For an element
$\theta$ of $\Theta$, let us assume that the fundamental weight $\omega_\theta$
defines a (projective) representation of $\mathsf G$ on  a vector space $V$ with $\mathsf
G$-equivariant maps $\Xi_\theta$  and $\Xi_\theta^*$ from $\cF_\Theta$ and \(\mathcal{F}_{\Theta}^{\mathrm{opp}}\) to the
projective space of $V$ and its dual (in general only a multiple of \(\omega_\theta\) will satisfy this). The {\em associated cross-ratio} on
$\cF_\Theta$ is 
$$
\bb^{\omega_\theta}(x,y,z,w)\defeq \bb(\Xi_\theta(x),\Xi_\theta(y),\Xi^*_\theta(z),\Xi^*_\theta(w))\ .
$$

Our first result is

\begin{theo}[\sc Positivity of the cross-ratio]\label{theo:A}
	Let $\ms G$ be a semisimple group admitting a positive structure relative to~$\Theta$,
	$\cF_\Theta$ be the  flag manifold associated to~$\Theta$,
	$\omega_\theta$ the fundamental weight associated to $\theta$ in $\Theta$. Then for every  positive quadruple $(x,y,X,Y)$ in $\cF_\Theta^4$
	\[\bb^{\omega_\theta} (x,y,X,Y)>1\ . 
	\]
\end{theo}

In Section~\ref{sec:cross-ratio} we extend the construction  
without assuming that \(\mathsf{G}\)~admits a positive structure relative to~\(\Theta\); there cross-ratios are defined for every proximal dominant weight that is a positive linear combination of the fundamental weights $\omega_\theta$
for $\theta$ in $\Theta$ (we call those $\Theta$-compatible proximal dominant weights,
cf.\ Section~\ref{sec:line-forms-posit}). Theorem~\ref{theo:A} is proved in
Section~\ref{sec:posit-cross-ratio}.

Special cases of this theorem were known before. This was established 
 in
\cite{Labourie:2005,Martone:2019aa} 
for the cross-ratios of
$\mathsf{PSL}_n(\mathbb R)$, in \cite{Burger:2005} for maximal representations
in $\mathsf{Sp}(2n,\mathbb R)$, and in \cite{BeyrerPozzetti} for quadruples in
the limit curve of a $\Theta$-positive representations in $\mathsf{SO}(p,q)$.

In \cite{Bridgeman:2023} Theorem~\ref{theo:A} is used by Martin Bridgeman and Labourie to obtain the convexity of length functions on moduli spaces of positive representations. In \cite{Labourie:2009wi} Labourie and Greg McShane showed that positive cross-ratios imply the existence of generalized McShane--Mirzakhani identities. In particular, Theorem~\ref{theo:A} implies that Theorem~1.0.1 of  \cite{Labourie:2009wi} holds for all positive representations. 
In \cite{Labourie:2026aa} Labourie, Jérémy Toulisse and Yilin Wang use the positivity of the cross-ratio to define a Liouville action for positive curves and prove it is finite for piecewise circles. 

In \cite{Martone:2019aa} Giuseppe Martone and Tengren Zhang introduced the notion of
\emph{positively ratioed} representations with respect to a parabolic subgroup
$\mP_\Theta$, and showed that the set of positively ratioed representations
admits appropriate embeddings into the space of geodesic currents. 
Theorem~\ref{theo:A} implies: 
\begin{coro}\label{coro:posratio}
	Let $\ms G$ be a semisimple group admitting a positive structure relative to~$\Theta$.  Let~$\rho$ be a $\Theta$-positive representation of a closed surface group, then $\rho$~is $\mP_\Theta$-positively ratioed.
\end{coro}

Theorem~\ref{theo:A} is also a crucial ingredient for the following Collar
Lemma.

For every  $\eta$ in $\mathsf G$ with attracting and repelling fixed points $\eta^+$ and $\eta^-$ in $\cF_{\Theta}^{\mathrm{opp}}$, the {\em period} of $\eta$ with respect to the  cross-ratio $\bb^{\omega_\theta}$ is  
$$\pp^{\omega_\theta}(\eta)\defeq \bb^{\omega_\theta}(\eta^+,\eta^-,y,\eta(y))\ ,$$ 
where one checks that the right-hand term does not depend on the choice of
$y$ in $\cF_\Theta$ transverse to $\eta^+$ and $\eta^-$.

Any linear form $\lambda$ in $\mathfrak a^*$ gives rise to a
character $\chi_{\lambda}\colon \mG \rightarrow \mathbb{R}$ given by $\chi_\lambda(\eta)\defeq\exp(\dual{h}{\lambda})$ where $h$~is the Jordan projection of~$\eta$ in the Weyl chamber~$\mathfrak a^+$ of~$\mathsf G$.  When $\lambda$ is a weight,  periods and characters are related by the following formula (Proposition~\ref{prop:periods-loxo}):
$$
|\pp^{\lambda}(g)|=\chi_\lambda(g)\chi_\lambda(g^{-1})\ .
$$

\begin{theo}[\sc Collar Lemma in the Lie group]\label{theo:Bprime}
	Let $\mG$ be a semisimple Lie group admitting a positive
	structure relative to~$\Theta$. Let~$A$ and~$B$ be $\Theta$-loxodromic elements of~$\mG$. Denote
	by $(a^+, a^-)$ and $(b^+, b^-)$ the pair of attracting and repelling fixed
	points of~$A$ and~$B$ respectively in the flag variety~$\cF_\Theta$. Assume
	that the sextuple
	$$(a^+, b^-, a^-, b^+, B(a^+), A(b^+))$$
	is positive. Then for any $\theta$ in $\Theta$, the following  holds
	\begin{equation*}
		\frac{1}{\pp^{\omega_\theta}\left(B\right)}
		+\frac{1}{\chi_{\theta} \left(A\right)}< 1\ .
	\end{equation*}
\end{theo}

When $S$ is an oriented surface of negative Euler characteristic (not
necessarily of finite type) we obtain the following consequence:
\begin{coro}[\sc Collar Lemma]
	\label{coro:B}
	Let  $\mG$ be a semisimple Lie group admitting a positive structure relative to~$\Theta$.
	Let $\rho\colon \pi_1(S)\to\mathsf G$ be a $\Theta$-positive homomorphism. Let $\gamma_0$ and $\gamma_1$ be elements of
	$\pi_1(S)$ whose associated free homotopy classes intersect geometrically.
	Let  $\theta$  be in $\Theta$,
	then
	\begin{equation*}
		\frac{1}
		{\pp^{\omega_\theta}\left(\rho(\gamma_0)\right)}
		+\frac{1}{\chi_{\theta}
			\left(\rho(\gamma_1)\right)}< 1\ .
	\end{equation*}
\end{coro}

The first Collar Lemma for representations of fundamental groups of closed surfaces  in groups of higher rank is a generalization of the
Hyperbolic Collar Lemma (cf.\ \cite{Keen}) and  is due to Gye-Seon Lee and Zhang
\cite{Lee-Zhang:2017}. Since this seminal work, the subject of Collar Lemmas has attracted a lot of attention. We discuss in Section~\ref{sec:Collar-discuss} the relation of  our work with the works of Marc Burger and Beatrice Pozzetti \cite{Burger-Pozzetti:2017}, Federica Fanoni and Pozzetti \cite{Fanoni-Pozzetti:2020}, Jonas Beyrer and Pozzetti \cite{Beyrer:2021aa, BeyrerPozzetti}, Nicolas Tholozan \cite{Tholozan:2017} and Brian Collier, Tholozan and Toulisse \cite{Collier:2019aa}.

Combining Corollary~\ref{coro:B} with a result of \cite{Guichard:2021aa}, we
deduce the closedness of positive representations of surface groups into Lie
groups admitting a positive structure relative to~$\Theta$.

\begin{coro}\label{cor:limit-of-positive-hom} Let $\mG$ be a semisimple Lie
	group admitting a positive structure relative to~$\Theta$. 
	If $\seq{\rho}$ is a sequence of $\Theta$-positive homomorphisms from a surface group to $\mG$ converging to a homomorphism $\rho$, then $\rho$ is $\Theta$-positive.
\end{coro}

We insist that in this corollary as well as in Corollary~\ref{coro:B}, we do not
restrict ourselves to 
closed surfaces nor to surfaces of finite type.

As a corollary, combining with \cite[Corollary C]{Guichard:2021aa}, we obtain
a solution to the mentioned conjecture of Guichard, Labourie, and Wienhard refining results in \cite{Guichard:2021aa} and generalizing results of \cite{BeyrerPozzetti}.
\begin{coro}
	The set of $\Theta$-positive representations is a union of connected
	components of the space of all representations of a closed surface group.
\end{coro}

In order to prove Theorem \ref{theo:A}, we investigate the symplectic geometry of products of flag manifolds. The proof of the Collar Lemma itself relies on the positivity of the cross-ratio as well as a new tool: the study of   {\em $\theta$-photons} for $\theta$ in $\Theta$, a study that we hope will be useful in future research.

We summarize briefly the construction of $\theta$-photons given in details in Section~\ref{sec:photo-construct}. Associated to a root $\theta$ in $\Theta$ is a
conjugacy class of subgroups $\mH_\theta$ in $\mG$ isogenic to $\psld$. A
$\theta$-photon is a closed orbit of a group $\mH_\theta$, hence isomorphic
to $\Rp$ (Proposition~\ref{pro:PhoRP}). In the special case of
the Hermitian group $\mG=\mathsf{SO}(2,n)$, $\Theta$ is reduced to one element,
the $\Theta$-flag manifold is the {\em Einstein universe}  equipped with a
conformal structure of type $(1,n-1)$, and the $\theta$-photons are the  closed
lightlike geodesics as described for instance in \cite{Barbot:2008aa}  and
\cite{Collier:2019aa}. Given a $\theta$-photon $\Phi_\theta$, we construct a
projection from an open set of the flag manifold to $\Phi_\theta$ and are able
to use this projection to relate the classical (projective) cross-ratio on the
photon to the cross-ratio introduced in Theorem~\ref{theo:A} and from this
deduce bounds on the character $\chi_\theta$ in terms of cross-ratios (Theorem
\ref{theo:sup-min}).

In the Appendix \ref{app:real}, we show how our results extend when
$\mathbb R$ is replaced by any ordered field containing~\(\mathbb{R}\). 

\subsection*{Thanks}
We thank Marc Burger, Blandine Galiay, and an anonymous referee for pointing out mistakes in previous versions of the paper. 

\section{Preliminaries}\label{sec:prelim}

In this section we recall some facts on semisimple Lie groups and Lie algebras and introduce some notation.  A classical reference for this material is \cite{Knapp_LieGrp}.

\subsection{Roots}
\label{sec:roots}
Let $\mG_0$ be a semisimple Lie group\footnote{In Section~\ref{sec:parab-subgr-flag} we will consider a group $\mG$ isomorphic to $\mG_0$ to have a treatment of flag manifolds suited to our purposes.} (by this we mean a connected Lie group
whose Lie algebra is semisimple) with finite center. Let~$\mk g_0$ be its Lie
algebra and~$\mk k$ the Lie algebra of a maximal compact subgroup~$\mathsf{K}$. The
associated Cartan involution~$\sigma$ is the Lie algebra involution of~$\mk
g_0$ whose fixed point set is equal to~$\mk k$. We fix a
Cartan subspace~$\mk a$ 
inside
the orthogonal complement of~$\mathfrak{k}$ with respect to the Killing form
on~$\mk g_0$. Throughout this article, scalar products, and in particular the
Killing form, as well as the induced forms on~$\mk a$ and on $\mk a^*$ will be
denoted by~$\braket{\cdot,\cdot}$.

Let $\Sigma$  be the subset  of $\mk a^*$ consisting  of the restricted roots of 
$\mG_0$:  $\Sigma$~is the set of non-zero weights for the adjoint action
of~$\mk a$ on the Lie algebra~$\mk g_0$ of~$\mG_0$; explicitly, $\beta$~belongs
to~$\Sigma$ if and only if $\beta\neq 0$ and
\[
\mk g_\beta\defeq\{v\in\mk g_0\mid \forall u\in \mk a, \, [u,v]=\beta(u) \, v\}
\] is not reduced to~$\{0\}$. We will often denote the
quantity $\beta(u)$ by $\dual{u}{\beta}$ and use a similar notation for every
duality.

Let $\Sigma^+$ be a fixed
choice of positive roots and $\Delta$ the corresponding set of simple roots. 
Later on, we will need to distinguish between the ``long roots'' and the
``short roots'', the understood notion of length behind this comes from the
Euclidean structure on~$\mk a^*$ induced by the Killing form.

\subsection{Weyl group}
The (closed) \emph{Weyl chamber} is the cone~$\mathfrak{a}^+$ in~$\mathfrak{a}$ defined
by the equations $\alpha(a)\geq 0$ for all~$\alpha$ in~$\Sigma^+$
(equivalently, for all~$\alpha$ in~$\Delta$).

The (restricted) \emph{Weyl group}~$\mW$ of~$\mG_0$ is the quotient of the
normalizer of~$\mk a$ in~$\mathsf{K}$ by the centralizer  of~$\mk a$ in~$\mathsf{K}$; it identifies with the subgroup of
$\gl(\mk a)$ stabilizing~$\Sigma$. The Weyl group is a Coxeter group
generated by hyperplane reflections $\{ s_\alpha\}_{\alpha\in \Delta}$
characterized (among hyperplane reflections that induce a permutation of~$\Sigma$)
by $s_\alpha( \alpha) = -\alpha$. We will sometimes use representatives
in~$\mK$ of elements~$s$ of~$\mW$, and we shall often denote
them~$\dot{s}$. 

The longest element~$w$ in the Weyl group~$\mW$ (with respect to the
generating family $\{s_\alpha\}_{\alpha\in \Delta}$) sends~$\Sigma^+$ to
$\Sigma^-\defeq\Sigma\smallsetminus \Sigma^+=-\Sigma^+$ and the mapping
$\iota\colon \alpha
\mapsto \iota(\alpha) \defeq -w\cdot \alpha$ induces a permutation of~$\Sigma^+$  and a permutation
of~$\Delta$, called the \emph{opposition involution}.

\subsection{$\skd$-triples}\label{sec:skd-triples}
For any positive root $\beta$, choose $(x_\beta,x_{-\beta},h_\beta)$ an
associated $\skd$-triple. This means that $h_\beta$ belongs to~$\mk a$, that $x_{\pm\beta}$ belongs to~$\mk g_{\pm\beta}$, and that the relations  $[h_\beta,
x_{\pm \beta}]=\pm 2x_{\pm \beta}$ and $[x_\beta, x_{-\beta}]=h_\beta$
hold. 
These elements can be used to construct representatives in~$\ms K$ of the
reflections~$s_\alpha$: indeed $\dot{s}_\alpha =
\exp(\frac{\pi}{2}(x_\alpha-x_{-\alpha}))$ represents the element~$s_\alpha$ 
of the Weyl group.

The family $\{h_\theta\}_{\theta\in\Delta}$ is a basis of the Cartan
subspace~$\mathfrak{a}$, the elements of the dual basis~$\{
\omega_\theta\}_{\theta\in\Delta}$ are called the \emph{fundamental weights}.

\subsection{Parabolic subgroups}
\label{sec:parab_sbgp}
Every subset $\Theta$ of $\Delta$ defines a parabolic subgroup~$\mP_\Theta$ in the following manner. First we consider
\[
\Sigma^+_\Theta\defeq \Sigma^+\smallsetminus\hbox{span}(\Delta\smallsetminus\Theta)\ .
\]
This is the set of positive roots whose decomposition as a sum of  simple
roots contains at least one element of $\Theta$. Equivalently,
$\Sigma_{\Theta}^{+}$ is the smallest subset of~$\Sigma$ containing~$\Theta$
and invariant by $\beta\mapsto \beta+\alpha$ for every~$\alpha$
in~$\Delta$. 
In particular $\Theta$~itself
is a subset of~$\Sigma_\Theta^+$. 
We set 
\[
\mk u_\Theta\defeq\bigoplus_{\alpha \in \Sigma^+_\Theta}\mk g_\alpha\ , \ \
\ \mk u^{\mathrm{opp}}_{\Theta}\defeq\bigoplus_{\alpha \in
	\Sigma^+_\Theta}\mk g_{-\alpha}\ .
\]
The  parabolic group $\mP_\Theta$ is the normalizer of $\mk u_\Theta$ in $\mG_0$. The unipotent
radical of~$\mP_\Theta$ is the group
$\mU_\Theta=\exp(\mathfrak{u}_\Theta)$. In this convention
	$\mP_{\emptyset}=\mG_0$, while $\mP_\Delta$ is the minimal parabolic.
	Similarly, we define
	the opposite parabolic subgroup~$\mP_{\Theta}^{\mathrm{opp}}$.
	Let $\mL_\Theta = \mP_{\Theta}\cap \mP_{\Theta}^{\mathrm{opp}} $ be
	the reductive part  in the Levi decomposition of $\mP_\Theta$ and
	$\mS_\Theta\defeq[\mL_{\Theta}^{\circ},\mL_{\Theta}^{\circ}]$ be the semisimple
	part of $\mL_{\Theta}^{\circ}$, the connected component of the identity
	of~$\mL_\Theta$.  The Lie algebra of~\(\mL_\Theta\) will be denoted \(\mk l_\Theta\). The opposite parabolic group~$\mP_{\Theta}^{\mathrm{opp}}$ is
	conjugate  to~$\mP_{\iota(\Theta)}$: for every representative~$\dot{w}$ of the
	longest element~$w$ of~$\mW$, one has $\dot{w}
	\mP_{\Theta}^{\mathrm{opp}} \dot{w}^{-1} = \mP_{\iota(\Theta)}$.
	A Cartan subspace of~$\mS_\Theta$ is
	\[
	\mk a_\Theta=\bigoplus_{\beta\in\Delta\smallsetminus\Theta}\mathbb R \ h_\beta\ ,
	\]
	and the Lie algebra of~$\mS_\Theta$ is
	\[
	\mk s_\Theta= \mk a_\Theta \oplus \mk m_\Theta \oplus\bigoplus_{\beta\in \Sigma \cap \operatorname{Span}(\Delta\smallsetminus\Theta)}\mk g_\beta 
	\ ,
	\]
	where $\mk m_\Theta$ is a compact Lie algebra.
	Hence, $\Delta\smallsetminus\Theta$ is a set of simple positive roots for
	$\mS_\Theta$ and the Dynkin diagram of~$\mS_\Theta$ is completely determined. 
	We have the following:
	\begin{proposition}\label{pro:ortho-atheta}
		Let $\mk b_\Theta$ be the orthogonal complement of $\mk a_\Theta$ in $\mk a$ (with respect to
		the Killing form). 
		Then $\mk b_\Theta$ is the intersection of the spaces $\ker(\beta)$ for
		$\beta$ varying in $\Delta\smallsetminus\Theta$. Moreover, the elements of
		$\mk b_\Theta$ commute with all elements of $\mk l_\Theta$, and
		$\mB_\Theta\defeq\exp({\mk b}_\Theta)$ is a central factor in
		$\mL_{\Theta}$.
	\end{proposition}
	
	\subsection{Irreducible factors for the action of the Levi} 
	\label{sec:special_subsets}
	For every~$\theta$ in~$\Sigma^{+}_{\Theta}$,  we set 
	\begin{align*}
		\Sigma_\theta&\defeq\bigl\{\beta \in\Sigma^+
		\mid\beta- \theta \in \operatorname{ Span}(\Delta\smallsetminus\Theta)\bigr\}\ ,	\\
		\mk u_ \theta&\defeq\bigoplus_{\alpha \in \Sigma_\theta}\mk g_\alpha \ , \ \ \mk u_{- \theta}\defeq\bigoplus_{\alpha \in \Sigma_\theta}\mk g_{-\alpha}\ .
	\end{align*}
	We will often use that, for an element~$\beta$ of~$\Sigma^+$, $\beta$~belongs
	to~$\Sigma_\theta$ if and only if $\beta-\theta$ is zero in restriction
	to~$\mk b_\Theta$.
	When $\theta$ belongs to $\Theta$,
	$\Sigma_\theta$ is the set of roots
	of the form 
	$
	\theta+\alpha
	$
	where $\alpha$ is  a linear combination of
	simple roots in $\Delta \smallsetminus \Theta$. Obviously $\Sigma_\theta$ is disjoint from  $\Sigma_\eta$ when
	$\eta$ and $\theta$ are distinct elements of~$\Theta$, and  for arbitrary $\eta$
	and $\theta$ in $\Sigma^{+}_{\Theta}$, the sets~$\Sigma_\theta$ and~$\Sigma_\eta$ are
	either disjoint of equal.

	The following result was established by Kostant \cite{Kostant:2010tr}.
	
	\begin{theorem}
		\label{theo:irred-fact-Theta}
		The subspaces $\mk u_\beta$ \textup{(}resp.\ $\mk u_{-\beta}$\textup{)}, for $\beta$ varying in
		$\Sigma^{+}_{\Theta}$, are the irreducible factors of the action of $\mL_\Theta$ on
		$\mk u_\Theta$ \textup{(}resp.\ $\mk u^{\mathrm{opp}}_\Theta$\textup{)}.
	\end{theorem}

	The Weyl group $\mW_{\mS_\Theta}$ of~$\mS_\Theta$ is generated by the~$s_\alpha$
	for~$\alpha$ in $\Delta\smallsetminus \Theta$, we then have:
	
	\begin{proposition}\label{pro:Weylgroup}
		The Weyl group $\mW_{\mS_\Theta}$ satisfies
		\[  \mW_{\mS_\Theta}\left(\Sigma_\theta\right)\subset\Sigma_\theta
		\]
		for every $\theta$ in $\Theta$.
	\end{proposition}

	Finally, for any $\theta$ in $\Theta$, we introduce the following
	subalgebra.  Let $\mathfrak{g}_{\theta}^{H} \defeq x_{-\theta}^{\perp} \cap
	\mathfrak{g}_\theta$, where we denote by $x_{-\theta}^{\perp}$ the orthogonal of~\(x_{-\theta}\) for the Killing
	form, this is a hyperplane in~$\mathfrak{g}_\theta$ ---hence the choice of the
	superscript~$H$--- not
	containing~$x_\theta$, and let
	\[
	\mk v_\theta\defeq \mathfrak{g}_{\theta}^{H} \oplus \sum_{\mathclap{\beta\in
			\Sigma_\Theta^+\smallsetminus\{\theta\}}}\mk g_\beta\ .
	\]
	Then $\mk u_\Theta=\mathbb{R} x_\theta\oplus \mk v_\theta$.  We set $\ms
	V_\theta=\exp(\mk v_\theta)$. We denote by $\mH_\theta$ the connected Lie group (isogenic to $\mathsf{SL}_2(\mathbb{R})$)
	whose Lie algebra is generated by the $\mathfrak{sl}_2$-triple $(x_\theta,
	x_{-\theta}, h_\theta)$.
	
	\begin{proposition}\label{pro:vtheta}
		For any $\theta$ in $\Theta$, the vector space $\mk v_\theta$ is an ideal in the Lie-algebra~$\mk u_\Theta$. In particular, we have
		\[
		\mU_\Theta=\exp(\mk g_\theta)\ltimes \ms V_\theta\ .
		\]
		The conjugates, by elements of~$\mH_\theta$, of the group~$\ms V_\theta$ are
		contained in~\(\mU_\Delta\) and hence in~$\mP_\Theta$.
	\end{proposition}
	\begin{proof}
		Since $\theta$ is a simple root,  $\Sigma_\Theta^+ \subset \Sigma^+$, and
		$[\mk g_\alpha, \mk g_\beta] \subset \mk g_{\alpha+\beta}$, we have $[\mk v_\theta, \mk v_\theta]\subset \mk v_\theta$, and $[x_\theta, \mk v_\theta] \subset \mk v_\theta$. 
		Thus, $\mk v_\theta$ is an ideal.
		
		Let $\dot{s}_\theta$ be an element of~$\mH_\theta$ representing the non-trivial
		element of the Weyl group of~$\mH_\theta$ (one can choose for example $\dot{s}_\theta = \exp(
		\frac{\pi}{2}( x_\theta-x_{-\theta}))$). The group~$\mH_\theta$ is generated
		by~$\dot{s}_\theta$ and $\exp(\langle x_\theta\rangle) \subset \exp(\mathfrak{g}_\theta)$. Since \(\Ad( \mathfrak{g}_\theta) \mathfrak{v}_\theta\subset \mathfrak{v}_\theta\), we have that $\mV_\theta$ is invariant
		by conjugation by~$\exp(\mathfrak{g}_\theta)$, and 
		it is hence sufficient to prove that the conjugate
		of~$\mV_\theta$ by~$\dot{s}_\theta$ is contained in~$\mU_\Delta$.
		
		By general properties of nilpotent Lie algebras, the decomposition \(\mk v_\theta= \mathfrak{g}_{\theta}^{H} \oplus \sum_{{\beta\in
			\Sigma_\Theta^+\smallsetminus\{\theta\}}}\mk g_\beta\) implies that \(\mathsf{V}_\theta = \exp( \mathfrak{g}_{\theta}^{H}) \exp( \sum_{{\beta\in
			\Sigma_\Theta^+\smallsetminus\{\theta\}}}\mk g_\beta) \).

		Let $\mV= \exp(
		\sum_{ \alpha\in \Sigma^+ \smallsetminus \{\theta\}}
		\mathfrak{g}_\alpha)$.  Since \(\Sigma_\Theta^+\smallsetminus\{\theta\} \subset \Sigma^+ \smallsetminus \{\theta\} \), one has the inclusion \(\exp( \sum_{{\beta\in
			\Sigma_\Theta^+\smallsetminus\{\theta\}}}\mk g_\beta) \subset \mV\) and hence $\mV_\theta \subset \exp( \mathfrak{g}_{\theta}^{H}) \mV $.   Let us
		prove that~$\mV$ and $\exp(\mathfrak{g}_{\theta}^{H})$ are invariant by
		conjugation by~$\dot{s}_\theta$. For~$\mV$, this 
		follows from the well known fact that $\dot{s}_\theta$~induces a permutation
		of~$\Sigma^+ \smallsetminus\{ \theta\}$; for
		$\exp(\mathfrak{g}_{\theta}^{H})$, this follows from the fact that
		$\mathfrak{g}_{\theta}^{H}$~is the trivial $\mH_\theta$-module.  
		
		The proposition now follows from the inclusions
		\[\dot{s}_\theta \mV_\theta \dot{s}_{\theta}^{-1} \subset \dot{s}_\theta \exp( \mathfrak{g}_{\theta}^{H})\mV \dot{s}_{\theta}^{-1} = \exp( \mathfrak{g}_{\theta}^{H})\mV \subset \mU_\Delta \subset \mP_\Theta. \qedhere\] 
	\end{proof}

	\subsection{Parabolic subgroups and flag manifolds}\label{sec:parab-subgr-flag}
	We consider the partial flag variety associated to
	$\mP_\Theta$. We choose a setup that allows us to identify the tangent spaces at points $z$ in the flag variety with the Lie algebra~$\mk u_\Theta^{\rm opp}$. 
	For this we consider a group $\mG$ isomorphic to $\mG_0$, and we let~${\cGG}$
	be the space of isomorphisms from~$\mG_0$ to~$\mG$. On $\cGG$
	the group~$\mG_0$ acts on the right by pre-conjugation and $\mG$~acts on the left by
	post-conjugation (in fact the groups of automorphisms of~$\mG$ and of~$\mG_0$
	act respectively on the left and on the right). We also fix once and for all $\cG$ a connected component
	of~${\cGG}$. As $\mG$ and $\mG_0$ are connected, $\cG$ is invariant by
	the actions of~$\mG$ and~$\mG_0$. For example, one could choose~$\mG$ to be
	equal to~$\mG_0$ and $\cG$ to be the connected component of the identity in the group of automorphisms of~$\mG_0$. In general, the space~$\cG$ identifies (not naturally) with
	the adjoint form of~$\mG_0$.
	
	The left and right actions of~$\mG$ and~$\mG_0$ on~$\mathcal{G}$ being locally
	free, we have, for every~$\phi$ in~$\mathcal{G}$, natural identifications
	$\T_\phi \mathcal{G}\simeq \mathfrak{g}$ and $\iota_{\phi}^{\cG}\colon \T_\phi \mathcal{G}\simeq
	\mathfrak{g}_0$; the composition of these identifications gives a map
	$\mathfrak{g}_0\to \mathfrak{g}$ which is precisely the differential $\phi_*
	\defeq\T_e \phi$ 
	of~$\phi$ at the identity.
	
	We consider the flag variety $\cF_\Theta\defeq\cG/\mP_\Theta$, which can be identified
	with a connected component in the set of 
	subalgebras~$\mk u$ of~$\mk g$ isomorphic to~$\mk u_\Theta$. The group~$\mG$ acts on the left on~$\cF_\Theta$.
	Let~$\pi^{\cF}$ be the projection from~$\cG$ to~$\cF_\Theta$. For every~$\phi$
	in~$\cG$ and $x=\pi^{\cF}(\phi)$, the differential $\T_\phi \pi^{\cF}$ is a map from $\T_\phi \mathcal{G}$
	to $\T_x \mathcal{F}_{\Theta}$; composing this map with the identification
	$\mathfrak{g}_0\simeq \T_\phi \mathcal{G}$ gives a projection
\[
\pi^{\cF}_{\phi}\colon \mk{g}_0\longrightarrow \T_x\cF_\Theta\ ,
\]
	whose kernel is $\mk p_\Theta$.
	This gives us, by restriction, an identification  of $\mk u^{\rm opp}_\Theta$
	with  $\T_x\cF_\Theta$ whose inverse will be denoted by $\iota_{\phi}^{\cF}\colon
	\T_x\cF_\Theta \to \mk u^{\rm opp}_{\Theta}$.
	
	Similarly, we introduce the \emph{opposite flag variety} $\cF_{\Theta}^{\mathrm{opp}}
	\defeq \cG/ \mP_{\Theta}^{\mathrm{opp}}$. As a $\mG$-space, it is isomorphic
	to $\cG/ \mP_{\iota(\Theta)}$ and the $\mG$-isomorphism is unique. The
	projection $\cG\to \cF_{\Theta}^{\mathrm{opp}}$ will be denoted
	by~$\pi^{\cF^\mathrm{opp}}$, and for every~$\phi$ in~$\cG$, letting
	$y=\pi^{\cF^\mathrm{opp}}(\phi)$, we have isomorphisms  $\iota_{\phi}^{\cF^\mathrm{opp}}\colon
	\T_y\cF^{\mathrm{opp}}_{\Theta} \to \mk u_{\Theta}$ and $\varpi_{\phi}^{\cF^\mathrm{opp}}\colon
	\mk u_{\Theta} \to \T_y\cF^{\mathrm{opp}}_{\Theta}$.
	
	The stabilizer in~$\mG$ of a point~$z$ in~$\cF_{\Theta}$ (or in~$\cF^{\mathrm{opp}}_{\Theta}$) will be
	denoted by~$\mP_z$, the unipotent radical of~$\mP_z$ will be denoted
	by~$\mU_z$. Their Lie algebras are denoted~$\mathfrak{p}_z$
	and~$\mathfrak{u}_z$ respectively.
	
	We say that a point~$z$  in~$\cF_\Theta$ and a point~$w$ in~$\cF^{\mathrm{opp}}_{\Theta}$ are {\em transverse}, and write $z\pitchfork w$, if 
	\[
	\mk p_z\oplus\mk u_w=\mk g\ .
	\]
	This is equivalent to the fact  $\mP_z\cap\mP_w$ is a Levi
	factor of both~$\mP_z$ and~$\mP_w$. Let finally
	\[
	\cL_\Theta \defeq\{(z,w)\in\cF_\Theta\times \cF^{\mathrm{opp}}_{\Theta}\mid z\pitchfork w\}\ ,
	\]
	and observe that $\cL_\Theta$ is canonically
	isomorphic to $\cG/\mL_\Theta$. The
	natural map $\cG \to \cL_\Theta$ will be denoted by~$\pi^{\cL}$, it is the
	corestriction of the map $(\pi^\cF, \pi^{\cF^\mathrm{opp}})$. The differential of~$\pi^\cL$ induces an onto
	morphism between the vector bundles $\T \cG$ and $\pi^{\cL*} \T \cL_\Theta$ that
	induces an isomorphism $\T \cG / \mathfrak{l}_\Theta \simeq \pi^{\cL*} \T \cL_\Theta$.
	The differential of~$\pi^{\cL}$ at a point $(z,w)= \pi^\cL(\phi)$, composed with the identification
	of~$\mathfrak{g}_0$ with $\T_\phi \mGG$ induces a map
	\[
	\pi^{\cL}_{\phi}\colon {\mk g_0} \longrightarrow{\T_{(z,w)}\cL_\Theta\ ,}
	\]
	which gives rise to an identification 	$\iota_{\phi}^{\cL}\colon \T_{(z,w)} \cL_\Theta = \T_z \cF_\Theta \oplus
	\T_w \cF_{\Theta}^{\mathrm{opp}} \to \mk u^{\rm opp}_\Theta \oplus \mk u_\Theta$, more precisely the
	map~$\iota_{\phi}^{\cL}$ is equal to $(\iota_{\phi}^{\cF}, \iota_{\phi}^{\cF^\mathrm{opp}})$
	where $\iota_{\phi}^{\cF}$ and~$\iota_{\phi}^{\cF^\mathrm{opp}}$ are given above.
	
	We denote by $\mL_{z,w}= \mP_z\cap\mP_w$ the stabilizer of a pair $(z,w)$ of transverse points
	in~$\cF_\Theta\times \cF_\Theta^\mathrm{opp}$.

	It will very often be the case in this paper that
	we are in the situation that the opposite
	parabolic~$\mP_{\Theta}^{\mathrm{opp}}$ is conjugate to~$\mP_\Theta$. In such
	case we will use the natural identification $\cF^{\mathrm{opp}}_{\Theta} \simeq \cF_{\Theta}$ and
	use the notion of transversality and the maps
	$\iota_{\phi}^{\cF^\mathrm{opp}}$ with elements of~$\cF_\Theta$ as well.

	\subsection{Loxodromic elements}
	\label{sec:loxodromic-elements}
	An element $g$ in $\mG$ is {\em $\Theta$-loxodromic} if and only if  $g$~has
	an attracting fixed point in~$\cF_\Theta$. In this case, $g$~has exactly one
	attracting fixed point~$z$ in~$\cF_\Theta$ and one repelling fixed point~$w$ in~$\cF_{\Theta}^{\mathrm{opp}}$, and
	those  fixed points~$z$ and~$w$ are transverse.

	An
	element $g_0$ 
	is {\em hyperbolic} if we can find an isomorphism~$\psi$ from~$\mG_0$ to~$\mG$ such that 
	$g_0=\psi(\exp(a))$ where $a$ is in the closed Weyl chamber.  Recall that $a$ is uniquely determined. 
	
	The {\em Kostant--Jordan decomposition} of~$g$ is the unique
	commuting product $g=g_hg_kg_u$ where $g_u$~is unipotent, $g_k$~generates a
	subgroup whose closure is compact, and $g_h$~is hyperbolic.  The unique element~$a$ of the closed Weyl chamber such that there exists~\(\psi\) in~\(\mathcal{G}\) with $\psi(\exp(a))=g_h$ is called the \emph{Jordan projection} of~$g$.

	We observe that we have the following: 
	\begin{proposition}\label{pro:loxo-hyper}
		An element~$g$ is $\Theta$-loxodromic if and only if its hyperbolic part~$g_h$
		is $\Theta$-loxodromic. In that case both~$g$ and~$g_h$ have the same
		repelling and attracting fixed points.
	\end{proposition}
	
	An algebraic definition is given by:
	
	\begin{proposition}\label{pro:loxo-alg}
		Let~$x$ in~$\cF_{\Theta}$ and~$y$ in~$\cF_{\Theta}^{\mathrm{opp}}$ be transverse points.
		\begin{enumerate}
			\item Let~$h$ be an hyperbolic element of~$\mG$, which is
			$\Theta$-loxodromic with attracting fixed point~$x$ in~$\cF_\Theta$ and
			repelling fixed point~$y$ in~$\cF_{\Theta}^{\mathrm{opp}}$. Let~$\psi$ be an isomorphism from~$\mG_0$ to~$\mG$ 
			such that
			$\psi(\mP_\Theta)=\mP_x$, $\psi(\mP^{\mathrm{opp}}_{\Theta})=\mP_y$,
			$h=\psi(\exp(a))$ with~$a$ in~$\mk a$. Then 
			we have \[\dual{a}{\theta}>0\ ,\] for
			all~$\theta$ in~$\Theta$.
			\item Conversely, assume that $\psi$~is an isomorphism from~$\mG_0$ to~$\mG$
			satisfying $\psi(\mP_\Theta)=\mP_x$,
			$\psi(\mP^{\mathrm{opp}}_{\Theta})=\mP_y$, let~$a$ be an element of the
			(closed) Weyl chamber~$\mk a^+$ such that for all $\theta$ in $\Theta$, we
			have $\dual{a}{\theta}>0$ then $\psi(\exp(a))$ is $\Theta$-loxodromic with
			attracting fixed point~$x$ and repelling fixed point~$y$.
		\end{enumerate}
	\end{proposition}
	
	\begin{proof}
		This follows from the fact that the tangent space to $\cF_\Theta$ at~$x$ identifies with
		$\mathfrak{u}^{\mathrm{opp}}_\Theta$ and the tangential action of~$h$ is given by $\exp(
		\Ad(a))$. We refer to \cite[Prop.~3.3]{ggkw_anosov}.
	\end{proof}
	\begin{remark}
		Of course the condition $\psi(\mP_\Theta)=\mP_x$ is equivalent to
		$\pi^\cF(\psi) = x$, and the conditions
		$\psi(\mP_\Theta)=\mP_x$, $\psi(\mP^{\mathrm{opp}}_{\Theta})=\mP_y$ are equivalent to
		$\pi^\cL(\psi) = (x,y)$.
	\end{remark}
	\subsection{Characters}
	\label{sec:characters}
	
	If a linear form~$\eta$ on~$\mathfrak{a}$ is given, the
	\emph{$\eta$-character} of~$g$ is the exponential of the evaluation of~$\eta$ on the
	Jordan projection~$a$ of $g$, it will be denoted by
	\[ \chi_\eta(g) = \exp( \dual{ a}{\eta})\ .\]

	\subsection{Dominant forms}
	\label{sec:line-forms-posit} We introduce  the notion of
	\emph{$\Theta$-compatible dominant forms and weights} in the dual of the
	Cartan subspace.
	
	\begin{definition}[\sc{$\Theta$-compatible dominant form}]
		\label{def:line-forms-posit}
		An element $\eta$ of $\mk
		a^*$ is called
		\begin{itemize}
			\item \emph{dominant} if $\braket{\eta,\theta}\geq 0$  for  all $\theta$ in $\Delta$;
			\item  {\em $\Theta$-compatible} if $\braket{\eta,\beta}=0$  for all~$\beta$ in $\Delta\smallsetminus\Theta$.
		\end{itemize}
	\end{definition}
	Equivalently, $\eta$ is a $\Theta$-compatible dominant form if and only if the restriction of~$\eta$ to~$\mk b_\Theta$ is zero and
	$\dual{ h_\theta}{ \eta}\geq 0 $ for all~$\theta$ in~$\Theta$, since 
	$\dual{ h_\theta}{ \eta} = 2\frac{ \braket{ \eta, \theta}}{ \braket{\theta,
			\theta}}$ (recall that $\mathfrak{b}_\Theta$ is the intersection
	$\bigcap_{\beta\in \Delta\smallsetminus \Theta} \ker \beta$). 
	When $\eta$~is \(\Theta\)-compatible and is
	the highest weight of a representation we will speak of a \emph{$\Theta$-compatible dominant weight}.  The sum of two $\Theta$-compatible dominant weights is a $\Theta$-compatible dominant weight.

	  Following \cite[Section~3.2]{GuichardWienhard_pos} we say that a  representation (not necessarily irreducible) is  proximal if it admits a highest weight and if the corresponding weight space is \(1\)-dimensional (cf.\ also \cite[Section~2.3]{benoist:2000ww}); note also that, when a representation admits a highest weight, then it is proximal if and only if the irreducible factor corresponding to this highest weight is proximal.    When a weight~\(\eta\) is the highest weight of a proximal representation we will say that \(\eta\)~is a proximal dominant weight.  The sum of two proximal dominant weights is a proximal dominant weight.

	Observe that if a non-zero dominant form is $\Theta$-compatible, there exists
	$\theta$ in $\Theta$ such that $\braket{\eta,\theta}>0$. Among those
	$\Theta$-compatible dominant forms are the fundamental weights
	$\{\omega_\theta\}_{\theta\in\Theta}$.  When \(\mathfrak{g}_\theta\) is \(1\)-dimensional, it is known that the fundamental weight~\(\omega_\theta\) is the highest weight of an irreducible proximal representation (see Appendix~\ref{app2_prox_rep}).  A general linear form~$\eta$ is
	$\Theta$-compatible and dominant if and only if it belongs to the convex cone
	generated by $\{\omega_\theta\}_{\theta\in\Theta}$ (since
	$\eta=\sum_{\theta\in\Delta} \dual{ h_\theta}{\eta}\omega_\theta$).

\section{Cross-ratios}
\label{sec:cross-ratio}

In this section we let \(\mathsf{G}_0\) and~\(\mathsf{G}\) be as in Section~\ref{sec:prelim} and~\(\Theta\) be any subset of~\(\Delta\), and we associate to every  $\Theta$-compatible proximal dominant weight~$\eta$
a cross-ratio $\bb^\eta$
defined on  the flag varieties $\cF_\Theta=\cG/\mP_\Theta$ and $\cF_{\Theta}^{\mathrm{opp}}=\cG/\mP_{\Theta}^{\mathrm{opp}}$.  This family of
cross-ratios satisfies the following multiplicative properties
\begin{equation}\label{e.multiplicativity}
	\bb^{\eta_1}\bb^{\eta_2}=\bb^{\eta_1+\eta_2}\ , \quad \bb^{n\eta}=\left(\bb^{\eta}\right)^n\ .
\end{equation}
Up to restricting the domain of definition of these cross-ratios, we give an
integral formula 
that allows us to give a construction of~$\bb^\eta$ for
every  $\Theta$-compatible dominant \emph{form}~$\eta$. This is possible due to a symplectic reinterpretation of such cross-ratio, which generalizes results of
\cite{Labourie:2005}.

\subsection{Cross-ratios and periods}\label{sec:def-cr}

Let $F$  and $F'$ be sets, which could be either a flag manifold and its opposite,  or the boundary at
infinity of a hyperbolic group (and itself). Let~$U$ be a subset of $F\times F'$. 
For flag manifolds, $U$~will be the set of pairs of transverse points; for the
boundary at infinity, $U$~will
be the set of pairs of distinct points. We  consider the subset~$O$ of $F^2\times
(F')^2$ defined by 
\[
O\defeq\bigl\{(x,y,X,Y)\in F^2\times (F')^2\mid (x,Y)\in U,\  (y,X)\in U\bigr\}\ .
\]
We recall from \cite{Labourie:2005}
that a {\em cross-ratio} is a non-constant function~$\bb$ defined on~$O$, taking values
in a field~$\mathbb{F}$, and satisfying the two cocycle identities
\begin{align*}
	\bb(x,w,X,Y)\ \bb(w,y, X,Y)&=\bb(x,y,X,Y)\ ,\\
	\bb(x,y,X,W)\ \bb(x,y,W,Y)&=\bb(x,y,X,Y)\ .
\end{align*}
Observe that, in these equations, if the left-hand side is defined, so is the
right and that the cocycle identities imply the equalities $\bb(x,x,X,Y)=1$ and
$\bb(x,y,X,X)=1$. Our cocycle identities are multiplicative and, when
$\mathbb{F}$ is~$\mathbb{R}$,
the cross-ratio may  also be negative. This is the definition from
\cite{Labourie:2005} (up to the ordering of the arguments). It differs from
Hamenstädt \cite{Hamenstadt:1997}
since we do not impose additional restrictive symmetries;
it also differs from Otal
\cite{Otal:1992} 
where the cocycle
identities are additive and symmetries are imposed as well.

\begin{example}
	For a field $\mathbb F$, the classical projective cross-ratio on
	${\mathbf P}^1(\mathbb F)$ is a cross-ratio in this sense. In projective
	coordinates it is defined by
	\[ [x,y,X,Y]=\frac{(X-x)(Y-y)}{(X-y)(Y-x)}\ .
	\]
	The convention for the order of elements in the projective cross-ratio is
	characterized by the fact that $[\infty,0,1,z]=z$. The projective
	cross-ratio is $\pgl_2(\mathbb{F})$-invariant.
\end{example}
\smallskip
We assume from now on that $F$ and~$F'$ are topological spaces and~$U$ is open.
We also assume that, for every~$x$ and~$y$ in~$F$, there is~$X$ in~$F'$ such
that $(x,X)$ and $(y,X)$ are in~$U$.

Let also~$\gamma$ act on~$F$ and on~$F'$   with exactly
one attracting fixed point $\gamma^+$ and one repelling fixed point
$\gamma^-$ on~$F$ and 
such that
its diagonal  actions 
preserve~$U$ and~$\bb$; then the period of~$\gamma$ with respect to $\bb$ is 
\[
\pp(\gamma)\defeq \bb(\gamma^+,\gamma^-,y,\gamma(y))\ ,
\]
for any $y$ in $F'$ such that $(\gamma^+, y)$ and $(\gamma^-,y)$ belong to~$U$ (this is
independent of the choice of~$y$ thanks to the cocycle identities).

\begin{example}
	When $\mathbb{F}$~is a valued field and if we take $\gamma$ acting on
	$\bP^1( \mathbb{F})$ defined by $\gamma(z)=\lambda\, z$ with the absolute
	value of~$\lambda$ being $>1$ so that $0$~is the repelling fixed point
	of~$\gamma$ and $\infty$~is the attracting fixed point, the period of
	$\gamma$ is given by
	\[\pp(\gamma)=\lambda\ .
	\]
\end{example}

\subsection{Projective cross-ratios}
\label{sec:proj-cross-rati}

The cross-ratio on the projective line generalizes to higher dimension. Let~$E$ be a vector space over a field~$\mathbb F$. We apply the general setting above to $F=\bP(E)$,
$F'=\bP(E^*)$ and $U$ the set of pairs of transverse elements, so that
\[
O =  \{(x,y,X,Y) \mid X\pitchfork y \hbox{ and  }  Y\pitchfork x\}.
\]
The \emph{projective cross-ratio} is the $\mathbb F$-valued function on the open
subset~$O$ given by 
\begin{equation}\label{e.projcr}
	\bb^E(x,y,X,Y)=\frac{\dual{\bar x}{\bar X}\dual{ \bar y}{\bar  Y}}{\dual{\bar y}{\bar X}\dual{ \bar x}{\bar Y}}\ ,
\end{equation}
where $\bar u$ denotes a non-zero vector in the line $u$. This does not depend on the choice of $\bar{x},\bar{y},\bar{X}$ and $\bar{Y}$.

When $E$ is finite dimensional and when $\mathbb{F}$~is a valued field,
$\bP(E)$ and $\bP(E^*)$ have natural topologies. In this case,
the period of an element~$g$ of $\pgl(E)$ for~$\bb^E$
is the ratio between the eigenvalue of greatest modulus and the eigenvalue of
lowest modulus of one (any) representative of~$g$ in
$\gl(E)$.

The projective cross-ratio behaves well under tensor product.

\begin{lemma}
	\label{lem:proj-cross-tensor}
	Let~$E_1$ and~$E_2$ be  vector spaces over a field~$\mathbb{F}$, and let
	$E=E_1\otimes E_2$. The natural maps $E_1\times E_2 \to E$ and
	$E_{1}^{*}\times E_{2}^{*} \to E^{*}$ induce maps  $\bP(E_1)\times \bP(E_2)
	\to \bP(E)$ and  $\bP(E_{1}^{*})\times \bP(E_{2}^{*}) \to \bP(E^*)$. All
	these maps will be denoted $(a,b)\mapsto a\otimes b$. One then has
	the identities
	\[
	\bb^E(x_1\otimes x_2, y_1\otimes y_2, X_1\otimes X_2, X_1\otimes X_2) =
	\bb^{E_1}(x_1, y_1, X_1, Y_1) \bb^{E_2}(x_2, y_2, X_2, Y_2)
	\]
	for all $(x_1,y_1,X_1,Y_1)$ and $(x_2,y_2,X_2,Y_2)$ such that
	$X_1\pitchfork y_1$, $Y_1\pitchfork x_1$, $X_2\pitchfork y_2$, and
	$Y_2\pitchfork x_2$.
\end{lemma}
\begin{proof}
	Using that $\dual{v_1\otimes v_2}{
		\phi_1 \otimes \phi_2} = \dual{v_1}{\phi_1}
              \dual{v_2}{\phi_2}$, this comes from a direct calculation.
\end{proof}

\subsection{The cross-ratio associated to a proximal dominant weight}\label{sec.cr}

We denote, as usual, by $\cF_\Theta =\cG/ \mP_\Theta$ the flag variety associated to
$\mP_{\Theta}$ and by $\cF_{\Theta}^\mathrm{opp} =\cG/ \mP_{\Theta}^\mathrm{opp}$ the flag
variety associated to $\mP_{\Theta}^{\rm opp}$. The space $\cL_\Theta =
\cG/\mL_\Theta$ is the unique open $\mG$-orbit in $\cF_\Theta \times \cF^{\mathrm{opp}}_{\Theta}$.

For a $\Theta$-compatible proximal dominant  weight~$\eta$, let
$\tau$ be the associated representation
$\mG\to \pgl(E)$
on a real vector space~$E$  (thus $\tau$~is the irreducible
proximal representation with highest weight~$\eta$, well-defined up to isomorphism).
This means that, for every~$\psi$ in~$\cG$, there is a (unique up to
scalar) vector~$v$ in~$E$ such that, for every~$X$ in~$\mathfrak{a}$,
$\tau_* \circ\psi_*(X) (v) = \dual{X}{\eta} v$; this vector~$v$ is mapped to 0 by
$\tau_* \circ\psi_*(\mathfrak{u}_\Theta)$ and is also an eigenvector of~$\tau_* \circ\psi_*(\mathfrak{l}_\Theta)$.
From this we deduce that
there are unique equivariant maps
$\Xi_\eta\colon \cF_\Theta\to \bP(E)$ and $\Xi_{\eta}^{*}\colon \cF^\mathrm{opp}_{\Theta} \to \bP(E^*)$. If we consider an element of~$\cF_\Theta$ as a nilpotent Lie algebra
in~$\mathfrak{g}$, its image under $\Xi_\eta$ is the unique eigenline in~$E$
for this subalgebra.

The \emph{cross-ratio associated to $\eta$} is (with $F= \cF_\Theta$,
$F'=\cF_{\Theta}^{\mathrm{opp}}$, and $U=\cL_\Theta$)
\begin{equation}\label{e.crgen} \bb^\eta(u,v,w,z)\defeq\bb^E\bigl(\Xi_\eta(u),\Xi_\eta(v),\Xi^{*}_{\eta}(w),\Xi^{*}_{\eta}(z)\bigr)\ , 
\end{equation}
for all $(u,v,w,z) \in O\subset \cF_\Theta\times\cF_\Theta\times\cF^\mathrm{opp}_{\Theta}\times \cF^\mathrm{opp}_{\Theta}$
(i.e.\ such that $(v,w)$ and $(u,z)$ belong to~$\cL_\Theta$), where $\bb^E$ is the projective cross-ratio defined in Equation \eqref{e.projcr} (with \(\mathbb{F}=\mathbb{R}\)).  It
follows directly from the definition that the cross-ratio associated to $\eta$
is a semi-algebraic function.

Assume  now that $h$ in $\mG$ is such that both $h$ and $h^{-1}$  are
$\Theta$-loxodromic elements (equivalently, $h$~is $(\Theta \cup
\iota(\Theta))$-loxodromic) and denote their attracting fixed points in $\cF_\Theta$
respectively~$h^+$ and~$h^-$. Hence, $h^-$ is the repelling fixed point of~$h$. The {\em $\eta$-period} of~$h$ is 
\[
\pp^\eta(h)=\bb^\eta(h^+,h^-,x,h(x))\ ,
\]
for (any) $x$ in $\cF^\mathrm{opp}_{\Theta}$ transverse to both $h^-$ and $h^+$ (i.e.\ the $\eta$-period is
the period with respect to the cross-ratio~$\bb^\eta$).

The periods are related to the $\eta$-characters (Section~\ref{sec:characters}).
\begin{proposition}
	\label{prop:periods-loxo}
	For any element $h$ such that both $h$ and $h^{-1}$ are $\Theta$-loxodromic, we have
	\[
	|\pp^\eta(h)|=\chi_\eta(h)\cdot \chi_\eta(h^{-1})\ .
	\]
\end{proposition}
Note that $\chi_\eta(h^{-1})= \chi_{\iota(\eta)}(h)$ where $\iota$~is the
opposition involution.

\begin{proof}
	Let~\(\tau\colon \mathsf{G}\to \PGL(E)\), \(\Xi_\eta\), and \(\Xi_{\eta}^{*}\) be the maps used in the definition of the cross-ratio~\(\bb^\eta\) (Equation~\eqref{e.crgen}).  Let \(A\) be an element of \(\GL(E)\) representing \(\tau(h)\) scaled so that \(\det A=\pm 1\).  The action of~\(h\) on~\(E^*\) is given by \({}^T A^{-1}\).

	Let~$h^+$ (resp.\ \(h^-\)) be the attracting fixed point for
	the action of~$h$ (resp.\ \(h^{-1}\))  on~$\cF_\Theta$.
	Let \(e_1\) be a non-zero vector in \(\Xi_\eta(h^+)\) and let \(e_2\) be a non-zero vector in \(\Xi_\eta(h^-)\).  Then \(e_1\) and~\(e_2\) are eigenvectors for the action of~\(A\) on~\(E\).  Let~\(\lambda\) and~\(\mu\) be the corresponding eigenvalues.  One then has \(\chi_\eta(h)= |\lambda|\) and \(\chi_\eta(h^{-1})= |\mu^{-1}|\).

	Furthermore, the \(\eta\)-period of~\(h\) is equal to
	\[\pp^\eta(h) = \bb^E( [e_1], [e_2], [f], [{}^T A^{-1} f])\]
	for any linear form~\(f\) that is not zero on~\(e_1\) and on~\(e_2\).  A calculation gives then \(\pp^\eta(h) = \lambda \mu^{-1}\) and implies the sought for result. 
\end{proof}
By construction the cross-ratio is multiplicative in $\eta$:
\begin{lemma}
	\label{lem:cross-ratios-mult}
	For every $\Theta$-compatible proximal dominant
	weights $\eta_1,\eta_2$ and any integer $n$ it holds 
	$$\bb^{\eta_1}\bb^{\eta_2}=\bb^{\eta_1+\eta_2}\ , \quad \bb^{n\eta}=\left(\bb^{\eta}\right){}^n\ .$$	
\end{lemma}	
\begin{proof}
	This follows readily since, if $\tau_i\colon \mG\to{\sf PGL}(E_i)$ are the
	representations associated to the weights $\eta_i$ for $i=1,2$, then the
	representation $\tau$ associated to $\eta=\eta_1+\eta_2$ can be realized as
	an irreducible factor of $E_1\otimes E_2$ and the associated equivariant
	maps $\Xi_\eta$, $\Xi_\eta^*$ are given by
	$\Xi_{\eta_1}\otimes\Xi_{\eta_2}$, respectively
	$\Xi_{\eta_1}^*\otimes\Xi_{\eta_2}^*$. The claim then follows from
	Lemma~\ref{lem:proj-cross-tensor}. The second claim follows from the first
	by induction on~$n$.
\end{proof}		

\subsection{A symplectic reinterpretation} 
\label{sec:sympl-reint}

We now give a symplectic reinterpretation of the cross-ratio $\bb^\eta$ in analogy with   \cite[Section~4.4]{Labourie:2005}. We first consider the case of the projective cross-ratio. 
The product $E\times E^*$ of the real vector space $E$ and its dual carries a
canonical symplectic form; this is the natural symplectic form~$\omega^E$ on the
cotangent bundle $T^*E = E\times E^*$, one has also $\omega^E = -d\beta^E$
where $\beta^E$~is the canonical $1$-form (or Liouville $1$-form); explicitly, for $(v,\phi)$ in $E\times E^*$ and $(\dot{v}, \dot{\phi})$ in $
\T_{(v,\phi)} (E\times E^*) \simeq  E\times E^*$, one has $\beta^{E}_{(v,\phi)}(\dot{v},
\dot{\phi}) = \dual{\dot{v}}{\phi}$. The real multiplicative group $\mathbb
R^*$ acts symplectically on $E\times E^*$ by $\lambda(x,X)=(\lambda
x,\lambda^{-1} X)$, with a moment map given by $\mu(x,X)=\dual{x}{X}$. The symplectic reduction at~$1$ ---that is the space  $\mu^{-1}(1)/\mathbb R^*$--- 
then identifies with 
\[U=\{ (x,X) \in \bP(E)\times\bP(E^*) \mid \dual{\bar x}{\bar X}\neq 0 \}\ ,\]
which hence carries a symplectic form that we call~$\omega$. More explicitly \cite[Section~4.4.3]{Labourie:2005}, if we identify the
tangent space to $\bP(E)\times\bP(E^*)$  at a pair $(L,P)$ ---where $L$ is a
line transverse to the hyperplane~$P$---  with $(L^*\otimes P)\oplus(P^*\otimes L)$ we have
\begin{eqnarray}
	\omega\left((f,g),(h,j)\right)=\operatorname{Trace} (f\circ j)-\operatorname{Trace}(h\circ g)\ .\label{eq:explicit-omega}	
\end{eqnarray}
The following is proved in \cite[Proposition 4.7]{Labourie:2005}.
\begin{proposition}\label{pro:proj-cr-om}
	Let $f$ be a continuous, piecewise $\mathcal{C}^1$ map from $[0,1]^2$ to $U$ such that, for every~$t$ in~$[0,1]$,
	$f(0,t) = (x, *)$ and $f(1,t)=(y,*)$, and for every~$s$ in~$[0,1]$,
	$f(s,0)=(*,X)$ and $f(s,1)=(*,Y)$.
	Then
	\[
	\bb^E(x,y,X,Y)=\exp \Bigl({\int_{[0,1]^2}f^*\omega}\Bigr)\ .
	\]
\end{proposition}

A map~$f$ satisfying the conditions in
Proposition~\ref{pro:proj-cr-om} can be constructed if and only if there are vectors $\bar{x}, \bar{y}, \bar{X}, \bar{Y}$ representing
$x,y,X,Y$ 
such that the pairings
$\dual{\bar{x}}{\bar{X}}, \dual{\bar{y}}{\bar{X}}, \dual{\bar{x}}{\bar{Y}},
\dual{\bar{y}}{\bar{Y}}$ are all positive.

The goal of the section is to obtain a version of
Proposition~\ref{pro:proj-cr-om} that calculates the cross-ratios~$\bb^\eta$
defined on other flag varieties. We
will introduce a form $\Omega$ on
$\mathcal{L}_\Theta=\cG/\mL_\Theta$ with values in $\mk b_\Theta$ and show:

\begin{proposition}\label{pro:cr-om} Let $\eta$ be a $\Theta$-compatible proximal dominant weight.
	Let $f$ be a continuous, piecewise $\mathcal{C}^1$ map from $[0,1]^2$ to $\cL_\Theta$ such that, for every~$t$ in~$[0,1]$,
	$f(0,t) = (x, *)$ and $f(1,t)=(y,*)$, and for every~$s$ in~$[0,1]$,
	$f(s,0)=(*,X)$ and $f(s,1)=(*,Y)$. Then
	\[
	\bb^\eta(x,y,X,Y)=\exp\Bigl({\int_{[0,1]^2}f^*\left(\dual{\Omega}{\eta}\right)}\Bigr)\ .
	\]
\end{proposition}

\begin{remark}
	If, as above,
	$\tau\colon \mG\to \pgl(E)$ is a representation on a
	real vector space~$E$ with highest weight~$\eta$, $\dual{ \Omega}{\eta}$ is
	the curvature form associated to the action of~$\mL_\Theta$ on the vector
	space~$E$.
\end{remark}

Proposition~\ref{pro:cr-om} enables 
us to extend the definition of $\bb^\eta$ for
every $\Theta$-compatible dominant form~$\eta$. For a general~$\eta$, the
cross-ratio $\bb^\eta$ is defined on the subspace~$O^\square$ of~$O$ consisting of all  
quadruples $(x,y,X,Y)$ bounding a continuous, piecewise $\mathcal C^1$ square
as in the assumptions of Proposition~\ref{pro:cr-om} and $\bb^\eta(x,y,X,Y)$ is
defined by the integral formula in that proposition.  When convenient we will denote this subspace \(O^{\square}_{\Theta}\) or (in the case when \(\Theta\)~is invariant under the opposition involution) \(O^{\square}(\mathcal{F}_\Theta)\) or else \(O^{\square}(\mathbf{P}^1(\mathbb{R}))\), etc.

Note that the conditions $(x,y,X,Y)$ in $O^\square$ and $(y,z,X,Y)$ in $ O^\square$
imply $(x,z,X,Y)$ in $ O^\square$ (respectively $(x,y,X,Y)$ in $ O^\square$ and $(x,y,Y,Z)$ in $ O^\square$
imply $(x,y,X,Z)$ in $ O^\square$) and that the cocycle identities also hold.
This family of cross-ratios is
multiplicative in~$\eta$ (i.e.\ $\bb^{t_1 \eta_1+t_2\eta_2} = (\bb^{\eta_1})^{t_1}
(\bb^{\eta_2})^{t_2}$) so that it is completely determined by the
cross-ratios $\{ \bb^{\omega_\theta}\}$ associated to the fundamental
weights (for \(\theta\) varying in~\(\Theta\)).

\subsubsection{The curvature form}\label{app:curv} 

Recall that, for every~$\phi$ in~$\cG$, we have isomorphisms
$\iota_{\phi}^{\cG}\colon \T_\phi \cG \to \mathfrak{g}_0$,
$\iota_{\phi}^{\cL}\colon \T_w \cL_\Theta \to \mathfrak{u}_{\Theta}^{\mathrm{opp}}
\oplus \mathfrak{u}_\Theta$ (where $w=\pi^\cL(\phi)$).

We introduce the following $\mathfrak{b}_\Theta$-valued forms
on~$\cG$: for~$\phi$ in~$\cG$ and~$v$ in~$\T_\phi \cG$
\[ \beta^{\cG}_{\phi}(v) = p\bigl( \iota_{\phi}^{\cG}(v)\bigr), \quad
\Omega^{\cG} = -d\beta^{\cG}\ ,\]
where $p\colon \mathfrak{g}_0\to \mathfrak{b}_\Theta$ is the orthogonal projection.

One has $\Omega^{\cG}_{\phi}(v,w) = p\bigl( [\iota_{\phi}^{\cG}(v),
\iota_{\phi}^{\cG}(w)]\bigr)$ for~$\phi$ in~$\cG$ and~$v$, $w$ in~$\T_\phi \cG$. 

The form~$\beta^\cG$ is a section of the vector bundle $(\T \cG)^* \otimes \mathfrak{b}_\Theta$.
As $\beta^{\cG}$~is equivariant and as the action of~$\mL_\Theta$ on~$\mathfrak{b}_\Theta$ is trivial (Proposition~\ref{pro:ortho-atheta}), the
form~$\beta^{\cG}$, seen as a section of $(\T \cG)^* \otimes \mathfrak{b}_\Theta$,  descends to a section of the vector bundle $(\T \cG/
\mathfrak{l}_\Theta)^* \otimes \mathfrak{b}_\Theta \simeq (\pi^{\cL*}\T\cL_\Theta)^*
\otimes \mathfrak{b}_\Theta$ over~$\cG$. This section is also equivariant
and, again by triviality of the action of~$\mL_\Theta$ on~$\mathfrak{b}_\Theta$, it descends to a
section of the vector bundle $(\T\cL_\Theta)^*
\otimes \mathfrak{b}_\Theta$, that is, to a
$\mathfrak{b}_\Theta$-valued $1$-form~$\beta$ on~$\cL_\Theta = \cG/L_\Theta$. One has
$\pi^{\cL*} \beta = \beta^{\cG}$.
We then define $\Omega = -d\beta$.

From the construction, one has:
\begin{proposition}
	The forms $\Omega^{\cG}$ and $\Omega$ are closed, one has 	$\Omega^{\cG}=\pi^{\cL*}\Omega$.
\end{proposition}
We call $\Omega$ the {\em curvature form} on $\cG/\mL_{\Theta}$. Some readers
will recognize a curvature form of some bundle as in \cite{Benoist:1992}.

\smallskip

We  describe now a special case for the group~$\slm$. Using the standard
numbering for the simple roots, let~$\Theta_0$ consist of the first simple
root  so that $\cG/\mP_{\Theta_0}=\bP^{m-1}(\mathbb R)$. The quotient
$U=\cG/\mL_{\Theta_0}$ is the space of pairs of a line and a
hyperplane  transverse to each other. In this case $\mathfrak{b}_{\Theta_0} \simeq \mathbb{R}$.

\begin{proposition}[{\cite[Proposition 4.7]{Labourie:2005}}]
	\label{prop:curvature-form-proj-case}
	In the case that  $\cG/\mP_{\Theta}=\bP^{m-1}(\mathbb R)$, the curvature form on~$U$ coincides with the symplectic form as in Equation \eqref{eq:explicit-omega}.
\end{proposition}

\subsubsection{Linear representations and cross-ratios}

Let $\eta$ be a $\Theta$-compatible proximal dominant weight and let $\tau\colon \mG
\to \pgl(E)$
be the associated irreducible representation. The equivariant
maps $\Xi_\eta\colon \cF_\Theta \to \bP(E)$ and $\Xi_{\eta}^{*}\colon
\cF^{\mathrm{opp}}_{\Theta}\to \bP(E^*)$  combine to an equivariant map
$\Psi_\eta\colon \cL_\Theta \to U$, where $U$~is the space of pairs of transverse
points in $\bP(E)\times \bP(E^*)$.

\begin{proposition}\label{prop:Omega-ind}
	If $\omega$ is the symplectic form on the open
	subset~$U$ of  $\bP(E)\times\bP(E^*)$ (cf.\ Proposition~\ref{prop:curvature-form-proj-case}), then $\Psi^{*}_{\eta}(\omega)=\dual{\Omega}{\eta}$, where $\Omega$~is the curvature form on~$\cG/\mL_\Theta$.
\end{proposition}

\begin{proof}
	Assume for simplicity that $\tau\colon \mG
	\to \pgl(E)$ lifts to $\gl(E)$, the general case follows through covering theory, by observing that all our computations are local. 
	
	Let $\beta^E$ be the Liouville form on $E\times E^*$ and let~$\mu$ be the
	moment map for the $\mathbb{R}^*$-action.
	
	Let us  fix an equivariant 
	map $\zeta\colon \cG \to E\times E^*$ lifting the
	map~$\Psi_\eta\circ \pi^\cL$ and with values in $\mu^{-1}(1)$, i.e.\ in the
	space of pairs $(v,\ell)$ such that $\dual{v}{\ell}=1$. It is then enough to check that $\zeta^* \beta^E =
	\dual{\beta^\cG}{\eta}$.
	
	Let $\phi$ be in~$\cG$. By construction $\zeta(\phi)=(v,\ell)$ where $v$~is an
	eigenvector with weight~$\eta$ for the action of the Cartan
	subspace~$\mathfrak{a}$ under $\tau_*\circ \phi_*$, and~$\ell$ in~$E^*$ is an
	eigenvector with weight~$\iota(\eta)$, in other words the kernel of~$\ell$
	is the sum of the weight spaces associated to the weights different from~$\eta$. This means that
	\begin{align*}
		\beta^{E}_{(v, \ell)}(\zeta_*\circ \phi_*(X) ) &= \dual{ \zeta_*\circ
			\phi_*(X)}{\ell}\\
		&= \dual{ \tau_*\circ
			\phi_*(X) ( v)}{\ell}\\
		&= \bigdual{ \dual{X}{\eta} v}{\ell}\\
		&= \dual{X}{\eta}
	\end{align*}
	for every~$X$
	in~$\mathfrak{a}$.
	
	One thus has $\bigl(\zeta^*  \beta^E\bigr)_\phi(\phi_*(X)) = \dual{X}{\eta}$ for
	every~$X$ in~$\mathfrak{a}$. Since $\eta$~is $\Theta$-compatible, we also
	get $\bigl(\zeta^*  \beta^E\bigr)_\phi(\phi_*(X)) = \dual{p(X)}{\eta}$  for
	every~$X$ in~$\mathfrak{a}$, indeed $X-p(X)$ belongs
	to~$\mathfrak{a}_\Theta$ and $\eta$~is zero in restriction
	to~$\mathfrak{a}_\Theta$. This last equality also holds
	\begin{itemize}
		\item for every~$X$ in~$\mathfrak{g}_\alpha$ since $\zeta_*\circ \phi_*(X)$ sends~$v$ to the
		kernel of~$\ell$ and $p(X)=0$;
		\item  for every~$X$ in~$\mathfrak{z}_\mathfrak{k}(\mathfrak{a})$ since
		$v$~is also an eigenvector for this compact Lie subalgebra and is thus
		cancelled by~$X$ and again $p(X)=0$;
		\item hence for every~$X$ in~$\mathfrak{g}_0$ since we have the
		decomposition $\mathfrak{g}_0 = \mathfrak{a} \oplus
		\mathfrak{z}_\mathfrak{k}(\mathfrak{a}) \oplus \bigoplus_{\alpha\in\Sigma}
		\mathfrak{g}_\alpha$.
	\end{itemize}
	Another way to formulate this equality is $\bigl(\zeta^*
	\beta^E\bigr)_\phi(v) = \dual{p( \iota_{\phi}^{\cG}(v))}{\eta}$. In view of
	the definition of~$\beta^\cG$ this proves the proposition.
\end{proof}

\begin{proof}[Proof of Proposition~\ref{pro:cr-om}]
	The result then follows from Proposition~\ref{prop:Omega-ind} and the projective
	case proven in \cite[Proposition 4.7]{Labourie:2005}, recalled here in Proposition~\ref{pro:proj-cr-om}.
\end{proof}

\section{Photons}\label{sec:photo-construct} 

In this section
we use the subgroup~$\mH_\theta$ of~$\mG_0$ which is locally isomorphic to
$\mathsf{PSL}_2(\mathbb{R})$ to associate to each root~$\theta$ in~$\Theta$,
such that $\mk g_\theta$~is one dimensional, a class of curves in~$\cF_{\Theta}$ that we call \emph{$\theta$-photons},
such that
\begin{enumerate}
	\item each $\theta$-photon $\Phi$ is an orbit for the action of a
	subgroup which is the image~$\mH_\Phi$ of~$\mH_\theta$ by an element
	$\psi\colon \mG_0 \to \mG$ of~$\cG$,
	\item Given any $\theta$-photon $\Phi$, there is a {\em photon
		projection}~$p_\Phi$ from an open set in $\cF_{\Theta}^{\mathrm{opp}}$ to $\Phi$, and $p_\Phi$~is
	equivariant with respect to the action of $\mH_\Phi$.
	\item This projection has some nice properties with respect to the
	cross-ratio~$\bb^\eta$ associated to a $\Theta$-compatible dominant
	form. In particular, it satisfies that if $(x,y,u)$ is a triple of pairwise distinct points 
	in~$\Phi$ and~$z$ and~$w$ are such that $p_\Phi(z)=p_\Phi(w)=u$, then
	\[ \bb^\eta(x,y,z,w)=1\ .
	\]
\end{enumerate}
A key step in the proof of the Collar Lemma is then to relate the
cross-ratio~$\bb^\eta$ to the projective cross-ratio on the photon, this is
done in Proposition~\ref{pro:phot-cr}.

\begin{remark}\label{r.Einstein}
	To motivate the terminology \emph{photon}, recall that  the Einstein
	universe is a flag manifold for the group $\mathsf{SO}(2,n)$. It admits a conformal structure of signature $(1,n-1)$, for which
	lightlike geodesics called photons play an important role. In this case a $\theta$-photon is precisely a lightlike geodesic or a photon in the classical sense.\end{remark}

\subsection{Photon subgroups and photons}
\label{sec:phot-subgr-phot}

Let us consider, for~$\theta$ in~$\Theta$ such $\dim \mk g_\theta=1$,  the connected subgroup
$\mH_\theta$ in~$\mG_0$ whose Lie algebra is
generated by the $\skd$-triple $(x_\theta,x_{-\theta},h_\theta)$ (cf.\ Section~\ref{sec:skd-triples}). Observe that $\dim(\mH_\theta\cap\mP_\Theta)=2$.

Given an element $\psi$ of $\cG$, we then consider the group
$\psi(\mH_\theta)$. Recall that~$\pi^{\cF}$ denotes the projection from~$\cG$
to~$\cF_\Theta$ (Section~\ref{sec:parab-subgr-flag}). We introduce the following definition.

\begin{definition}
	A \emph{$\theta$-photon through $x$} is a subset $\Phi=\psi(\mH_\theta)\cdot x$ of~$\cF_\Theta$, for some~$\psi$ such that $\pi^{\cF}(\psi)=x$.	
\end{definition}

Note that in the situation of the definition, $\mP_x \cap \psi( \mH_\theta) =
\psi( \mP_\Theta \cap \mH_\theta)$ is a parabolic subgroup of~$\psi(\mH_\theta)$.

\begin{remark}
	In the case of the Grassmanian $\mathrm{Gr}_p( \mathbb{R}^n)$, photons
	appeared in the work of Van Limbeek and Zimmer \cite{LimbeekZimmer} under
	the name ``rank one lines''. Photons have also been introduced by Galiay
	\cite{Galiay} in the Shilov boundary of a tube type Hermitian Lie
	group, and in \cite[Section~6.3]{galiay_phd} in flag
        varieties that are Nagano
        spaces of ``real type''; she introduced as well generalizations of
        photons when the root space~$\mk g_\theta$ is not necessarily one
        dimensional (see \cite[Section~6.5]{galiay_phd}).\end{remark}

The family of $\theta$-photons through~$x$ will be denoted
by~$\mathbf{\Phi}(x)$.
We have:
\begin{proposition}\label{prop:photons-through-x}
  Let~$x$ be a point in~$\cF_\Theta$. Then:
  \begin{enumerate}
  \item\label{item:1:prop:photons-through-x} The unipotent radical $\mU_x$ of
    $\mP_x$ acts trivially on the family~$\mathbf{\Phi}(x)$.
  \item\label{item:3:prop:photons-through-x} The center of $\mL_x$ acts
    trivially on the family~$\mathbf{\Phi}(x)$.
  \item\label{item:2:prop:photons-through-x} The group $\mL_x= \mP_x/\mU_x$
    acts transitively on the family~$\mathbf{\Phi}(x)$ (equivalently
    $\mL_\theta$ acts transitively).
  \end{enumerate}
\end{proposition}

\begin{proof}
  We first prove (\ref{item:1:prop:photons-through-x}).  Let~$u$ be in~$\mU_x$
  and let~$\Phi$ be a photon through~$x$. Choose $\psi\colon \mG_0 \to \mG$
  such that $\pi^{\cF}(\psi)=x$ and $\Phi= \psi(\mH_\theta)\cdot x$. In
  particular $\mU_x=\psi(\mU_\Theta)$. The element $\psi^{-1}(u)$ can be
  written as the product $sv$ with $s$ in $ \exp(\mathbb{R} x_\theta)$ and $v$
  in $ {\ms V}_\theta$ (Proposition~\ref{pro:vtheta}). This implies that for
  every $y=\psi({s'})\cdot x$ in~$\Phi$ (where $s'$ in $ \mH_\theta$), one
  has
  $$u\cdot y= \psi(svs')\cdot x= \psi(ss' s^{\prime -1}vs')\cdot x=
  \psi(ss')\cdot x,$$ since $s^{\prime -1}vs'$ belongs to~$\mP_\Theta$
  (Proposition~\ref{pro:vtheta}). Hence, $u\cdot \Phi=\Phi$ and the first item
  is proved.  Similar considerations show the second item.

	Let~$\Phi_1$ and~$\Phi_2$ be photons through~$x$. Let $\psi_i$ ($i=1,2$)
	in~$\cG$ be such that $\pi^{\cF}(\psi_i)=x$ and $\Phi_i = \psi_i(\mH_\theta)\cdot
	x$. Since $\pi^{\cF}(\psi_2)=\pi^{\cF}(\psi_1)=x$, there is~$p$ in~$\mP_x$ such that
	$\psi_2 = \operatorname{int}_p\circ \psi_1$ (where
	$\operatorname{int}_p\colon \mG\to \mG \mid g\mapsto pgp^{-1}$). One then
	has $\Phi_2 = p\cdot \Phi_1$. This implies the transitivity in the third
	item.
\end{proof}

We have:
\begin{proposition}\label{pro:PhoRP}
	A $\theta$-photon is diffeomorphic to $\bP^1(\mathbb R)$.
	More precisely, let~\(\psi\) be in~\(\mathcal{G}\), \(x=\pi^{\mathcal{F}}(\psi)\), \(\mathsf{H}_\Phi = \psi(\mathsf{H}_\theta)\), and \(\Phi= \mathsf{H}_\Phi\cdot x \); then the action on~$\Phi$ of the (connected) group~$\mH_\Phi$
	factors through the adjoint group associated to~$\mH_\Phi$. This adjoint
	group is isomorphic to $\mathsf{PSL}_2(\mathbb{R})$, and $\Phi$ is
	equivariantly diffeomorphic to $\bP^1(\mathbb R)$.
\end{proposition}

\begin{proof}
	The flag variety~$\cF_\Theta$ can be identified with a $\mG$-orbit in the space of Lie subalgebras
	in~$\mk g$ isomorphic to~$\mathfrak{u}_\Theta$.   In view of the next lemma, which is of
	independent interest, applied to $\mH=\mH_{\theta}$, its action on
	$V=\mathfrak{g}_0$, and $W=\mathfrak{u}_\Theta$, it is thus enough to note
	that the stabilizer of~$\mathfrak{u}_\Theta$ in~$\mathfrak{h}_{\theta}$ is a
	Borel subalgebra.
\end{proof} 

\begin{lemma}
	Let $\mH$ be a connected Lie group such that $\mk{h}\simeq \skd(\mathbb{R})$ and $\mB$
	the subgroup of~$\mH$ which is the normalizer of the Borel
	subalgebra~$\mk b$ consisting of upper triangular matrices.
	
	Let $V$ be a finite dimensional real vector space, and
	$\tau\colon \mH \to \gl(V)$ be a continuous morphism with tangent Lie
	algebra morphism $\tau_*\colon \mk h\to \operatorname{End}(V)$.  Assume
	that $W$~is a linear subspace of~$V$ whose stabilizer in $\mk{h}$ ---via
	the morphism~$\tau_*$---
	is equal to~$\mk{b}$. Then
	the stabilizer of~$W$ in~$\mH$ is equal to~$\mB$. Therefore, the action
	of~$\mH$
	on the orbit~$\Psi$ of~$W$ in the corresponding Grassmannian factors through
	the
	adjoint group, this adjoint group is isomorphic to $\mathsf{PSL}_2(\mathbb{R})$, and $\Psi$~is
	equivariantly diffeomorphic to $\bP^1(\mathbb R)$. 
\end{lemma}
\begin{proof} 
	Recall that for every positive integer~$n$ there is a unique (up to
	isomorphism)  irreducible
	$\mathfrak{sl}_2(\mathbb{R})$-module of
	dimension~$n$, that 
	this module
	integrates into a continuous
	homomorphism defined on~$\sld$, and that every
	$\mathfrak{sl}_2(\mathbb{R})$-module decomposes as a sum of irreducible
	modules. 
	
	This implies that the representation $\tau_*$ integrates into a continuous
	homomorphism $\hat{\tau}\colon \sld\to \gl(V)$. The images in~$\gl(V)$ of the
	homomorphisms~$\tau$ and~$\hat{\tau}$ then coincide. Thus, it is enough to prove
	the conclusion for $\mH=\sld$.

	Let $d$ be the dimension of $W$. Let 
	$$
	E={\textstyle\bigwedge^d} V\ ,
	$$
	let~$\rho$ be the associated representation of $\gl(V)$ on $E$, and let~$q$ be the
	$\gl(V)$-equivariant (injective) map from the Grassmannian of $d$-planes in $V$ to $\mathbf P(E)$. Then an element~$g$ of $\gl(V)$ stabilizes the subspace~$W$ if and only if of $\rho(g)$
	stabilizes~$q(W)$.
	
	Hence, we can assume that $d=1$. In
	this case, we know from the representation theory of $\sld$, that the $\mathfrak{sl}_2(\mathbb{R})$-module generated by
	the $\mathfrak{b}$-invariant line~$W$ is irreducible; this means that we can
	assume that~$V$ is irreducible. However, for irreducible modules 
	the conclusion is well known (cf.\ \cite[Section~11.1]{Fulton-Harris}). 
\end{proof}

As a direct corollary of Proposition~\ref{pro:PhoRP} we have:
\begin{corollary}
	\label{coro:unipotent-orbit-dense-in-photon}
	Let~$\Phi$ be a photon and let~$x$ be in~$\Phi$. For any~$\psi$ in~$\cG$
	such that $\pi^\cF(\psi)=x$ and $\Phi=\psi(\mH_\theta)\cdot x$, one has
	\[ \Phi = \overline{ \psi( \exp( \braket{ x_{-\theta}}))\cdot x}\ .\]
\end{corollary}

\begin{proposition}\label{pro:uniq-phot}
	Assume that two photons are tangent at a point $x$. Then they coincide.
\end{proposition}
\begin{proof} Given any $\psi$ such that $\pi^{\cF}(\psi)=x$, recall the projection (Section~\ref{sec:parab-subgr-flag}) 
	$$\pi^{\cF}_{\psi}\colon \mk g_0\to\T_x\cF_\Theta\ ,$$
	and observe that the restriction of~$\pi^{\cF}_{\psi}$ to any vector subspace intersecting $\mk p_\Theta$ trivially is injective.
	
	Let $\psi$ and $\phi$ be in ${(\pi^{\cF})^{-1}}(x)$ such that the photons
	$\Phi_\psi = \psi(\mH_\theta)\cdot
	x$ and $\Phi_\phi = \phi(\mH_\theta)\cdot x$ are tangent at~$x$.
	
	By item~(\ref{item:2:prop:photons-through-x}) of
	Proposition~\ref{prop:photons-through-x}, we can assume that there is~$g$ in $\mL_\Theta$ such that
	$\psi=\phi\circ \operatorname{int}_g$.
	By
	Corollary~\ref{coro:unipotent-orbit-dense-in-photon}
	\begin{equation}
		\label{eq:photon-closure-orbit}
		\Phi_\phi=\overline{\phi\left(\exp\braket{x_{-\theta}}\right)\cdot x}\ ,
		\text{ and }
		\Phi_\psi=\overline{\psi\left(\exp\braket{x_{-\theta}}\right)\cdot x} =\overline{\phi\left(\exp\braket{\Ad(g)x_{-\theta}}\right)\cdot x}\ .
	\end{equation}
	It thus suffices to show that 
	\[
	\Ad(g)x_{-\theta}=x_{-\theta}\ .
	\]
	Since $\Phi_\psi $ is tangent at~$x$ to $\Phi_\phi$, by construction we have
	\[
	\pi^{\cF}_{\phi}(x_{-\theta})=\pi^{\cF}_{\psi}(x_{-\theta})\ .
	\]
	It follows that 
	\[
	\pi^{\cF}_{\phi}(x_{-\theta})=\pi^{\cF}_{\phi}(\Ad(g)x_{-\theta})\ .
	\]
	However, both $x_{-\theta}$ and $\Ad(g)x_{-\theta}$ lie in the $\mL_\Theta$-invariant subspace
	$\mathfrak{u}^{\mathrm{opp}}_{\Theta}$, which intersects $\mk p_\Theta$
	trivially and thus $\pi^{\cF}_{\psi}$ is injective
	in restriction to~$\mathfrak{u}_{\Theta}^{\mathrm{opp}}$.  This implies the wanted equality \(\Ad(g)x_{-\theta}=x_{-\theta}\) and the equality $\Phi_\psi=\Phi_\phi$  by Equation~\eqref{eq:photon-closure-orbit}.
\end{proof}
\begin{examples}
	\label{exam:photons}
	Let us illustrate what the photons are for some of the groups admitting a
	positive structure (cf.\ Section~\ref{sec:def-pos}).
	\begin{enumerate}[leftmargin=*]
		\item
		For \noindent{\bf the symplectic group $\mathrm{Sp}(2n,\mathbb{R})$}, the
		 flag variety is  $\cF_\Theta = \mathrm{Lag}(\mathbb{R}^{2n})$  the space
		of Lagrangians. This is also the Shilov boundary of this Hermitian group. In this case $\Theta$ consists of a single element $\theta$.
		
		Let~$x$ be in~$\cF_\Theta$ and
		fix a symplectic basis $\{e_1,\dots, e_n, f_1,\dots,f_n\} $ such that  $x$~is the Lagrangian  $\braket{f_1,\dots, f_n}$. 
		Then $\T_{x} \cF_\Theta$ and $\mathfrak{u}_{\Theta}$ both identify with the space of
		symmetric $n\times n$ matrices, and, under this identification, $\mk g_\theta$~corresponds to the matrices
		whose only non-zero entry is in position $(n,n)$. An example of a photon~$\Phi$ through~$x$ 
		consists of the set of all Lagrangians~$L$ that contains 
		the subspace
		$V\defeq\braket{f_1,\dots, f_{n-1}}$ and are contained in $W\defeq\braket{e_n, f_1, \dots, f_n}$. 
		
		The isomorphism between the photon $\Phi$ and~the projective line $\bP(W/V)$ is given by $L\mapsto L/V$.
		\item\noindent For {\bf the orthogonal group $\mathrm{SO}(p+1,p+k)$},  with $p\geq 1, k\geq 2$, the flag variety $\cF_\Theta=
		\cF_{1,\dots,p}$ is the space of partial isotropic flags consisting of~$p$
		nested isotropic subspaces of dimension~$1$ up to~$p$. 
		The set  $\Theta $ is $ \{\alpha_1, \dots ,\alpha_p\}$, with the standard
		numbering of the simple roots,
		in the Dynkin diagram 
		$\alpha_i$~is connected to~$\alpha_{i+1}$ for $i<p$. We pick a basis $\{e_1, \dots
		,e_{2p+k+1}\}$, such that the orthogonal form is given by
		\[
		\sum_{i=1}^{p+1} x_i x_{2p+k+2-i} - \sum_{i=1}^{k-1} x^{2}_{p+1+i} 
		\ ,
		\]
		and choose $x$ to be the flag whose $j$-th subspace is given by $x^{(j)} =
		\braket{ e_1, \dots,e_{j}}$. To ease further notation we also set $x^{(p+1)} \defeq
		\braket{ e_1, \dots,e_{p+1}}$ and $x^{(0)}\defeq\{0\}$.
		Then a photon associated to the root $\alpha_i$ for $i\leq p-1$ is the set 
		\[
		\Phi_i = \{F \in \cF_\Theta\mid F^{j} = x^{(j)} \text{ for all } j=1, \dots, p, \ j\neq i\}. 
		\]
		A photon associated to the root $\alpha_p$ is the set 
		\[
		\Phi_p = \{F \in \cF_\Theta\mid F^{j} = x^{(j)} \text{ for all } j\neq p \text{
			and } {F^{(p)}} \subset x^{(p+1)}\}. 
		\]
		In all cases, the isomorphism with $\bP^1(\mathbb{R}) = \bP( x^{(i+1)}/x^{(i-1)})$ is now
		given by $F\mapsto F^{(i)}/x^{(i-1)}$. 	\end{enumerate}
\end{examples} 

\subsection{The set of $\theta$-lightlike vectors}
\label{sec:theta_like_vect}
Let $Z_\theta$ 
be  the $L_\Theta$-orbit
of~$x_{-\theta}$ in $\mathfrak{u}_{-\theta}$.

\begin{lemma}
	\label{lemma:theta-light-closed}
	The projectivization $\bP( Z_\theta)$ is a closed subset in $\bP( \mathfrak{u}_{-\theta})$.
\end{lemma}
\begin{proof}
	The group~$\mL_\Theta$ is the almost product of~$\mS_\Theta$,
	$\exp(\mathfrak{b}_\Theta)$ and a compact factor~$\mM_\Theta$.
	Since the action of~$\mL_\Theta$ on~$\mathfrak{u}_{-\theta}$ is irreducible,
	the subgroup $\exp(\mathfrak{b}_\Theta)$ acts by homotheties
	on~$\mathfrak{u}_{-\theta}$.  This
	implies that $\bP( Z_\theta)$ is the $\mM_\Theta \times \mS_\Theta$-orbit of the element
	$t_{-\theta}$ in $ \bP( \mathfrak{u}_{-\theta})$ represented by~$x_{-\theta}$. Since
	$\mathfrak{g}_{-\theta}$ is the highest weight space, the stabilizer of~$t_{-\theta}$
	in~$\mS_\Theta$ is a parabolic subgroup of~$\mS_\Theta$ and hence the
	$\mS_\Theta$-orbit of $t_{-\theta}$ is compact. Since $\mM_\Theta$~is
	compact, the result follows.
\end{proof}
We first prove:
\begin{proposition}
	\label{prop:theta-light-well-def}
	If $\pi^{\cF}(\phi)=\pi^{\cF}(\psi)=x$ then $\pi^{\cF}_{\phi}(Z_\theta)=\pi^{\cF}_{\psi}(Z_\theta)$.
\end{proposition}
\begin{proof}
	Consider the isomorphisms $\bar{\pi}^{\cF}_{\phi}$,
	$\bar{\pi}^{\cF}_{\psi}$ from 	$\mathfrak{g}_0/ \mathfrak{p}_\Theta$ to $\T_x \cF_\Theta$ obtained by modding out
	the common kernel of $\pi^{\cF}_{\phi}$, $\pi^{\cF}_{\psi}$. It is enough to prove that
	$\bar{\pi}^{\cF}_{\phi}({Z}_{\theta}^{\prime})=\bar{\pi}^{\cF}_{\psi}({Z}_{\theta}^{\prime})$ where
	${Z}_{\theta}^{\prime}$ denote the canonical image of~$Z_\theta$ in
	$\mathfrak{g}_0/ \mathfrak{p}_\Theta$.
	
	Since $\pi^{\cF}(\phi)=\pi^{\cF}(\psi)$, it follows that $\phi=\psi\circ
	\operatorname{int}_g$ for some~$g$ in $\mP_\Theta$. Thus, $\pi^{\cF}_{\phi}=\pi^{\cF}_{\psi}\circ\Ad(g)$.
	Since $\mU_\Theta$ acts trivially on $(\mathfrak{u}_{-\theta}\oplus
	\mathfrak{p}_\Theta) / \mathfrak{p}_\Theta$, it also acts trivially on~${Z}_{\theta}^{\prime}$; hence it follows that 
	$\pi^{\cF}_{\phi}(Z_\theta)=\pi^{\cF}_{\psi}\circ\Ad(h)(Z_\theta)$ for the
	element~$h$ in~$\mL_\Theta$ equal to~$g$ modulo~$\mU_\Theta$. Since $Z_\theta$ is $\mL_\Theta$-invariant, the result follows.
\end{proof}

Proposition~\ref{prop:theta-light-well-def} allows us to set:

\begin{definition} A {\em $\theta$-lightlike vector} in $\T_x \cF_\Theta$ is a vector
	in
	\[
	Z^\theta_x\defeq \pi^{\cF}_{\phi}(Z_\theta)\ ,
	\]
	for one (equivalently every) isomorphism~$\phi$ in~$\cG$ such that $\pi^\cF(\phi)=x$. 
\end{definition}

We now have: 
\begin{proposition}\label{pro:PhoUni}
	There exists a unique $\theta$-photon through any $\theta$-light like vector. 
\end{proposition} 
\begin{proof}
	Proposition~\ref{pro:uniq-phot} proves  uniqueness. If $u$~belongs to $Z^\theta_x$, then
	$v=\iota^{\cF}_{\psi}(u)$ belongs to~$Z_\theta$ for any~$\psi$ such that $\pi^{\cF}(\psi)=x$. By
	construction $v=\Ad(g)\cdot x_{-\theta}$ for $g$ in $ \mL_\Theta$. Hence,
	$u=\pi^{\cF}_{\phi}(x_{-\theta})$ for $\phi=\psi\circ \operatorname{int}_g$. Then
	$u$ is tangent to the photon~\mbox{$\phi(\mH_\theta)\cdot x$}.
\end{proof}

The set $\mathbf{\Phi}(x)$ of photons through $x$ identifies by the above
discussion with $\mL_x\slash \mathrm{stab}_{\mL_x}(\Phi) \simeq
\mathbf{P}(Z_{\theta})$ for some $\Phi$ in $ \mathbf{\Phi}(x)$. Thus, $\mathbf{\Phi}(x)$ is a
compact manifold (see Lemma~\ref{lemma:theta-light-closed}).

\subsection{Photon projection} Given a $\theta$-photon $\Phi$, we may now
define the {\em photon projection}~$p_\Phi$.	Let \[
O_\Phi\defeq
\{y\in\cF^{\mathrm{opp}}_{\Theta}\mid \hbox{there exists $x$ in $\Phi$ such that $x\pitchfork y$}\}\ ,
\]
and observe that $O_\Phi$~is an open set.   We have:
\begin{proposition}\label{p.Phipro}
	Let~$y$ be in~$O_\Phi$. Then:
	\begin{enumerate}
		\item\label{item1bis:p.Phipro} For any~\(w\) in~\(\Phi\) that is transverse to~\(y\) there is~\(\psi\) in~\(\mathcal{G}\) such that \(\pi^\cL(\psi)=(w,y)\) and \(\Phi = \psi( \mathsf{H}_\theta)\cdot w\).  Fixing such a~\(\psi\), we set $\mH_\Phi = \psi(\mH_\theta)$.  
		\item\label{item1:p.Phipro} There exists a unique~$z$ in~$\Phi$ such that
		$z$~is not transverse to~$y$.
		\item\label{item2:p.Phipro} The group $\mH_\Phi\cap\mU_{y}$ is a $1$-parameter
		unipotent subgroup.
		\item \label{item3:p.Phipro} $z$~is 
		the unique fixed point of $\mH_\Phi\cap\mU_y$ on~$\Phi$.
	\end{enumerate}
\end{proposition}
\begin{proof}
	Point~(\ref{item1bis:p.Phipro}) follows from Proposition~\ref{prop:photons-through-x}.(\ref{item:2:prop:photons-through-x}).  It means that $\psi(\mP_\Theta)=\mP_w$ and
	$\psi(\mU_{\Theta}^{\mathrm{opp}})=\mU_y$.

	It follows that $\mH_\Phi\cap\mU_y$ is the 	unipotent group
	$\mathsf{V}_y=\psi( \exp( \mathbb{R} x_{-\theta}))$ (thus item~(\ref{item2:p.Phipro}) holds).
	Since by Proposition~\ref{pro:PhoRP} a photon is identified with $\bP^1(\mathbb{R})$ as an $\sld$-space, it follows that the orbit of~$w$ under $\mathsf{V}_y$ is
	$\Phi\smallsetminus\{z\}$ for some~$z$ in~$\Phi$. Observe now that $\mathsf{V}_y\cdot
	w$ is
	precisely $\Phi\cap O_y$, where
	\[
	O_y\defeq\mU_y\cdot w=\{x\in\cF_{\Theta}\mid \ x\pitchfork y\}\ .
	\]
	This implies that $z$~is not transverse to~$y$ and concludes
	item~(\ref{item1:p.Phipro}). Since furthermore the action of~$\mathsf{V}_y$ on
	$\Phi\smallsetminus\{z\}$ is simply transitive, we get item~(\ref{item3:p.Phipro}).
\end{proof}

We now define:
\begin{definition}
	Given a photon~$\Phi$, the {\em photon projection} is the map~$p_\Phi$
	from~$O_\Phi$ to~$\Phi$, which associates to~$y$, the point $p_\Phi(y)$
	which is the unique point in~$\Phi$ not transverse to~$y$.
\end{definition} 
From the definition, it follows that the graph of  $p_\Phi$ is algebraic, hence $p_\Phi$ is continuous.

\begin{remark}
	In the case of the Shilov boundary of a tube type Hermitian Lie group, the
	photon projection is also defined in \cite[Section~6.2.2]{Galiay}
         (cf.\ also \cite[Lemma~6.3.7]{galiay_phd}). \end{remark}

We can rephrase point~(\ref{item1:p.Phipro}) of Proposition~\ref{p.Phipro} by saying that transversality between an element of~$\Phi$ and an element of~$O_\Phi$ can be asserted using the photon projection:
\begin{corollary}
	\label{coro:photon-proj-trans}
	Let~$y$ be in~$O_\Phi$ and let~$x$ be in~$\Phi$.  Then $y$ is transverse to~$x$ if and only if $p_\Phi(y)$ is not equal to~$x$.
\end{corollary}

The following result is also a direct consequence of the definition:
\begin{proposition}\label{pro:phi-equi}
	Let~\(\psi\) in~\(\mathcal{G}\) such that \(x=\pi^{\mathcal{F}}(\psi)\) belongs to~\(\Phi\) and \(\Phi= \psi(\mathsf{H}_\theta)\cdot x\).  Set \(\mathsf{H}_\Phi = \psi(\mathsf{H}_\theta)\).  Then the photon projection~$p_\Phi$ is equivariant under~$\mH_\Phi$.
\end{proposition}
As a corollary we give another characterization of the photon projection.

\begin{corollary}\label{coro:fixed-fibre}
	Let $x$ be a point in a photon $\Phi$ and~$y$ a point in $O_\Phi$.  Let~\(\phi\) be in~\(\mathcal{G}\) such that there exists~\(w\) in~\(\Phi\) with \(\pi^{\mathcal{F}}(\phi)=(w,y)\) and \(\Phi= \phi( \mathsf{H}_\theta)\cdot w\).  Set \(\mathsf{H}_\Phi = \phi( \mathsf{H}_\theta)\).  Then  $p_\Phi(y)=x$ if and only if we have the following inclusion:
	\begin{equation*}
		\ms H_\Phi \cap \ms U_x\ \subset \ \ms U_y\ .
	\end{equation*}
\end{corollary}
\begin{proof}
	Suppose that $p_\Phi(y)=x$. By Proposition
	\ref{p.Phipro}.(\ref{item3:p.Phipro}), the intersection $\mH_\Phi \cap \mU_y$
	is contained in~$\mP_x$; this intersection is therefore the unipotent
	subgroup of~$\mH_\Phi$ fixing~$x$. As this unipotent subgroup of~$\mH_\Phi$ is contained
	in~$\mU_x$ we also get the inclusion $\mH_\Phi \cap \mU_x \subset \mU_y$.
	
	Conversely, assume that $\mH_\Phi \cap \mU_x \subset \mU_y$. Then $\mH_\Phi
	\cap \mU_x$ and $\mH_\Phi \cap \mU_y$ are two unipotent subgroups
	of~$\mH_\Phi$ contained one in the other. This forces the equality  $\mH_\Phi
	\cap \mU_x = \mH_\Phi \cap \mU_y$ and thus the inclusion $\mH_\Phi \cap \mU_y
	\subset \mP_x$. By Proposition
	\ref{p.Phipro}.(\ref{item3:p.Phipro}) we conclude that $p_\Phi(y)=x$.
\end{proof}

In the case when $\Theta$ is invariant by the opposition involution, points in
a given fiber of~$p_\Phi$ are not transverse:

\begin{corollary}\label{coro:fibre-not-transverse}
	Suppose that $\Theta$~is invariant by the opposition involution so that we
	identify $\cF_{\Theta}^{\mathrm{opp}}$ with~$\cF_\Theta$ and transversality
	makes sense between elements of~$\cF_\Theta$.
	Let $x$ be a point in a photon $\Phi$. If $y$ and $z$ in $O_\Phi$ are
	transverse, then $p_\Phi(y)\not=p_\Phi(z)$.
\end{corollary}
\begin{proof} 
	If $p_\Phi(y)=p_\Phi(z)$, then
	$\ms U_y\cap\ms U_z$ is not reduced to zero by Corollary~\ref{coro:fixed-fibre}, hence $y$ and $z$ are not
	transverse.
\end{proof}

\begin{examples}
	Let us illustrate the photon projections in the Examples~\ref{exam:photons} we
	discussed before, we use the notation (and the photons) introduced there.
	\begin{enumerate}[leftmargin=*]
		\item For {\bf the symplectic group $\mathrm{Sp}(2n,\mathbb{R})$} with $x  =
		\braket{f_1,\dots, f_n}$ and $\Phi$ the photon given by the orbit of
		$\mH_\theta$ (here the stabilizer of~$x$ in~$\mathsf{H}_\theta$ is the standard opposite Borel subgroup). Then
		\[
		O_\Phi = \{ L \in \mathrm{Lag}(\mathbb{R}^{2n})\mid  \dim(L \cap
		\braket{e_n, f_1, \dots, f_n}) = 1\} \ , 
		\]
		and the projection sends $L $ in $ O_\Phi$ to the Lagrangian $(L \cap
		\braket{e_n, f_1, \dots, f_n}) \oplus \braket{f_1,\dots,f_{n-1}}$. 
		
		\item For {\bf the orthogonal group $\mathrm{SO}(p+1,p+k)$},  and $x$ the flag whose
		$j$-th subspace is given by $x^{(j)} = \braket{e_1, \dots,e_{j}}$ ($j=1,
		\dots, p$). For  $1\leq i\leq p$, let
		$\Phi_i$ be the photon through~$x$ associated to~$\alpha_i $ in $ \Theta$, 
		then we have   \[
		O_{\Phi_i} =\bigl\{ F\in  \cF_{1,\dots,p}\,\bigl|\bigr.\,
		F^{(j)\perp} \cap x^{(j)}=\{0\}  \text{ for all } j\neq i,\text{ and } \dim(F^{(i)\perp} \cap
		{x^{(i+1)}})= 1\bigr\}, 
		\]
		(again with the notation $x^{(p+1)} =
		\braket{ e_1, \dots,e_{p+1}}$).
		The projection sends $F$ to the flag whose $j$-th subspace is $x^{(j)}$  for
		$j\neq i$ and the $i$-th subspace is $ x^{(i-1)} \oplus
		(F^{(i)\perp} \cap x^{(i+1)})$. 
	\end{enumerate}
\end{examples}

\subsection{Photon projection and photon cross-ratio}	
\label{sec:photon-cross-ratio}
Let $\eta$ be a $\Theta$-compatible dominant
form. 
Let $\Phi$ be a $\theta$-photon. In this section we will prove the following: 

\begin{proposition}\label{prop:photon-cross-ratio-equal-1}
	Let~$z$ and~$y$ be in $O_\Phi$ such that
	$p_\Phi(y)=p_\Phi(z)$. Then for all~$x$ and~$w$ in~$\Phi$ which are pairwise
	transverse to~$z$ and~$y$ (i.e.\ distinct from $p_\Phi(y)=p_\Phi(z)$), we have the quadruple \((x,w,z,y)\) belongs to \(O^{\square}_{\Theta}\) (cf.\ Section~\ref{sec:sympl-reint}) and
	\[
	\bb^\eta(x,w,z,y)=1\ .\label{eq:bb=1}
	\]
\end{proposition}

Let us first show the following:

\begin{lemma}[\sc Infinitesimal lemma]\label{lem:infi-bphi} 
	Let~$u$ be a 
	tangent vector to~$\Phi$ at a point~$z$. Let $c\colon [-1,1]\to\cF_{\Theta}^{\mathrm{opp}}$, be a
	curve differentiable at~$0$, with $y=c(0)$
	transverse to~$z$, such that $p_\Phi(c(t))$ is constant in~$t$. Let  $v=\dot c(0)$.  Then
	\[
	\bigdual{\Omega\left((u,0),(0,v)\right)}{\eta}=0 \ .
	\]
\end{lemma}
\begin{proof}
	We can assume that $u$~is non-zero.
	
	Let us write $c(t)=c_0(t)\cdot y$ with $c_0(t)$ in $\ms U_z$ (this is
	possible in a neighborhood of~$0$). 
	Let~$\phi$ be in~$\cG$ such that
	$\pi^\cL(\phi)=(z,y)$ and \(\Phi = \phi( \mathsf{H}_\theta)\cdot z\).  Set \(\mathsf{H}_\Phi = \phi( \mathsf{H}_\theta)\).
	Let
	$x=p_\Phi(c(t))$, we have that $x\neq y$. Since $\mH_\Phi \cap \mU_x$ acts
	simply transitively on $\Phi\smallsetminus \{x\}$, there is~$w$ in the Lie
	algebra of  $\mH_\Phi \cap \mU_x$ such that $u = \frac{d}{ds}|_{s=0}
	\exp(sw)\cdot z$.
	By
	Corollary~\ref{coro:fixed-fibre}, we have, for all real $s$,
	$$
	\exp(s w)\in \mH_\Phi\cap \ms U_x\ \subset \ms U_{c(t)}=c_0(t)\ms U_y c_0(t)^{-1}\ .
	$$
	Thus, for all real number~$s$ and all~$t$ close enough to~$0$
	$$
	c_0(t)^{-1} \exp(s w)c_0(t) \in \ms U_y\ .
	$$
	After taking the derivatives at $s=0$ and $t=0$, it follows that
	$$
	[w,\dot{c}_0(0)]\in \mk u_y\ .
	$$
	
	By the choice of~\(\phi\), one has $\phi_{*}^{-1}(\dot{c}_0(0))$ belongs to
	$\mathfrak{u}_{\Theta}^{\mathrm{opp}}$, $\phi_{*}^{-1}(w)$ belongs to
	$\mathfrak{u}_\Theta$, and the previous equation means that
	$\phi_{*}^{-1}([w,\dot{c}_0(0)])$  belongs to
	$\mathfrak{u}_\Theta$.
	By the construction of~$\Omega$ (Section~\ref{app:curv}), one has 	$$
	\bigdual{\Omega\left((u,0),(0,v)\right)}{\eta} = \bigdual{p(\phi_{*}^{-1}([w,\dot{c}_0(0)]))}{\eta}=0
	$$
	since $p(\phi_{*}^{-1}[(w,\dot{c}_0(0)]))=0$.
	This concludes the proof.
\end{proof}
We need another lemma:
\begin{lemma}[\sc Fiber is connected]\label{lem:connected}
	Let $x$ be a point in $\Phi$, then the set $p_\Phi^{-1}(x)\subset O_\Phi$ is
	a connected submanifold.
\end{lemma}
\begin{proof}
	Let~\(\phi\) be in~\(\mathcal{G}\) such that \(\Phi = \phi(\mathsf{H}_\theta)\cdot x\) and set \(\mathsf{H}_\Phi = \phi(\mathsf{H}_\theta)\).

	Since $p_\Phi$~is smooth (as it is algebraic) and by $\mH_\Phi$-equivariance (Proposition \ref{pro:phi-equi}),
	we see that $p_\Phi$~is a submersion so that $W\defeq p_\Phi^{-1}(x)$ is a submanifold.
	
	Let~$z$ be in~$\Phi\smallsetminus\{x\}$. Recall that by definition of the
	photon projection, $W$ is included in the set $O_z$ of points transverse to~$z$. Let $\ms u$ be the unipotent subgroup $\ms H_\Phi\cap \ms U_z$. We then have a continuous map 
	$$
	\xi\colon O_z\to \ms u\ ,
	$$ 
	characterized uniquely by $p_\Phi(w)=\xi(w)\cdot x$. Let us consider the map from $\ms u\times W$ to $O_z$ given by 
	$$
	\psi(u,w)=u\cdot w .
	$$
	The map $\psi$ is a diffeomorphism: its inverse is given by
	$$
	w\mapsto (\xi(w),\xi(w)^{-1} w)\ ,
	$$
	(this follows from the $\ms u$-equivariance of~$p_\Phi$).
	Thus, $O_z$ is diffeomorphic to $\ms u\times W$, hence $W$~is connected since $O_z$~is.
\end{proof}

\begin{proof}[Proof of Proposition~\ref{prop:photon-cross-ratio-equal-1}]
	Let $p=p_\Phi(y)=p_\Phi(z)$.
	
	Let~$W=p_\Phi^{-1}(p)$. By Lemma~\ref{lem:connected}, 
	we can find a continuous, piecewise $\mathcal{C}^1$ curve $c_1(t)$ joining~$y$
	to~$z$.  Any point in~$W$ is transverse to any point in the interval in~$\Phi$
	joining~$x$ to~$w$ and not containing $p$; let $c_0\colon [0,1]\to \cF_\Theta$ be
	a $\mathcal{C}^1$ parameterization of this interval.
	
    The map $f\colon [0,1]^2\to \cL_\Theta \mid
	(s,t)\mapsto (c_0(s), c_1(t))$ shows that \((x,w,z,y)\) belongs to~\(O^\square\).

	Then by
	Proposition~\ref{pro:cr-om} applied to~\(f\) and  the Infinitesimal Lemma~\ref{lem:infi-bphi},
	the equality   
	\[
	\bb^\eta(x,w,z,y)=1\ ,
	\]
	holds. 
\end{proof}
Using the cocycle identity we obtain as a corollary that
\begin{equation}
	\label{eq:equality-of-bb-eta}
	\bb^\eta(x,w,z_0,y_0)= \bb^\eta(x,w,z'_0,y'_0) \ ,
\end{equation}
if $p_\Phi(z_0) = p_\Phi(z'_0)$  and $p_\Phi(y_0) = p_\Phi(y'_0)$

Therefore,
Proposition~\ref{prop:photon-cross-ratio-equal-1} allows us to define the
photon cross-ratio associated to the $\Theta$-compatible proximal dominant
weight~$\eta$:
\begin{definition}
	Let $x,w, z,y$ be points in $\Phi$ satisfying the transversality condition
	$x\neq y$ and $w\neq z$.
	The {\em photon cross-ratio} on $\Phi$ associated to $\eta$ is
	\[
	\bb^\eta_\Phi(x,w,z,y)\defeq \bb^\eta(x,w,z_0,y_0)\ ,
	\]
	where $z_0$, $y_0$ are any points such that $p_\Phi(z_0)=z$,
	$p_\Phi(y_0)=y$.
\end{definition}
\begin{remark}
	\label{rem_photonCR_for_form}
For \emph{any} $\Theta$-compatible dominant \emph{form}~\(\eta\), we can define the photon cross-ratio as a map \(\bb^{\eta}_{\Phi}\colon O^\square(\Phi)\to \mathbb{R}\) using the same formula.

Indeed, when \((x,w,z,y)\) belongs to \(O^\square(\Phi)\), namely when there are disjoint intervals \([x,w]\) and \([z,y]\) in~\(\Phi\) with extremities~\(x\) and~\(w\) (respectively~\(z\) and~\(y\)), for any $z_0$ and $y_0$ such that $p_\Phi(z_0)=z$ and
	$p_\Phi(y_0)=y$, the quadruple \((x,w,z_0, y_0)\) belongs to \(O^{\square}(\mathcal{F}_\Theta)\). Indeed, by the properties of the photon projection any point in \([x,w]\) is transverse to any point in \(p_{\Phi}^{-1}([z,y])\).  Choosing a \(\mathcal{C}^1\) parameterization \(c_0\colon [0,1]\to [x,w]\) and a \(\mathcal{C}^1\) map \(c_1\colon [0,1]\to p_{\Phi}^{-1}([z,y])\) joining \(z_0\) and~\(y_0\) then the map $[0,1]^2\to \cL_\Theta \mid
	(s,t)\mapsto (c_0(s), c_1(t))$ shows that \((x,w,z,y)\) belongs to~\(O^\square( \mathcal{F}_\Theta)\).
\end{remark}

\subsection{Photon cross-ratio and projective cross-ratio}

The photon cross-ratio is a cross-ratio on a projective line invariant under the projective group, therefore on positive quadruples it is a power of the projective cross-ratio. More generally for all quadruples, we have the following 

\begin{proposition}\label{pro:phot-cr}
	Let $\eta$ be a $\Theta$-compatible proximal dominant
	weight and $\theta$ in $\Theta$, 	then for any $\theta$-photon
	$\Phi$
	\begin{equation}
		\label{eq:pro:phot-cr}
		\bb^\eta_\Phi(x,y,z,w)=[x,y,z,w]^{\dual{ h_\theta}{ \eta}}\ ,
	\end{equation}
	where $[a,b,c,d]$ denotes the projective cross-ratio on
	$\Phi\cong\mathbf{P}^1(\mathbb R)$. In particular if $\omega_\theta$ is the fundamental weight associated to $\theta$,
	\begin{equation}
		\label{eq:pro:phot-cr1}
		\bb^{\omega_\theta}_\Phi(x,y,z,w)=[x,y,z,w]\ .
	\end{equation}
	When \(\eta\) is a $\Theta$-compatible dominant form, Equation~\ref{eq:pro:phot-cr} holds for \((x,y,z,w)\in O^\square(\Phi)\).
\end{proposition}

\begin{proof} 
	Let~$\psi$ be in~$\cG$ and let~$f_0$ be the attracting fixed point in~$\cF_\Theta$ for the action of~$h = \psi(\exp(a))$ for some (and equivalently any) $a$~in the open Weyl chamber.
	We can assume that $\Phi = \mH_\Phi \cdot f_0$
	where $\mH_\Phi = \psi(\mH_\theta )$.
	Let also~$f_\infty$ be the repelling fixed point
	in~$\cF_{\Theta}^{\mathrm{opp}}$ for~$h$. The $\mH_\Phi$-orbit $\Phi^\vee =
	\mH_\Phi \cdot f_\infty$ is also 
	equivariantly isomorphic
	to the projective line $\mathbf{P}^1(\mathbb R) \simeq \mH_\theta /
	\mB_\theta$ (where $\mB_\theta$ is the standard Borel subgroup in
	$\mH_\theta$). Precisely, the isomorphisms are given by
	$$\begin{array}{cllccll}
		\mH_\theta / \mB_\theta &\longrightarrow &\Phi &\qquad& \mH_\theta / \mB_\theta &\longrightarrow &\Phi^\vee \\
		g\cdot \mB_\theta & \longmapsto &\psi(g)\cdot f_0&&
		g\cdot \mB_\theta & \longmapsto &\psi(g \dot{s}_\theta) \cdot f_\infty,
	\end{array}$$
		where $\dot{s}_\theta$ is an element of $\mH_\theta$ representing the non-trivial
	element in the Weyl group of $\mH_\theta$.
	
	These maps allow defining an $\mH_\Phi$-equivariant identification $z\mapsto z^\vee$ from
	$\Phi$ to $\Phi^\vee$. Moreover by equivariance, this identification has the following properties:\begin{itemize}
		\item The point $z^\vee$ is not transverse to $z$. Indeed from the Schubert's cells decomposition,
		$\psi(\dot{s}_\theta) \cdot f_\infty$ is not
		transverse to~$f_0$.
		\item For all~$w$ in~$\Phi$ distinct from~$z$, the elements $w$
		and~$z^\vee$ are transverse, indeed  $f_0$ and $f_\infty$ are
		transverse.
	\end{itemize}
	This implies that $p_\Phi(z^\vee)=z$, and thus we can use this map $z\mapsto z^\vee$ to calculate the photon
	cross-ratio. For
	$x$, $y$, $z$, and $t$ in~$\Phi$,
	\begin{equation*}
		\bb^{\eta}_{\Phi}(x,y,z,t)= \bb^\eta(x,y,z^\vee,t^\vee)\ .
	\end{equation*}
	Let now~$E$ be the real vector space underlying an irreducible
	proximal
	representation $\tau\colon
	\mG \to \gl(E)$ of highest weight~$\eta$. We choose a basis
	$(e_i)_{i=1}^{d}$ of weight vectors such that $e_1$ generates the highest
	weight space with respect to the Cartan subspace $\psi_*(\mathfrak{a})$ of~$\mG$.
	The equivariant maps $\Xi\colon
	\cF_\Theta\to \mathbb{P}(E)$ and $\Xi^*\colon \cF_{\Theta}^{\mathrm{opp}}\to \mathbb{P}(E^*)$
	are then given 
	by $\Xi( g\cdot f_0) = \tau(g)\cdot [e_1]$ and $\Xi^*( g\cdot
	f_\infty) = \tau^{*}(g)\cdot [e^{*}_{1}]$ where $\tau^*\colon \mG \to
	\gl(E^*)$ is the contragredient representation $g\mapsto {}^T
	\tau(g)^{-1}$. Then
	\begin{equation*}\bb^\eta(a,b,c,d)=\bb^E(\Xi(a), \Xi(b), \Xi^*(c), \Xi^*(d))\ ,   \end{equation*} where
	$\bb^E( [v], [w], [\phi], [\psi]) = \frac{ \dual{v}{\phi}
		\dual{w}{ \psi}}{ \dual{v}{ \psi} \dual{w}{ \phi}}$.
	We now prove 
	Equality~\eqref{eq:pro:phot-cr}. By continuity, it is sufficient to treat the case when $x\neq y$, and by $\mH_\theta$-equivariance we may assume
	that 
	$$x= \psi(\dot{s}_\theta)\cdot f_0\ , \ y = f_0\ , \ z= \psi(\exp( \lambda x_\theta)\dot{s}_\theta)
	\cdot f_0\ \text{, and}\ t= \psi(\exp( \mu x_\theta)\dot{s}_\theta)
	\cdot f_0\ .$$ 
	The point~$x$ is thus the repelling fixed point in~$\Phi$ for the Weyl
	chamber of $\mH_\theta$.
	The projective cross-ratio $[x,y,z,t]$ is equal to
	$[0,\infty, \lambda,\mu] =\lambda/\mu$. Furthermore, we have 
	\begin{align*}
		\bb^{\eta}_{\Phi}(x,y,z,t)
		&= \bb^{\eta}(x,y,z^\vee,t^\vee)\\
		&= \bb^{\eta}( \psi(\dot{s}_\theta)\cdot f_0, f_0, \psi( \exp( \lambda x_\theta))\cdot  f_\infty, \psi(\exp(
		\mu x_\theta))\cdot  f_\infty)\\
		&= \bb^E (\tau\circ\psi(\dot{s}_\theta) [e_1], [e_1], \tau^*\circ\psi(\exp( \lambda
		x_\theta))[e_{1}^{*}], \tau^*\circ\psi(\exp( \mu x_\theta))[e_{1}^{*}])\\
		&= \frac{ \dual{ \tau\circ\psi(\dot{s}_\theta) e_1}{ \tau^*\circ\psi(\exp( \lambda
				x_\theta))e_{1}^{*}} \dual{ e_1}{  \tau^*\circ\psi(\exp( \mu
				x_\theta))e_{1}^{*}} }{ \dual{\tau\circ\psi(\dot{s}_\theta) e_1}{ \tau^*\circ\psi(\exp( \mu
				x_\theta))e_{1}^{*}} \dual{ e_1}{  \tau^*\circ\psi(\exp( \lambda
				x_\theta))e_{1}^{*}} }\\
		&= \frac{ \dual{\tau\circ\psi(\dot{s}_\theta) e_1}{  \tau^*\circ\psi(\exp( \lambda
				x_\theta))e_{1}^{*}}  }{ \dual{\tau\circ\psi(\dot{s}_\theta) e_1}{  \tau^*\circ\psi(\exp( \mu
				x_\theta))e_{1}^{*}} }
	\end{align*}
	since $\dual{ e_1}{ \tau^*\circ\psi(\exp( \mu
		x_\theta))e_{1}^{*}} = \dual{  \tau\circ\psi(\exp( \mu
		x_\theta)) e_1}{e_{1}^{*}} = \dual{  e_1}{e_{1}^{*}} =1$ and similarly $\dual{  e_1}{ \tau^*\circ\psi(\exp( \lambda
		x_\theta))e_{1}^{*}} =1$. The proposition will be proven if we can show
	that there is a non-zero number~$c$ such that, for all~$\lambda$
	in~$\mathbb{R}$,
	\begin{equation}
		\label{eq:dual-power}
		\dual{ \tau\circ\psi(\dot{s}_\theta) e_1}{ \tau^*\circ\psi(\exp( \lambda
			x_\theta))e_{1}^{*}} = c \lambda^{\dual{\eta}{h_\theta}}\ .
	\end{equation}
	For this, note first that 
	$$\dual{ \tau\circ\psi(\dot{s}_\theta) e_1}{ \tau^*\circ\psi(\exp( \lambda
		x_\theta))e_{1}^{*}} =\dual{ (\tau\circ\psi(\exp( \lambda
		x_\theta)) \tau\circ\psi(\dot{s}_\theta)) e_1}{ e_{1}^{*}} .$$ Furthermore, denoting
	$\tau_*\colon \mathfrak{g}\to \mathrm{End}(E)$ the Lie algebra homomorphism
	associated to~$\tau$ and $\psi_*\colon \mathfrak{g}_0\to \mathfrak{g}$ the
	isomorphism associated to~$\psi$, classical calculations in $\mathfrak{sl}_2$-modules
	give that $\tau\circ\psi(\dot{s}_\theta) e_1$ is a non-zero multiple of
	$(\tau_*\circ\psi_*(x_{-\theta}))^{ \dual{h_\theta}{\eta}} e_1$ and
        that
        \begin{align*}
        \dual{(\tau_*\circ\psi_*(x_\theta))^k (\tau_*\circ\psi_*(x_{-\theta}))^{ \dual{h_\theta}{\eta}}
		e_1}{ e_{1}^{*}}& =0 \text{ if } k\neq \dual{ h_\theta}{\eta}
                                  \text{ and }\\
          \dual{(\tau_*\circ\psi_*(x_\theta))^{ \dual{h_\theta}{\eta}} (\tau_*\circ\psi_*(x_{-\theta}))^{ \dual{h_\theta}{\eta}}
		e_1}{ e_{1}^{*}}& \neq 0.  
        \end{align*}
	 Using the equality $\tau\circ\psi(\exp(\lambda
	x_\theta)) = \sum_k \frac{1}{k!} \lambda^k (\tau_*\circ\psi
	_*(x_\theta))^k$ gives the existence of~$c$ such that
	Equation~\eqref{eq:dual-power} holds, hence the
	wanted conclusion.
\end{proof}

\begin{remark}
	Galiay obtained Proposition~\ref{pro:phot-cr} for the photons in
	the Shilov boundary of tube type Hermitian Lie groups, see
	\cite[Lemma 6.11]{Galiay} (cf.\ also \cite[Lemma~6.3.13]{galiay_phd}).
\end{remark}
\section{Positivity}\label{sec:def-pos}

Now we will restrict to semisimple Lie groups $\mG_0$ that admit a {\em
	positive structure relative to a subset} $\Theta$ of  $\Delta$ as defined in \cite[Definition~3.1]{GuichardWienhard_pos}. 
By definition, 
this means that for every $\theta$ in  $\Theta$ there exists a
convex acute open cone $c_\theta $ inside $\mk u_\theta$, which is invariant by
$\mL_{\Theta}^{\circ}$. 
(Note that in \cite{GuichardWienhard_pos}, the symbol
$c_\theta$ stands for the \emph{closed} invariant cone, but, since the closed cone does not play a big role in the present work, we simplify notation and denote here by  $c_\theta$ the  open cone.)

As the action of~$\mL^{\circ}_{\Theta}$ on~$\mathfrak{u}_\theta$ is
irreducible, there exist exactly two such invariant cones (namely $c_\theta$
and $-c_\theta$). We distinguish between the two invariant cones by requesting that the element~$x_\theta$ (of the $\skd$-triple associated
to~$\theta$, cf.\ Section~\ref{sec:skd-triples}) belongs to the closure of~$c_\theta$ (cf.\ 
\cite[Theorem~3.13]{GuichardWienhard_pos}). Similarly, the
cone~$c_{-\theta}$ in~$\mathfrak{u}_{-\theta}$ is the invariant cone whose
closure contains~$x_{-\theta}$. Equivalently one can set $c_{-\theta} =
-\sigma(c_\theta)$ where $\sigma$~is the Cartan involution.  The closed invariant cones that are the closures of \(c_\theta\) and \(c_{-\theta}\)
will be denoted \(\bar{c}_\theta\) and \(\bar{c}_{-\theta}\) respectively.

\begin{remark}
	There are exactly four families of simple Lie groups admitting a positive structure
	relative to some subset $\Theta$ of $\Delta$ (see \cite[Theorem~1.1]{GuichardWienhard_pos}). 
	Up to isogeny, these correspond to the following cases: 
	\begin{enumerate}
		\item $\mG_0$ is a split real form, and $\Theta = \Delta$;
		\item $\mG_0$ is Hermitian of tube type and of real rank~$r$ and $\Theta = \{ \alpha_r\}$, where $\alpha_r$ is the long simple restricted root;
		\item $\mG_0$ is $\mathrm{SO}(p+1,p+k)$, $p>1$, $k>1$ and
		$\Theta = \{ \alpha_1, \dots, \alpha_{p}\}$, where $\alpha_1, \dots,
		\alpha_{p}$ are the long simple restricted roots; 
		\item $\mG_0$ is the real form of $F_4$, $E_6$, $E_7$, or of $E_8$ whose system of restricted roots
		is of type~$F_4$, and $\Theta = \{ \alpha_1, \alpha_2\}$, where $\alpha_1,
		\alpha_{2}$ are the long simple restricted roots.
	\end{enumerate}
	In general a semisimple Lie group admits a positive structure relative to~$\Theta$ if it
	is the almost product of simple Lie groups $\mG_i$, $i= 1, \dots, n$, where
	each $\mG_i$ admits a positive structure relative to~$\Theta_i$ and $\Theta  = \Theta_1 \cup
	\cdots \cup \Theta_n$. Here the parabolic subgroup~$\mP_\Theta$ is the almost
	direct product of the parabolic subgroups in the factors~$\mG_i$ and the flag
	manifold $\cF_\Theta$ is the product of the flag manifolds corresponding to the
	different factors.
\end{remark}

The fact that $\mG_0$ admits a positive structure relative to~$\Theta$ implies in particular that 
\begin{enumerate}
	\item the parabolic subgroup $\mP_\Theta$ is conjugate to its opposite, 
	\item we have $\dim \mk g_\theta = 1$ for all $\theta$ in $\Theta$ and therefore
        \(\omega_\theta\) is a proximal dominant weight.  As a consequence
        every \(\Theta\)-compatible dominant weight, being an integral
        combination of the \(\omega_\theta\), for~\(\theta\) in~\(\Theta\), is proximal;
          thus in particular, all the results of
          Section~\ref{sec:photo-construct} apply. 
\end{enumerate}

Furthermore, the positive structure gives rise to an open and sharp
semigroup
$\mN_\Theta$ in $\mU_\Theta$ invariant under $\mL_{\Theta}^{\circ}$. The
properties of~$\mN_\Theta$ will be reflected in the properties of the diamonds
that we introduce next. We refer to
\cite{GuichardWienhard_pos} for a precise description of~$\mN_\Theta$ and its
algebraic properties.

\subsection{Diamonds}
\label{sec:diamonds-positive-n}

Let~$x$ and~$y$ be transverse points in~$\cF_\Theta$ (recall that
$\cF_{\Theta}^{\mathrm{opp}} \simeq \cF_\Theta$), and consider an element~$\psi
$ in~$\cG$ such that 
$\pi^{\cL}(\psi) =
(x,y)$. Note that $\psi$~depends only on~$(x,y)$ up to
precomposition by the conjugation by an
element in~$\mL_\Theta$. 
Then by \cite[Proposition 2.6]{Guichard:2021aa}  and \cite[Theorem
8.1]{GuichardWienhard_pos} we have:

\newcommand{\diamant}{D}
\begin{proposition}
	Given $(x,y)$ and $\psi$ as above, the set 
	$\psi(\mN_\Theta)\cdot x$ is a connected component of the set 
	\[
	\{z \in \cF_\Theta\mid z\pitchfork x \text{ and } z\pitchfork y \}\ .
	\]
\end{proposition}
Such a connected component~$\diamant$ is called a \emph{diamond with
	extremities~$x$ and~$y$}. A diamond  with extremities~$x$ and~$y$ will be denoted $\diamant(x,y)$; observe the
slight abuse of notation since $\diamant(x,y)$ does not only depend on~$x$
and~$y$: there are $2^{\sharp \Theta}$ diamonds with the same extremities.  (Note that in order to construct all the diamonds, all the possible choices of cones ---and hence of semigroups in~\(\mU_\Theta\)--- are to be allowed, cf.\ \cite[Section~13.4]{GuichardWienhard_pos}.)  

The map~$\pi^{\cF}_{\psi}$
(see Section~\ref{sec:parab-subgr-flag}) gives an identification of $\mk
u^{\mathrm{opp}}_{\Theta} $ with~$ \T_x \mathcal{F}$ and the  map~$\varpi^{\cF^{\mathrm{ opp}}}_{\psi}$  gives an identification of~$\mk u_{\Theta} $
with~$\T_y \mathcal{F}$. 
The tangent cone at~$x$ of the closure of a diamond $\diamant=\psi(\mN_\Theta)\cdot x$ is exactly the
image by $\pi^{\cF}_{\psi}$  of the closed cone $\sum_{\theta\in\Theta}
\bar{c}_{-\theta}$ inside $\mathfrak{u}^{\mathrm{opp}}_{\Theta}$
(cf.\ 
\cite[Section~8.5]{GuichardWienhard_pos}):
\begin{itemize}
	\item The tangent vectors belonging to this tangent cone will be
	called \emph{non-negative} (with respect to~$\diamant$);
	\item  The tangent vectors in
	the (relative) interior of the tangent cone, i.e.\ those belonging to
	$\pi^{\cF}_{\psi}\bigl(\sum_{\theta\in\Theta} {c}_{-\theta}\bigr)$, are called
	\emph{positive}.
\end{itemize}
Equivalently, a vector~$v$ in $\T_x \cF_\Theta$ is positive (respectively
non-negative) with respect to~$\diamant$ if $\iota_{\psi}^{\cF}(v)$ belongs to $\sum_{\theta\in\Theta}
{c}_{-\theta}$ (respectively to $\sum_{\theta\in\Theta}
\bar{c}_{-\theta}$).

Note that the shape of a diamond near its extremities should really be thought of
as a ``cusp''; 
indeed the diamond is open in~$\cF_\Theta$ whereas the dimension
of its tangent cone at~$x$ is of positive codimension in $\mathsf{T}_x\cF_\Theta$ as soon as the set 
$\Theta$ has at least 2 elements.

The tangent cone at~$y$ of~$\diamant$ is the image by~$\varpi^{\cF^\mathrm{opp}}_{\psi}$ of
$\sum_{\theta\in\Theta} \bar{c}_{\theta}$.

The subset $\psi(\mN_{\Theta}^{-1})\cdot x$ is also a diamond with extremities~$x$ and~$y$ and is called the \emph{diamond opposite to~$\diamant$} and will be denoted by $\diamant^\vee$.
Its tangent cone at~$x$ is the image by $\pi^{\cF}_{\psi}$ of
$\sum_{\theta\in\Theta} -
\bar{c}_{-\theta}$ whereas its tangent cone at~$y$ is the image by
$\varpi^{\cF^\mathrm{opp}}_{\psi}$ of $\sum_{\theta\in\Theta} - \bar{c}_{\theta}$. This
opposite diamond $\diamant^\vee$ depends only on~$\diamant$ and not on the isomorphism~$\psi$.

\subsection{ Positive tuples}
\label{sec:pos-tuples}

When $z$~belongs to a diamond~$\diamant$ with extremities~$x$ and~$y$, the triple
$(x,z,y)$ of $\cF^{3}_{\Theta}$ will be called \emph{positive}. Positive triples form a
$\mG$-invariant and $S_3$-invariant subset of~$\cF^{3}_{\Theta}$. When $z$~belongs
to~$\diamant$ and~$w$ to~$\diamant^\vee$, the quadruple $(x,z,y,w)$ is called
\emph{positive}. Positive quadruples form a $\mG$-invariant subset of $\cF^{4}_{\Theta}$
that is invariant by the cyclic permutation $(x,z,y,w)\mapsto (z,y,w,x)$ as well as  the double transposition $(x,z,y,w)\mapsto (z,x,w,y)$.

Finally, for any $k$ greater than~$4$, positive
$k$-tuples are characterized 
in \cite[Section~2.4]{Guichard:2021aa} as those $(x_1, x_2, \dots, x_k)$ in $
\cF^{k}_{\Theta}$ such that  $(x_i, x_j, x_\ell, x_m)$ is a positive
quadruple for all $1\leq i<j<\ell<m\leq k$.

Let~$E$ be a set with a cyclic ordering. A map $f\colon E\to\cF_\Theta$ will be called
\emph{positive} if, for every $k\geq 3$ and for every cyclically ordered
$k$-tuple $(t_1, \dots, t_k)$ of~$E^k$, the $k$-tuple $(f(t_1), \dots, f(t_k))$ 
of $\cF^{k}_{\Theta}$ is positive. In view of the definition of positive $k$-tuples, it is
enough to check this property with $k=3$ or~$4$ and, in the case when $\sharp
E>3$, only with $k=4$.

Special positive maps arise from special three-dimensional simple subgroups of $\mG$. 
In \cite[Section~7]{GuichardWienhard_pos} it is shown that a positive structure on $\mG$ relative to $\Theta$ implies the existence 
of a $\Theta$-principal three-dimensional subalgebra. 
The associated equivariant map from the circle $S^1$ to the orbit of the
corresponding three-dimensional simple subgroup of $\mG$ in $\cF_\Theta$  is a
positive map. We call this (as well as its translates) a \emph{positive circle}. We
refer to \cite[Section~2.5]{Guichard:2021aa} for more details on positive circles. Here we will make use of the following properties, which can be deduced from the results in \cite[Section~2.5]{Guichard:2021aa}. 

\begin{lemma}
	\label{lem:positive-arc}
	Let~$x$ and~$y$ be two transverse elements of~$\mathcal{F}_\Theta$ and let~$D$ be a diamond with extremities~$x$ and~$y$.  There exist then
	\begin{itemize}
		\item a positive circle~\(C\) through~\(x\) and~\(y\) and such that \(D\cap C\) is an open interval (in~\(C\)) with extremities~\(x\) and~\(y\);
		\item  a smooth positive arc $c\colon [0,1]\to \cF_\Theta$
	such that $c(0)=x$, $c(1)=y$, and $c(t)$ is in $D$ for all $t$ in $(0,1)$.
	\end{itemize}  
\end{lemma}

Finally:
\begin{definition}\label{defi:semipositive}
	We say that a quadruple of points $(X,Y,x,y)$ is {\em semi-positive} if	$X$ and $Y$ are both transverse to~$x$ and~$y$, and moreover $(X,Y,x,y)$ is the limit of a sequence of positive quadruples.
\end{definition}

Observe that if $(X,Y,x,y)$ is semi-positive, then $(Y,X,y,x)$ is semi-positive as well.

We prove now that the photon projection of a positive
quadruple is a cyclically ordered quadruple (on the projective line) and that
photon projections give rise to semi-positive quadruples. Note here that we
can (and will) identify $\cF_{\Theta}^{\mathrm{opp}}$ with~$\cF_\Theta$ so
that the photon projection~$p_\Phi$ is indeed defined on the open subset
of~$\cF_\Theta$ of elements that are transverse to some point in~$\Phi$. With
this in mind:
\begin{proposition}\label{pro:semi-pos}
	Assume that $(X,Y,x,y)$ is a positive quadruple in~$\cF_\Theta$. Let~$\Phi$
	be a photon through~$X$, then
	\begin{enumerate}
		\item the configuration $(X,p_\Phi(Y),p_\Phi(x),p_\Phi(y))$ in the
		projective line~$\Phi$ is positive,
		\item the configurations $(p_\Phi(Y),Y,x,y)$ and $(Y,p_\Phi(Y),y,x)$
		in~$\cF_\Theta$ are semi-positive. 
	\end{enumerate} 
\end{proposition}

\begin{proof}
	Since $Y$, $x$, and $y$ are transverse to $X$ in $\Phi$, we indeed have that
	$Y$, $x$, and $y$ belong to~$O_\Phi$ so that we can consider the photon
	projections $p_\Phi(Y)$, $p_\Phi(x)$, and $p_\Phi(y)$.
	
	Corollary~\ref{coro:fibre-not-transverse}
	gives that $X,p_\Phi(Y),p_\Phi(x),p_\Phi(y)$ are pairwise distinct.  Given any positive
	configuration $(X,Y,x_0,y_0)$, we can find a deformation $(X,Y,x_t,y_t)$
	through positive configurations so that $(X,Y,x_1,y_1)$ lie on a positive
	circle (this follows from the connectedness properties of diamonds, cf.\ \cite[Lemma 3.7]{Guichard:2021aa}).  Equally, if $C$~is a positive circle though~$X$ and~$Y$, the restriction
	of~$p_\Phi$ to~$C$ is an injective continuous map to the photon $\Phi$. It is thus monotonous and sends
	positive configurations in~$\cF_\Theta$ to positive configurations in~$\Phi$ (with respect to the
	positive structure on the projective line~$\Phi$).  Hence for any positive $(X,Y,x,y)$,
	the  configuration $(X,p_\Phi(Y),p_\Phi(x),p_\Phi(y))$ is  positive with
	respect to~$\Phi$. This proves the first item.
	
	In particular, $p_\Phi(x)$ and $p_\Phi(y)$ both lie in the same connected
	component of $\Phi\smallsetminus\{X,p_\Phi(Y)\}$. Let $I$ be the other
	component of $\Phi\smallsetminus\{X,p_\Phi(Y)\}$, we now observe that for all
	$Z$ in $I$, $Z$ is distinct from~$p_\Phi(x)$, from~$p_\Phi(y)$, and
	from~$p_\Phi(Y)$, hence transverse (by the definition of $p_\Phi$) to~$x$,
	to~$y$, and to~$Y$. It follows by continuity and transversality that
	$(Z,Y,x,y)$ is positive. Letting $Z$~tend to $p_\Phi(Y)$, we get that
	$(p_\Phi(Y),Y,x,y)$ is semi-positive. Since double  transpositions preserve
	positivity  $(Y,p_\Phi(Y),y,x)$ is also semi-positive.
\end{proof}

As an important consequence of the previous proposition and of Proposition~\ref{pro:phot-cr}, we have that the photon cross-ratio of a positive
quadruple is positive.

\begin{proposition}\label{pro:phot>0}
	Let~$\eta$ be a $\Theta$-compatible dominant non-zero weight.
	Given a photon $\Phi$ through a point $x$, as well as $y$, $z$ and $w$ in
	$\cF_\Theta$ such that $(x,y,z,w)$ is positive, then
	$\bb^{\eta}(p_\Phi(w), x, y, z) > 1$.
\end{proposition}
\begin{proof}
	Indeed, 
	$$\bb^{\eta}(p_\Phi(w), x, y, z) = \bb^{\eta}_{\Phi}(p_\Phi(w), x,
	p_\Phi(y), p_\Theta(z)) = [p_\Phi(w), x,
	p_\Phi(y), p_\Theta(z)]^{\dual{h_\theta}{\eta}}>1$$ since the quadruple $(p_\Phi(w), x,
	p_\Phi(y), p_\Theta(z))$ in $\Phi\simeq \Rp$ is a positive configuration on
	the projective line so that its cross-ratio $[p_\Phi(w), x,
	p_\Phi(y), p_\Theta(z)]$ is~$>1$.
\end{proof}

\subsection{Positivity of bracket}
\label{sec:positivity-bracket}

We prove here Theorem \ref{theo:brak-pos} which  is an important step towards positivity of the cross-ratio.

Recall that we denote by $p\colon \mathfrak{g}_0\to \mathfrak{b}_\Theta$ the
orthogonal projection onto~$\mathfrak{b}_\Theta$ (its kernel is $\mathfrak{a}_\Theta \oplus \mathfrak{z}_{\mathfrak{k}}(\mathfrak{a}) \oplus
\bigoplus_{\alpha\in \Sigma} \mathfrak{g}_\alpha$).
Our goal is to prove the following result:

\begin{theorem}[{\sc Positivity of bracket}]\label{theo:brak-pos}
	Let $\eta$ be a $\Theta$-compatible dominant form. Let~$\theta$ be an
	element of~$\Theta$. Let~$u$ and~$v$ be respectively elements of the open
	cones~$c_\theta$ and~$c_{-\theta}$, then
	\begin{equation}
		\dual{ p([u,v])}{\eta}\geq 0\ ,\label{eq:non-negativity-of-bra}
	\end{equation}
	If furthermore, $\braket{\eta, \theta}>0$, then 
	\begin{equation}
		\dual{ p([u,v])}{\eta}> 0\ .\label{eq:positivity-of-bra}
	\end{equation}
\end{theorem}

We first begin by introducing and discussing boundary roots.

\subsubsection{Boundary roots}

In Section~\ref{sec:special_subsets} we introduced the subsets~$\Sigma_\theta$, for any 
$\theta$ in $\Sigma_{\Theta}^{+}$. Boundary roots are extremal elements
of~$\Sigma_\theta$:

\begin{definition}
	Let~$\theta$ be  in $\Theta$.
	A {\em boundary root with respect to~$\theta$} is a root~$\beta$ in
	$\Sigma_\theta$ such that there exists~$u$ in~$\mk a$ for which 
	\[
	\beta(u)>\alpha(u)\ ,
	\]
	for every~$\alpha$ in~$\Sigma_\theta \smallsetminus \{ \beta\}$.
\end{definition}
We denote
by~$B_\theta$ the set of boundary roots with respect to~$\theta$.

As, for all $\alpha,\beta $ in $ \Sigma_\theta$, the difference $\beta-\alpha$ is zero in restriction to~$\mathfrak{b}_\Theta$ and as
$\mathfrak{a}= \mathfrak{b}_\Theta \oplus \mathfrak{a}_\Theta$ (cf.\ 
Proposition~\ref{pro:ortho-atheta}), we can always assume that the element~$u$
in the definition belongs to~$\mathfrak{a}_\Theta$, the Cartan subspace
of~$\mS_\Theta$ (Section~\ref{sec:parab_sbgp}).

We first have:
\begin{proposition}\label{pro:theta-boundary}
	Every root $\theta$ in $\Theta$ is a boundary root with respect to $\theta$.
\end{proposition}
\begin{proof}
	Let~$v$ be in the opposite of the standard Weyl chamber. One has $\alpha(v)<0$ for every
	positive root~$\alpha$. For all~$\alpha$ in~$\Sigma_\theta\smallsetminus
	\{\theta\}$, $\alpha-\theta$ is a sum of simple roots; hence
	$\alpha(v)-\theta(v)<0$. Thus, $\theta(v)>\alpha(v)$, which is what we wanted to prove. 
\end{proof}

As $\Sigma_\theta$ is invariant by the Weyl group $W_{\mS_\Theta}$, the set of
boundary roots is invariant by~$W_{\mS_\Theta}$.

When $\Sigma_\theta=\{\theta\}$, the only boundary root is $\theta$. Thus, the definition of boundary roots is meaningful  when $\Sigma_\theta\neq
\{\theta\}$, namely when $\Theta \neq \Delta$ and $\theta$ is
a ``special root'' (i.e.\ connected to $\Delta\smallsetminus\Theta$ in the
Dynkin diagram, cf.\ \cite[Section~3.4]{GuichardWienhard_pos}). In this case the factor $\mS$ of~$\mS_\Theta$ that acts non-trivially
on~$\mk u_\theta$ is of type $A_d$ for some~$d$ and, if $\epsilon_0 - \epsilon_1, \dots, \epsilon_{d-1}- \epsilon_d$
are the simple roots in type~$A_d$ where the $\epsilon_i$ are weights summing to
zero (i.e.\ the $\epsilon_i$ are the weights of the standard
representation~$\mathsf{V}$ of~$\mathsf{S}$),  the weights of $\mS$ in~$\mk u_\theta$ are $2\epsilon_i$ ($i=0,\dots, d$, and
$\theta= 2\epsilon_0$) and
$\epsilon_i+\epsilon_j$ ($0\leq i<j\leq d$) (those are the weights
of~$\mathsf{S}$ acting on $\mathrm{Sym}^2 \mathsf{V}$; cf.\ \cite[Section
3.5]{GuichardWienhard_pos}). 
With this, we can now prove the following: 
\begin{proposition}\label{pro:boundary} Let~$\theta$ be
	in~$\Theta$. The set $B_\theta$ of boundary roots with respect to~$\theta$ is the $W_{\mS_\Theta}$-orbit of~$\theta$.
	In particular, 
	we have $\dim \mk g_\beta = 1$ for 	all~$\beta$ in~$B_\theta$. 
\end{proposition}
\begin{proof}
	For the proof we can restrict without loss of generality to the case when $\mk
	g_0$ is simple.
	The case when $\Theta=\Delta$ corresponds to the case when $\mathfrak{g}_0$ is
	split over~$\mathbb{R}$ and one has $\Sigma_\theta=\{ \theta\}$ and $\dim
	\mathfrak{g}_\theta =1$ so that the results follow immediately.
	
	Otherwise, the subsets $\Theta$ and $\Delta\smallsetminus\Theta$ of the set of
	simple roots are both non-empty and 
	connected and there is a unique root $\alpha_\Theta$ in $\Theta$ that is connected to $\Delta \smallsetminus \Theta$. 
	When $\theta \in \Theta \smallsetminus \{\alpha_\Theta\}$, we have, similarly
	to the split case, $\Sigma_\theta=\{\theta\}$ and the result is immediate.

	When $\theta  = \alpha_\Theta$, we will use the notation
	introduced before the proposition: the weights in $\mathfrak{u}_{\theta}$ are 
	$2\epsilon_i$ ($i=0,\dots, d$) and
	$\epsilon_i+\epsilon_j$ ($0\leq i<j\leq d$). The Weyl group acts here as the
	permutation group~$S_{d+1}$ and has therefore two orbits on the weights: the orbit of
	$2\epsilon_0$ and the orbit of~$\epsilon_0+\epsilon_1$. To conclude we
	examine which of these orbits are contained in~$B_\theta$.
	
	Choosing a
	vector~$u$ in the open Weyl chamber of~$\mS_\Theta$ (so that
	$(\epsilon_i-\epsilon_{i+1})(u)>0$ for all $i=0, \dots, d-1$) shows that $2\epsilon_0$
	is a boundary root (cf.\ also Proposition~\ref{pro:theta-boundary}).
	
	The weight $\epsilon_0+\epsilon_1$ does not correspond to  a boundary root since, for an
	element~$u$ in~$\mathfrak{a}$, the inequalities $(\epsilon_0+\epsilon_1)(u)>
	2\epsilon_0(u)$ and $(\epsilon_0+\epsilon_1)(u)>
	2\epsilon_1(u)$ cannot be simultaneously satisfied.
\end{proof}

The following proposition (notably item~(\ref{item:4pro:bdryroot-prop}))
explains the terminology boundary root.
\begin{proposition}
	\label{prop:boundary-roots-space-in-the-boundary}
	Let~$\theta$ be in~$\Theta$.
	Let $\beta$ be a boundary root with respect to $\theta$ and let 
	$t_\beta$ be in  $\mathbf{P}(\mk u_\theta)$ the element represented by the line~$\mathfrak{g}_\beta$.
	\begin{enumerate}
		\item\label{item:1pro:bdryroot-prop} The point $t_\beta$ is an attracting point for the action of an
		hyperbolic element in the Cartan subgroup $A$ of $\mS_\Theta$ on
		$\mathbf{P}(\mk u_\theta)$.
		\item\label{item:4pro:bdryroot-prop} The point $t_\beta$ belongs to the boundary of the set
		$\mathbf{P}(c_\theta)$.
	\end{enumerate}
\end{proposition}
\begin{proof}
	By equivariance under the Weyl group $W_{\mS_\Theta}$, it is enough to prove
	the statements for $\beta=\theta$.
	
	The tangent space at $t_\theta$ to $\bP(\mathfrak u_\theta)$ identifies
	$A$-equivariantly with
	\[
	\bigoplus_{\mathclap{\alpha\in\Sigma_\theta\smallsetminus{\theta}}} \mathfrak g_\theta^*\otimes \mathfrak g_\alpha\ .
	\] 
	Let $u$ in $\mk a_\Theta$ be as in the definition of boundary root. The
	eigenvalue of $\T_{t_\theta} \exp(u)$ on the factor $g_\theta^*\otimes \mathfrak g_\alpha$
	of the above decomposition is $\exp(\alpha(u)-\theta(u))$. 	These quantities being strictly smaller than~$1$ 
	implies the first item.
	
	The basin of attraction of~$u$ on $\bP(\mk u_\theta)$ is open and dense and thus intersects
	$\mathbf{P}(c_\theta)$. This implies that $t_\theta$~belongs to the closure of
	$\mathbf{P}(c_\theta)$.  It does not belong to~$\mathbf{P}(c_\theta)$ since otherwise the basin of attraction of~\(u\) would be contained in \(\mathbf{P}(c_\theta)\) contradicting the fact that \(c_\theta\)~is a proper convex cone.
\end{proof}

Recall that, for every root~$\beta$,
we fixed  an $\skd$-triple $(x_\beta, x_{-\beta}, h_\beta)$ with~$x_{\pm\beta}$
in $\mathfrak{g}_{\pm\beta}$. In view of Point~(\ref{item:4pro:bdryroot-prop})
of Proposition~\ref{prop:boundary-roots-space-in-the-boundary}, we can and will assume that the element~$x_\beta$
belongs to the closure of~$c_\theta$. With these choices, the following
proposition holds:

\begin{proposition}\label{pro:bdryroot-prop}
	Let~$\theta$ be in~$\Theta$.
	Let $\beta$ be a boundary root with respect to $\theta$ and let 
	$t_\beta$ be in  $\mathbf{P}(\mk u_\theta)$ the element represented by~$x_\beta$.
	\begin{enumerate}
		\item \label{item:2pro:bdryroot-prop} The group~$\mL_{\Theta}^{\circ}$ acts
		transitively on~$c_\theta$.
		\item\label{item:3pro:bdryroot-prop} The sum $\sum_{\beta\in B_\theta}
		x_\beta$ belongs to~$c_\theta$.
		\item\label{item:5pro:bdryroot-prop} The convex set~$\mathbf{P}(c_\theta)$ is contained in
		\[O_\theta = \bP\biggl\{\sum_{\alpha\in\Sigma_\theta}u_\alpha \mid \forall
		\alpha \in \Sigma_\theta, \, u_\alpha\in \mk g_\alpha \text{ and }
		\forall \beta\in B_\theta, \,
		u_\beta \in \mathbb{R}_{>0} x_\beta \biggr\}\ .\]
	\end{enumerate} 
\end{proposition}
\begin{proof}
	Point~(\ref{item:2pro:bdryroot-prop}) is
	\cite[Proposition~5.1]{GuichardWienhard_pos}. Point~(\ref{item:3pro:bdryroot-prop})
	is \cite[Theorem~5.12]{GuichardWienhard_pos}.
	
	Using that the action of~$\mL_{\Theta}^{\circ}$ is transitive on~$c_\theta$,
	that the stabilizers in~$\mL_{\Theta}^{\circ}$ of points in~$c_\theta$   contain a maximal compact
	subgroup (cf.{} \cite[Proposition~5.1]{GuichardWienhard_pos}) and using the
	Iwasawa decomposition in~$\mL_{\Theta}^{\circ}$, the proof of the last item follows from the statement and the proof of 
	\cite[Proposition~4.7]{benoist:2000ww} for $\beta=\theta$; by equivariance
	under $W_{\mS_\Theta}$ the property holds for every boundary root $\beta$.  
\end{proof}

\subsubsection{Proof of Theorem \ref{theo:brak-pos}}

Let $\eta$ be a $\Theta$-compatible dominant form and consider the map 
\begin{equation}
	\mappingnew{q\colon\mk u_\theta\times \mk u_{-\theta}}{\mathbb R}{(u,v)}{\dual{p([u,v])}{\eta}\ .}
\end{equation}
Observe that $q$ is $\Ad(\mL_{\Theta})$-invariant. Thanks to 
Proposition~\ref{pro:bdryroot-prop} it is thus enough to check the property
for~$u=\sum_{\beta\in B_\theta} x_\beta$ (where $B_\theta$ is the set of boundary roots) and any~$v$ in $c_{-\theta}$ that is
\begin{equation*} 
	v=\sum_{\alpha\in \Sigma_\theta} v_\alpha\ ,
\end{equation*} with
$v_\alpha$ in $\mk g_{-\alpha}$ for
every~$\alpha$ in~$\Sigma_\theta$, and $v_\beta = \mu_\beta x_{-\beta}$ with $\mu_\beta>0$  for
every~$\beta$ in~$B_\theta$. Using the
decomposition $\mathfrak{g}_0= \mathfrak{a}\oplus
\mathfrak{z}_{\mathfrak{k}}(\mathfrak{a}) \oplus \bigoplus
\mathfrak{g}_\alpha$, 
it follows that the projection of $[u,v]$ on~$\mathfrak{a}$  is equal to 
\[
\sum_{\beta\in B_\theta} \mu_\beta h_\beta\ .
\]
Hence, as $\eta$ is zero on~$\mathfrak{a}_\Theta$ and since $p([u,v])$ differs
from the above element by an element in~$\mathfrak{a}_\Theta$, one has
\[ 
\dual{ p([u,v])}{\eta} = \sum_{\beta\in B_\theta}\mu_\beta \dual{h_\beta}{\eta}=2 \sum_{\beta\in B_\theta}\mu_\beta 
\frac{\braket{\eta,\beta}}{\braket{\beta,\beta}} \ .
\]
Thus, from the definition of a dominant
form, we have 
\[
\dual{ p([u,v])}{\eta}\geq 2 \mu_\theta
\frac{\braket{\eta,\theta}}{\braket{\theta,\theta}}\ .
\]
From this last inequality, the lower bounds in
Equations~\eqref{eq:non-negativity-of-bra} and~\eqref{eq:positivity-of-bra} of
Theorem~\ref{theo:brak-pos} follow.
\qedhere

\section{Positivity of the cross-ratio}
\label{sec:posit-cross-ratio}

We continue with the setup of the previous section: $\mG_0$ is a semisimple
Lie group admitting a positive structure relative to~$\Theta$.
The main result is the following:

\begin{theorem}[{\sc Positivity of the cross-ratio}]\label{theo:cr-pos}
	Let $\eta$ be a  $\Theta$-compatible dominant non-zero form. Let $\bb^\eta$ be the
	cross-ratio associated to~$\eta$ (cf.\ Section~\ref{sec:cross-ratio} and
	more particularly Section~\ref{sec:sympl-reint}). Then any positive
	quadruple $(x,y,z,w)$ in $\cF_\Theta$ belongs to \(O^\square(\mathcal{F}_\Theta)\) and we have 
	\[
	\bb^{\eta}(x,y,z,w)>1\ . 
	\]
\end{theorem}

The terminology ``positivity of the cross-ratio'' becomes justified after one takes the logarithm.

The proof of Theorem \ref{theo:cr-pos} relies on the integral formula for the
cross-ratio given in Section~\ref{sec:sympl-reint}.

We state a useful corollary to Theorem~\ref{theo:cr-pos}:

\begin{corollary}\label{coro:pos-line-fund}
	Let $\eta$ be a $\Theta$-compatible dominant form, $\omega_\theta$ a
	fundamental weight and $(x,y,z,x)$ a positive quadruple. Then
	\begin{align*}
		\bb^\eta(x,y,z,w)
		&\geq \left( \bb^{\omega_\theta}(x,y,z,w) \right)^{\braket{h_\theta\mid
				\eta}}\ .\\
		\intertext{In particular, for all~$\gamma$ in~$\mG$}
		\pp^\eta(\gamma)
		&\geq \left(\pp^{\omega_\theta}(\gamma)\right)^{\braket{h_\theta\mid
				\eta}}\ .
	\end{align*}
\end{corollary}
\begin{proof}
	Indeed, we can write 
	\[
	\eta=\eta_0 + \braket{h_\theta\mid \eta}\omega_\theta\ ,
	\]
	where $\eta_0$ is a $\Theta$-compatible dominant form. It then follows by
	Assertion~\eqref{e.multiplicativity} (p.~\pageref{e.multiplicativity}) that
	\[
	\bb^\eta=\left(\bb^{\omega_\theta}\right)^{\braket{h_\theta\mid
			\eta}}\bb^{\eta_0}\ ,
	\] 
	and the statement follows from Theorem~\ref{theo:cr-pos} applied to~\(\eta_0\).
\end{proof}

\subsection{Infinitesimal positivity}
Denote as always  $\cL_\Theta=\cG/\mL_{\Theta}$.

Let $x$ and $y$ be transverse points in $\cF_\Theta$.
Let $\diamant$ be a diamond with extremities~$x$ and~$y$. Recall from
Section~\ref{sec:diamonds-positive-n} that a tangent vector at~$x$ is
non-negative 
with respect to $\diamant=\psi(\mN_\Theta)\cdot x$ if it belongs to the image
by $\pi^{\cL}_{\psi}$ of the closed cone $\sum_{\theta\in\Theta}
\bar{c}_{-\theta}$ inside $\mathfrak{u}^{\mathrm{opp}}_{\Theta}$. We have:
\begin{proposition}\label{pro:infi-pos}
	Let $\eta$ be a $\Theta$-compatible dominant form and let $v$ be a non-negative tangent vector
	at~$x$ (with respect to~$\diamant$) and $w$ be a non-negative tangent vector
	at~$y$ (with respect to the opposite diamond~$\diamant^\vee$). Then 
	\begin{equation*}
		\bigdual{ \Omega_{(x,y)}((v,0),(0,w)) }{\eta}\geq  0\ .
	\end{equation*}
	If
	furthermore $\eta$~is non-zero and $v$ and $w$ are positive tangent vectors,
	then 
	\begin{equation*}
		\bigdual{ \Omega_{(x,y)}((v,0),(0,w)) }{\eta}> 0\ .
	\end{equation*}
\end{proposition}

\begin{proof}
	Recall the decomposition
	\[\mk g_0=\mk l_\Theta\oplus \mk u_\Theta \oplus \mk u^{\mathrm{opp}}_{\Theta}\ ,
	\]
	and $\pi^\cL$ the projection from~$\cG$ to~$\cL_\Theta$. Let~$\psi$ be
	in~$\cG$ such that $\pi^\cL(\psi) = (x,y)$ and $\diamant = \psi(
	\mN_\Theta)\cdot x$. 
	We have an identification $\iota_{\psi}^{\cL}$ (see Section~\ref{sec:parab-subgr-flag})  of $\T_{(x,y)}\cL_\Theta$  with 
	\[ \mk u^{\mathrm{opp}}_{\Theta} \oplus \mk u_\Theta.\] 
	By definition 
	\[\Omega((v,0),(0,w))=p\left([\iota^{\cL}_{\psi}((v,0)),\iota^{\cL}_{\psi}((0,w))]\right)\ ,
	\]
	where $p$ is the orthogonal projection from~$\mk g_0$ to~$\mk b_\Theta$. 
	Hence, the
	proposition reduces to Theorem~\ref{theo:brak-pos} using that $\iota^{\cL}_{\psi}((v,0))$ is a vector in $\sum_{\theta\in\Theta}
	\bar{c}_{-\theta}$, that $\iota^{\cL}_{\psi}((0,w))$ is a vector in
	$\sum_{\theta\in\Theta} -
	\bar{c}_{\theta}$ (Section~\ref{sec:diamonds-positive-n}), and that the Lie
	bracket is antisymmetric.
\end{proof}

\subsection{Proof of Theorem \ref{theo:cr-pos}}\label{sec:proof-theo:cr-pos}

We are in a setting where $\iota(\Theta) =\Theta$ so that
$\cF^{\mathrm{opp}}_{\Theta} \simeq \cF_{\Theta}$ and $\cL_\Theta$ is the open $\mG$-orbit in $\cF_\Theta\times \cF_\Theta$.

We begin the proof of Theorem~\ref{theo:cr-pos} by showing that positive quadruples belong to \(O^\square(\mathcal{F}_\Theta)\):
\begin{proposition}\label{p.squaresforpositive}
	Let $(x,y,z,w)$ be a positive quadruple, then there
	exist $\mathcal{C}^1$ arcs $c_0\colon [0,1]\to \cF_\Theta$ and $c_1\colon [0,1]\to \cF_\Theta$
	such that $c_0(0)=x$, $c_0(1)=y$, $c_1(0)=z$, and $c_1(1)=w$ and such that,
	for all~$s$ in $(0,1)$ and all~$t$ in $(0,1)$, the sextuple $(x, c_0(s), y,
	z, c_1(s), w)$ is positive.
	
	For every such arcs~$c_0$ and~$c_1$ the map
	$f\colon [0,1]^2\to \cF_\Theta\times \cF_\Theta$ defined by $f(s,t)=(c_0(s), c_1(t))$
	takes value in~$\cL_\Theta$ and one has, for every~$t$ in~$[0,1]$,
	$f(0,t) = (x, *)$ and $f(1,t)=(y,*)$, and for every~$s$ in~$[0,1]$,
	$f(s,0)=(*,z)$ and $f(s,1)=(*,w)$.
\end{proposition}
\begin{figure}[hbt] 
	\begin{center}
		\includegraphics[height=4.5cm]{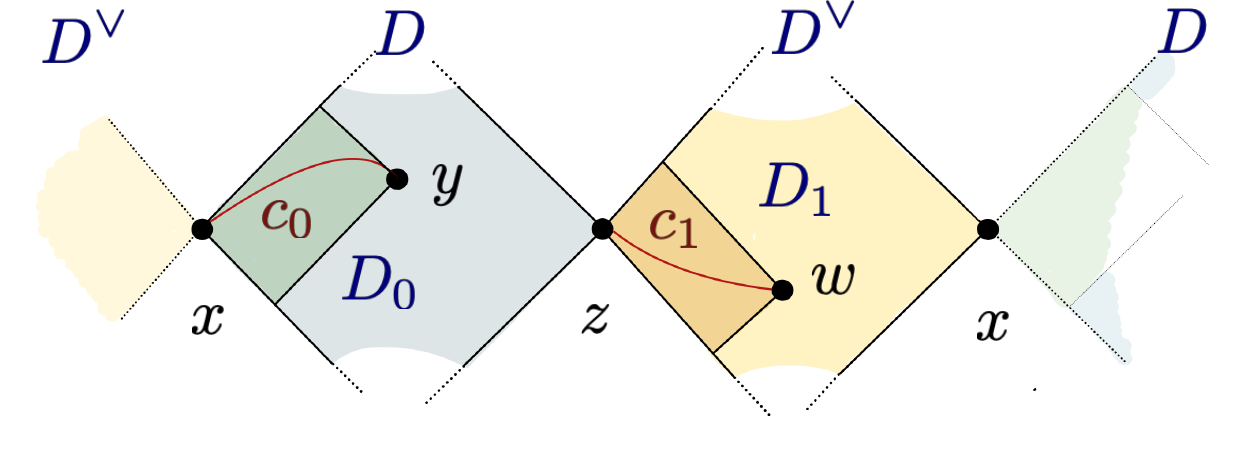}
	\end{center}
	\caption{Configuration of the positive quadruple $(x,y,z,w)$, the $\mathcal{C}^1$ arcs $c_0$, $c_1$ and the diamonds $\diamant,\diamant^\vee, \diamant_0, \diamant_1$ in a $\mathbb Z$-covering of an annulus} 
	\label{fig.diamonds}
\end{figure}

\begin{proof}
	Let~$\diamant$ be the diamond with extremities~$x$ and~$z$ containing~$y$;
	by positivity of the quadruple $(x,y,z,w)$ 
	the opposite diamond $\diamant^\vee$ contains~$w$. There exists a
	unique diamond~$\diamant_0$ contained in~$\diamant$ and with extremities~$x$
	and~$y$ and there exists a unique diamond~$\diamant_1$ contained
	in~$\diamant^\vee$ and with extremities~$z$ and~$w$, see Figure \ref{fig.diamonds}. 
	
	We can now choose two arcs of positive circles: $c_0$~joining~$x$ to~$y$ and
	contained in~$\diamant_0$, $c_1$~joining~$z$ to~$w$ and contained
	in~$\diamant_1$ (Lemma~\ref{lem:positive-arc}). 
	The inclusions of diamonds give that, for all $s$ and $t$, the sextuple $(x, c_0(s), y,
	z, c_1(s), w)$ is positive.
	
	Given  such arcs~$c_0$ and~$c_1$, 
	by positivity, for all $s$ and $t$, $c_0(s)$ is transverse to $c_1(t)$. The map  $f$  given by
	\[(s,t)\mapsto (c_0(s),c_1(t))\ 
	\]
	takes thus value in~$\cL_\Theta$ and has the wanted properties.
\end{proof}
We can now conclude the proof of Theorem~\ref{theo:cr-pos} using $c_0$, $c_1$,
and~$f$ as in the previous proposition. 
By Proposition~\ref{pro:cr-om},
\[\bb^{\eta}(x,y,z,w)=\exp\Bigl(\int_{[0,1]^2} f^*(\dual{\Omega}{\eta})\Bigr)\ .
\]
By definition if $(u,v)$ belongs to $[0,1]^2$,
\[f^*(\dual{\Omega}{\eta})_{(u,v)}=\bigdual{\Omega\bigl(
	\dot c_0(u),\dot c_1(v)\bigr)}{\eta}\cdot\d s\wedge \d t\ .
\]
By Proposition~\ref{pro:infi-pos}, we have
\[ \bigdual{\Omega\bigl( \dot c_0(u),\dot c_1(v)\bigr)}{\eta}>0\ .
\] 
The result now follows.\qed

\section{The photon cross-ratio bounds the $\theta$-character}
\label{sec:simple-roots-sup}
In this section we relate the photon cross-ratios to characters of simple roots. In particular, we prove the following: 
\begin{theorem}\label{theo:sup-min}
	Let $\theta$ be an element of $\Theta$, $\eta$ a $\Theta$-compatible dominant
	form 
	such
	that $\braket{\eta,\theta}>0$, and $\gamma$ in $\mG$ be a $\Theta$-loxodromic element with attracting and repelling fixed points $\gamma^+$, $\gamma^-$, let
	$x$ in $\cF_\Theta$ be 
	such that 
	$(\gamma^+, \gamma^-,x,\gamma \cdot x)$ is a positive quadruple   then
	\begin{equation*}
		\chi_\theta(\gamma)^{\dual{ h_\theta}{\eta}} \geq
		\min_{\Phi\in \bPhi(\gamma^-)}
		\bb^\eta(p_\Phi(\gamma^+),\gamma^-, x, \gamma\cdot x)\ ,
	\end{equation*}
	where $\bPhi(\gamma^-)$ is the family of $\theta$-photons through~$\gamma^-$.  
\end{theorem}
Note that the result makes sense since semi-positive quadruples belong to \(O^\square(\mathcal{F}_\Theta)\) (cf.\ Proposition~\ref{p.squaresforpositive}).

Also, by the homogeneity properties of the cross-ratios we can restrict to the case where \(\eta=\omega_\theta\) so that we can assume that the cross-ratio \(\bb^\eta\) is defined for every quadruple \((x,y,z,w)\) with \(x\pitchfork w\) and \(y\pitchfork z\).

We first state and prove two preliminary results. Let~$a$ and~$b$ be two transverse points in $\cF_\Theta$
and $\mL\defeq \mL_{a,b}$ be the stabilizer in~$\mG$ of the pair
$(a,b)$.

\begin{proposition}[\sc The compact case]\label{pro:ssimple} Let $\mM$ be a
	compact subgroup of $\mL$. Assume that $k$ belongs to  $\mM$ and that $x$~is
	transverse to both~$a$ and~$b$, then for any $\mM$-invariant compact subset~$M_0$  of  $\bPhi(a)$ \[
	\min_{\Phi\in M_0}
		\bb^\eta(p_\Phi(b),a,x, k(x))
	\, \leq 1 \ .
	\]	
\end{proposition}
\begin{proof} Let $\cS$ be the  $\mM$-orbit of~$x$ in~$\cF_\Theta$. 
	All~$z$ in~$\cS$ are transverse to $a$ and $b$, and hence to $p_\Phi(b)$ for all $\Phi$ in $\bPhi(a)$ by  Corollaries \ref{coro:photon-proj-trans} and \ref{coro:fibre-not-transverse}. Thus, the function
	\[
	\Psi\colon (z,y,\Phi)\longmapsto  \bb^\eta(p_\Phi(b), a,z,y)
	\]
	on $\cS\times \cS\times M_0$
	is 
	defined and continuous. 
	We consider the function on $\cS^2$ 
	\begin{equation*}
		G(z,y)=\min_{\Phi\in M_0}
		\left\vert\bb^\eta(p_\Phi(b),a, z,y)\right\vert\ ,
	\end{equation*}
	which is
	continuous by the continuity of~$\Psi$ and the compactness
	of $M_0$. 
	As a consequence of the cocycle identity we
	have
	\begin{equation}
		G(z,y) \geq G(z,w) G(w,y)\ .\label{ineq:cocyc-ineq}
	\end{equation}
	Since $M_0$ is $\mM$-invariant, for every $g$ in $\mM$ we have
	\begin{equation*}
		G(g(z),g(y))=G(z,y)\ .
	\end{equation*}
	By the compactness of $\cS$ there is a constant $A$ such that for all $z$ and
	$y$ in $\cS$,
	\begin{equation*}
		G(z,y)\leq A\ .
	\end{equation*} 
	For any $z$ and $y$ in $\cS$, let $g$ in $\mM$ be such that  $y=g(z)$, we obtain
	by iterating the cocycle inequality~\eqref{ineq:cocyc-ineq}  and using the
	$\mM$-invariance of~$G$, that for all~$n$
	\[
	A\geq G(z,g^n(z))\geq G(z,g(z))^n=G(z,y)^n\ .
	\]
	This shows that  $G(z,y)\leq 1$ for all $y$ and $z$ in the  $\mM$-orbit of
	$x$. Hence, $G(x, k (x))$ is at most $1$ 
	and this 
	concludes the proof.
\end{proof}
In the next proposition, we use the Kostant--Jordan decomposition recalled in the beginning of Section~\ref{sec:loxodromic-elements}.
\begin{proposition}[\sc A photon is preserved]
	\label{pro:phot-preserv}
	Let $g$ be a $\Theta$-loxodromic element in~$\mG$ such that~$a$ and~$b$
	are respectively the repelling and attracting fixed points of~$g$. 
	Let $g=g_hg_ug_e$ be the Kostant--Jordan decomposition of~$g$ (in~$\mG$).
	Then there exists a $\theta$-photon~$\Phi$ in $\mathbf{\Phi}(a)$ invariant
	by~$g_h$ and by~$g_u$ and
	\begin{equation}\label{e.Photonpreserved}
		\chi_\theta(g)^{\dual{ h_\theta}{\eta}} =
		\bb^\eta(p_{\Phi}(b),a,y, g_hg_u(y))\ ,
	\end{equation}
	for all $y$ transverse to~$a$ and to~$b$.
\end{proposition}
\begin{proof}  Let $\psi$ be an isomorphism of $\mG_0$ with $\mG$ such that
	$\pi^\cL(\psi)=(a,b)$ and  $g_h=\psi(\exp(X))$ with
	$X$~in the closed Weyl chamber~$\mathfrak{a}^+$ ---cf.\ the Kostant--Jordan
	decomposition in Section~\ref{sec:loxodromic-elements}. One also has  $\psi(\mL_\Theta)=\mL_{a,b}$.

	As $g$~is $\Theta$-loxodromic (see Section~\ref{sec:loxodromic-elements}),
	the element~$X$ satisfies that $\dual{X}{\alpha}>0$ for all~$\alpha$
	in~$\Theta$ and $\dual{X}{\alpha}\geq 0$ for all~$\alpha$ in
	$\Delta\smallsetminus \Theta$.
	
	Let 
	$E=\{ v \in \mathfrak{u}_{-\theta} \mid \ad(X)v = -\dual{X}{\theta}v\}$, and
	let~$Z_\theta$ the $\mL_\Theta$-orbit of~$x_{-\theta}$ in~$\mathfrak{u}_{-\theta}$ (Section~\ref{sec:theta_like_vect}).  We
	know that the image by $\pi_{\psi}^{\cF}$ of  every vector~$v$  in~$Z_\theta$ is tangent to a photon through~$a$ (Proposition~\ref{pro:PhoUni}). If
	furthermore this vector $v$ is in~$E$, the $\eta$-period of~$g_h$ on this photon satisfies the stated equality thanks to
	Proposition~\ref{pro:phot-cr}.
	
	The proposition will be proved if we can find a vector in~$E\cap Z_\theta$ that is
	also invariant by~$g_u$ (in which case the action of~$g_u$ on the corresponding
	photon will be trivial). Note that the space~$E$ is $\Ad(g_u)$-invariant since
	$E$~is the intersection of~$\mathfrak{u}_\theta$ with
	$\ker(\ad(X)-\dual{X}{\theta} \mathrm{Id})$ and both these spaces are
	$\Ad(g_u)$-invariant. The projectivization $\mathbf{P}({Z_\theta\cap E})$ of~$Z_\theta$
	in~$\mathbf{P}(E)$ is a closed $\Ad(g_u)$-invariant subset (cf.\ Lemma~\ref{lemma:theta-light-closed}). Since $\Ad(g_u)$ is
	unipotent, every $\langle\Ad(g_u)\rangle$-orbit in $\mathbf{P}(E)$ accumulates
	to a point  fixed by $\Ad(g_u)$. These last two remarks imply that
	$\mathbf{P}(Z_\theta\cap E)$ contains points fixed by~$\Ad(g_u)$. This finishes the proof. 
\end{proof}

\begin{proof}[Proof of Theorem~\ref{theo:sup-min}] 
	
	We can now prove the inequality of Theorem \ref{theo:sup-min}.
	
	Let us write $\gamma=\gamma_0\gamma_e$ with $\gamma_0=\gamma_h\gamma_u$,
	where $\gamma_h$, $\gamma_u$, and $\gamma_e$ are pairwise commuting and
	respectively the hyperbolic, unipotent, and elliptic parts of~$\gamma$. 
	Let then $\mM$ be the closure of the group generated by $\gamma_e$ and $M_0$ be the compact  set of photons~$\Phi$  preserved by $\gamma_0$ in $\bPhi(\gamma^-)$ and satisfying Equation~\eqref{e.Photonpreserved}, namely such that
	\begin{equation*}
		\chi_\theta(\gamma)^{\dual{h_\theta}{\eta}} =
		\bb^\eta(p_{\Phi}(\gamma^+),\gamma^-,y, \gamma_0\gamma_u(y))\ ,
	\end{equation*}
	for all $y$ transverse to $\gamma^-$ and $\gamma^+$.
	We observe that $M_0$ is invariant by $\mM$, and non-empty by
	Proposition~\ref{pro:phot-preserv} (applied with
	$a=\gamma^-$ and $b=\gamma^+$). Let finally~$\Phi_0$ be a photon in~$M_0$ such that 
	$$
	\bb^\eta(p_{\Phi_0}({\gamma}^+),{\gamma}^-,x, {\gamma}_e(x))=\min_{\Phi\in M_0}\bb^\eta(p_{\Phi}({\gamma}^+),{\gamma}^-,x, {\gamma}_e(x))\leq 1\ ,
	$$
	where the inequality comes from Proposition~\ref{pro:ssimple}.  
	We then have  by the cocycle identities
	\begin{align*}
		\bb^\eta(p_{\Phi_0}({\gamma}^+),{\gamma}^-,
		&\ x,{\gamma}(x))\\
		&= \bb^\eta(p_{\Phi_0}({\gamma}^+),{\gamma}^-,x, {\gamma}_e(x))\ 
		\bb^\eta(p_{\Phi_0}({\gamma}^+),{\gamma}^-,\gamma_e(x),
		\gamma_0{\gamma}_e(x))\\
		&\leq \bb^\eta(p_{\Phi_0}(\gamma^+),{\gamma}^-,\gamma_e(x),
		\gamma_0(\gamma_e(x)))\\
		&= \chi_\theta(\gamma_0)^{\dual{ h_\theta}{\eta}}=\chi_\theta(\gamma)^{\dual{ h_\theta}{\eta}}
		\ .
	\end{align*}
	It follows that
	\begin{equation*}
		\min_{\Phi\in\bPhi(\gamma^-)} \bb^\eta(p_{\Phi}({\gamma}^+),{\gamma}^-,
		x,{\gamma}(x)) \leq \chi_\theta(\gamma)^{\dual{ h_\theta}{\eta}} \ ,
	\end{equation*}
	and the result follows.
\end{proof}

\section{The collar inequality}
\label{sec:collar-inequality}

Our goal in this section is to prove the main result of this paper, it
generalizes Theorem \ref{theo:Bprime} in the sense that general
$\Theta$-compatible dominant forms are allowed.  As above, we let $\mG_0$ be a semisimple Lie group admitting a positive structure relative to~\(\Theta\),  the group \(\mG\) is isomorphic to~\(\mG_0\). 

\begin{theorem}[\sc Collar Lemma in the group]\label{thm:collarineq}
	Let~$A$ and~$B$ be $\Theta$-loxodromic elements of~$\mG$. Denote
	by $(a^+, a^-)$ and $(b^+, b^-)$ the pair of attracting and repelling fixed
	points of~$A$ and~$B$ respectively in the flag variety~$\cF_\Theta$. Assume
	that the sextuple
	$$(a^+, b^-, a^-, b^+, B(a^+), A(b^+))\ ,$$
	is positive (see Figure~\ref{fig:sext}).  Let~$\theta$ be an element
	of~$\Theta$, and $\eta$ be a $\Theta$-compatible dominant form with
	$\dual{h_\theta}{\eta}>0$. Then
	\begin{equation}
		\left(
		\frac{1}{ \pp^{\eta}\left(B\right)  }
		\right)^{{1}/{\dual{ h_\theta}{\eta}}}
		+ \frac{1}{\chi_\theta(A)}< 1\ .\label{ineq:collarFOL}
	\end{equation}	
\end{theorem}
Observe that when $\dual{ h_\theta}{\eta}=0$, the above
inequality is still true but of little use.
\begin{figure}[hbt]
	\begin{center}
		\includegraphics[height=4cm]{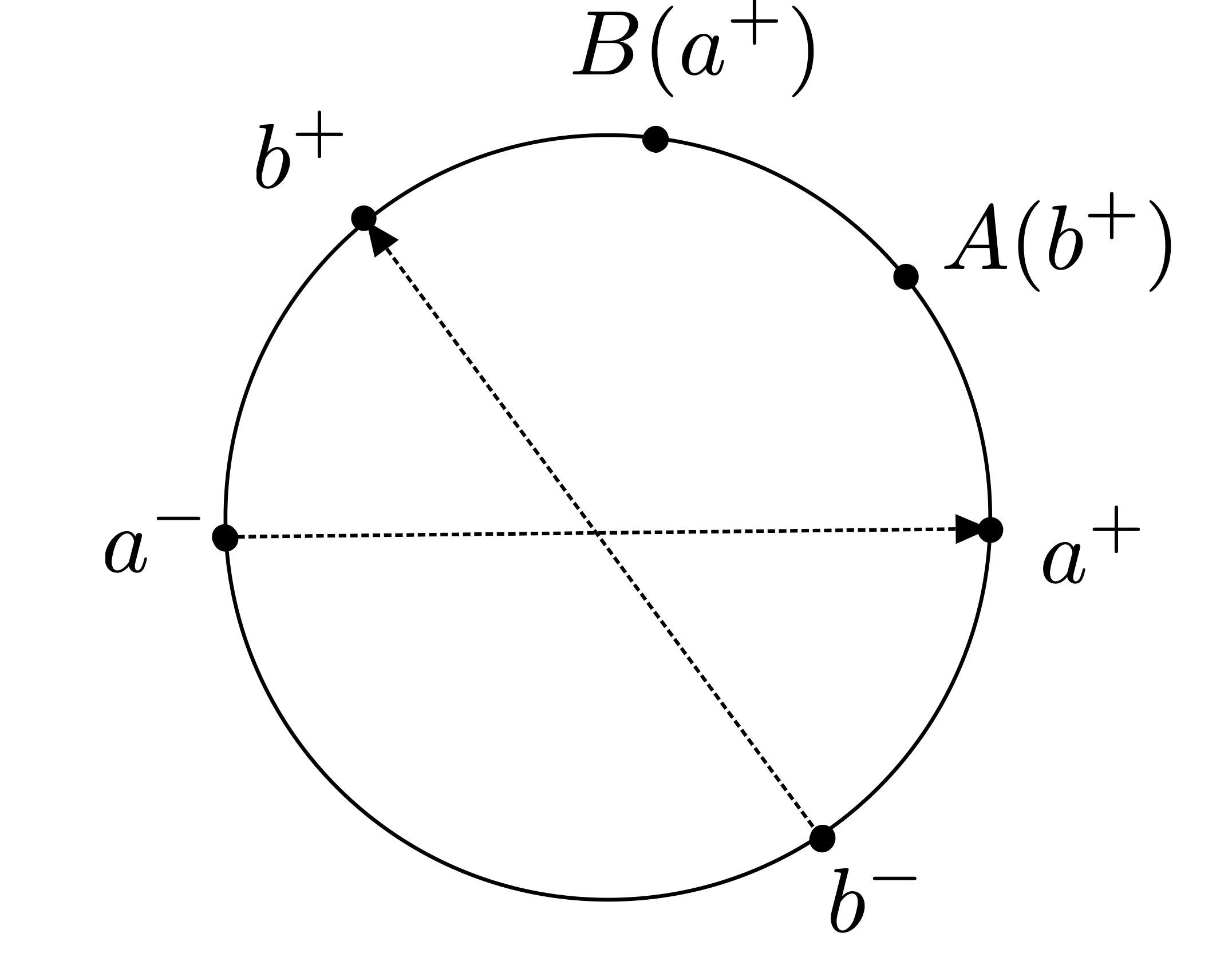}
	\end{center}
	\caption{Positive sextuple} 
	\label{fig:sext}
\end{figure}

\begin{proof}\label{rk:postive-linea/weight} From Corollary \ref{coro:pos-line-fund} it is enough to prove the inequality whenever $\eta$ is the fundamental weight $\omega_\theta$. 
	
	Let~$\Phi$ be a $\theta$-photon
	through~$a^-$. From Proposition~\ref{pro:phot-cr} and the
	classical relation for the projective cross-ratio, we have 	\begin{equation*}
		\bb^{\omega_\theta}(a^-,p_\Phi(b^+),a^+, A(b^+)) +
		\bb^{\omega_\theta}(a^-,p_\Phi(a^+),b^+,A(b^+)) = 1\ .
	\end{equation*}
	We will now obtain a minoration of the first term in the left-hand side of
	this equation, we will then apply Theorem~\ref{theo:sup-min} in order to
	obtain the wanted majoration.
	In these computations, we will use freely that the cross-ratio  is greater
	than~$1$ for positive quadruples (Theorem \ref{theo:cr-pos}).
	\vskip 0.2truecm 
	\noindent{\sc First step:}
	We first bound from below the first term of the previous equality. 	\begin{equation*}
		\bb^{\omega_\theta}(a^-,p_\Phi(b^+),a^+, A(b^+))>\pp^{\omega_\theta}(B)^{-1}\ .
	\end{equation*}
	Let $L\defeq \bb^{\omega_\theta}(a^-,p_\Phi(b^+),a^+, A(b^+))$.
	By the cocycle relation we have 
	\[
	L = \bb^{\omega_\theta}(a^-,b^+,a^+,A(b^+)) \cdot
	\bb^{\omega_\theta}(b^+,p_\Phi(b^+),a^+,A(b^+))\ .
	\]
	
	The quadruple $(a^-,b^+,A(b^+),a^+)$ is positive, 
	hence by
	Proposition~\ref{pro:semi-pos}, the quadruple
	$(b^+, p_\Phi(b^+),a^+,A(b^+))$ is  also
	a semi-positive quadruple, thus Theorem~\ref{theo:cr-pos} gives        
	\[
	\bb^{\omega_\theta}(b^+,p_\Phi(b^+),a^+,A(b^+))\geq  1\ .
	\]
	Then
	\begin{align*}
		L & \geq  \bb^{\omega_\theta}(a^-,b^+,a^+,A(b^+))\\
		& =\bb^{\omega_\theta}(a^-,b^+,a^+,B(a^+))\cdot
		\bb^{\omega_\theta}(a^-,b^+, B(a^+),A(b^+))\\
		& =\bb^{\omega_\theta}(a^-,b^-,a^+,B(a^+)) \cdot \bb^{\omega_\theta}(b^-,b^+,a^+,B(a^+)) \cdot
		\bb^{\omega_\theta}(a^-,b^+, B(a^+),A(b^+))\ ,	
	\end{align*}
	where we used the cocycle identities twice. As the quadruples $(a^-,b^+, B(a^+),A(b^+))$ and $(a^-,b^-,a^+,B(a^+))$ are positive (the latter follows as $(b^-,a^-,B(a^+),a^+)$ is positive), their cross-ratios are
	greater than~$1$ and we get 
	\[
	L> \bb^{\omega_\theta}(b^-,b^+,a^+,B(a^+)) = \pp^{\omega_\theta}(B)^{-1}\ ,
	\]
	which is what {we wanted to prove}.
	
	\vskip 0.2truecm 
	\noindent{\sc Second step:}
	We obtain from the first step that, for
	every photon~$\Phi$ through~$a^-$,
	\[
	\pp^{\omega_\theta}\left(B\right)^{-1}+
	\bb^{\omega_\theta}(a^-,p_\Phi(a^+),b^+,A(b^+))< 1\ .
	\]
	Letting~$\Phi$ vary in $\bPhi(a^-)$, we get
	\begin{equation*}
		\pp^{\omega_\theta}\left(B\right)^{-1} +
		\max_{\Phi\in
			\bPhi(a^-)}
		\left(\bb^{\omega_\theta}(a^-,p_\Phi(a^+),b^+,A(b^+))\right) <
		1\ .
	\end{equation*}
	Observe now that $(a^-,a^+,A(b^+),b^+)$ is a positive quadruple, thus
	Theorem~\ref{theo:sup-min} gives in particular that
	\begin{align*}
		\chi_\theta(A)^{-1}
		&\leq \bigl(\min_{\Phi\in \bPhi(a^-)}
		\bb^{\omega_\theta}(a^-,p_\Phi(a^+),A(b^+),b^+)\bigr)^{-1}\\
		&=  \max_{\Phi\in \bPhi(a^-)}
		\bb^{\omega_\theta}(a^-,p_\Phi(a^+),b^+,A(b^+)) \ .	
	\end{align*}
	Thus combining the two last inequalities, we get	
	\begin{equation*}
		\pp^{\omega_\theta}\left(B\right)^{-1} +
		\left(\chi_{\theta}(A)\right)^{-1}< 1\ .
	\end{equation*}
	
	This is the inequality that we wanted to prove.
\end{proof}

\section[Positive representations]{Positive representations of finite type and infinite type
	surfaces}
\label{sec:posit-repr-infin}
In this section, we give the definition of positive representations in a
setting that allows surfaces that are not closed, or not even of finite
topological type, i.e.\ we do not assume that the fundamental group is
finitely generated.

Let $\Sigma$ be a ---possibly non-compact--- connected oriented surface whose fundamental group $\Gamma$~contains a free group. 
Among loops not homotopic to zero, we distinguish between  {\em peripheral loops}  and {\em non-peripheral loops}
in $\Sigma$: peripheral loops are curves in $\Sigma$ which are freely
homotopic to a multiple of a boundary component or a cusp, otherwise a loop is  non-peripheral. We  use the same  terminology for conjugacy classes of elements of $\pi_1(\Sigma)$, seen as free homotopy classes of loops. 

We denote by $\Lambda$ the classes of non-peripheral
elements of $\pi_1(\Sigma)$ up to positive powers; i.e.\ $\gamma$
and~$\delta$ represent the same element in~$\Lambda$ if and only if there are
positive integers~$n$ and~$m$ such that $\gamma^n=\delta^{m}$. The
class in~$\Lambda$ of a non-peripheral element~$\gamma$ will be denoted
by~$\gamma^+$.  Similarly, we will denote~\(\gamma^-\) the class of~\(\gamma^{-1}\).
The set~$\Lambda$ should be thought of as the set of attracting fixed
points of non-peripheral elements of $\pi_1(\Sigma)$ in the boundary at
infinity of the group. The conjugation induces a natural action of~$\pi_1(\Sigma)$
on~$\Lambda$. 
We will introduce a cyclic order on~$\Lambda$.
Since $\grf$ might not be finitely generated, we cannot directly use the boundary at infinity. Instead, we use the following trick. 

\begin{proposition}[\sc Reduction to finite type]\label{pro:Red-finite}
	Given finitely many elements $\gamma_1,\ldots,\gamma_p$ in $\pi_1(\Sigma)$,
	there exists an incompressible connected surface~$S$ of finite type,
	in~$\Sigma$, whose fundamental group contains all the $\gamma_i$. If
	furthermore, none of the $\gamma_i$ are peripheral, we can choose $S$ such
	that all the~$\gamma_i$ remain non-peripheral in~$S$.
\end{proposition}

In the situation of the proposition, we say that $S$~\emph{encloses}
$(\gamma_1,\ldots,\gamma_p)$. 
Similarly, given finitely many elements $t_1, \dots, t_n$ in~$\Lambda$, we
say that  an
incompressible  connected surface~$S$  of finite type \emph{encloses} them
if there are $\gamma_1, \ldots, \gamma_n$ in $\pi_1(\Sigma)$ such that, for
all~$i$, $\gamma_{i}^{+}=t_i$ and  $S$~{encloses}
$(\gamma_1,\ldots,\gamma_n)$. 

If $S$~encloses a curve~$\gamma$, it also encloses every~$\delta$
representing the same element in~$\Lambda$.

\medskip

Given an $n$-tuple 
$(\gamma^{+}_{1},\dots,\gamma^{+}_{n})$ 
in~$\Lambda$ and a surface~$S$ of finite type enclosing the tuple, we say that
$(\gamma^{+}_{1},\dots,\gamma^{+}_{n})$ is {\em $S$-cyclically oriented} if
the tuple $(\gamma^{+}_{1,S},\dots,\gamma^{+}_{n,S})$ is cyclically oriented in $\partial_\infty\pi_1(S)$, where $\gamma_{i,S}^{+}$ is the attracting fixed point of $\gamma_i$ in $\partial_\infty\pi_1(S)$.

\begin{remarks}\
	\begin{itemize}
		\item Note that if $\gamma_{i}^{n_i}=\delta_{i}^{ m_{i}}$, then
		$\gamma_{i,S}^{+} = \delta_{i,S}^{+}$ so that the definition makes
		sense.
		\item Of course, it is enough here to define cyclically oriented triples and
		the definition of cyclically oriented $n$-tuples follows by compatibility.
		\item When $\Sigma$~is already of finite type, $\Lambda$~is identified with
		a subset of $\partial_\infty \pi_1(\Sigma)$.
	\end{itemize}
\end{remarks}

\begin{proposition}\label{pro:finite-trick}
	If $(\gamma^{+}_{1},\dots,\gamma^{+}_{n})$ is $S_0$-cyclically oriented for a
	subsurface~$S_0$ of finite type enclosing them, it is $S$-cyclically
	oriented for any subsurface~$S$  of finite type
	enclosing them.
\end{proposition}
\begin{proof}
	Let $S_0,S_1$ be two incompressible finite type connected subsurfaces
	enclosing $(\gamma_1,\dots,\gamma_n)$. We find an incompressible finite type
	connected subsurface~$S$ containing both~$S_0$ and~$S_1$. Then there are
	embeddings $\iota_i\colon \partial_\infty \pi_1(S_i)\to \partial_\infty \pi_1(S)$,
	$i=0,1$, such that a tuple in $\partial_\infty \pi_1(S_i)$ is cyclically
	oriented if and only if its image under $\iota_i$ is. This proves the claim as
	$i_0(\gamma_{i,S_0}^{+})=i_1(\gamma_{i,S_1}^{+})$.
\end{proof}

As a conclusion, there is a well-defined cyclic ordering on~$\Lambda$. 
With this at hand we say that two elements $\gamma_0,\gamma_1\in\pi_1(\Sigma)$ \emph{intersect geometrically} if $(\gamma_0^+,\gamma_1^+,{\gamma_0^{-}},{\gamma_0^{-}})$ is cyclically oriented. Note that if two elements  intersect geometrically then neither of them can be peripheral.

We now use the cyclic ordering on~$\Lambda$
to define the notion of {\em $\Theta$-positive} representations.

In the next two definitions we assume that $\mG_0$ has a positive
structure relative to~\(\Theta\) and we let $\mG$, and ${\mathcal F}_\Theta=\cG/\mP_\Theta$ be as in Section~\ref{sec:prelim}.

\begin{definition}\label{def:posmap}
	Let $C$ be a set with a cyclic order.  A map $\xi\colon C \to \cF_\Theta$ is called
	\emph{positive} if every cyclically ordered tuple in $C$ is mapped to a
	positive tuple in~$\cF_\Theta$ by~$\xi$ (cf.\ also
	Section~\ref{sec:diamonds-positive-n}).
\end{definition}

\begin{definition}\label{def:posrep}
	A representation $\rho\colon \pi_1(\Sigma) \to \mG$ is said to be {\em
		$\Theta$-positive} if there exists a $\rho$-equivariant positive map from
	$\Lambda$ to~$\cF_\Theta$.
\end{definition}

We then have:
\begin{proposition}\label{prop:non-peri_to_loxo}
	Let $\rho$ be a $\Theta$-positive representation of $\pi_1(\Sigma)$ in
	$\mG$, if $\gamma$ is a non-peripheral element, then $\rho(\gamma)$ is
	$\Theta$-loxodromic, and $\xi$ maps attracting fixed points to attracting fixed points.\end{proposition}

\begin{proof}
	We can reduce using Proposition~\ref{pro:Red-finite} to the case when
	$\Sigma$ is a finite type surface. Then the result follows from
	\cite[Proposition 3.18]{Guichard:2021aa}.
\end{proof}

When $\Sigma$ is closed, every non-trivial element in $\grf$ is
non-peripheral. In fact, in that case, Definition~\ref{def:posrep} agrees with
the definition from \cite{Guichard:2021aa} (cf.\ Proposition~5.7 in that reference).

\section{Collar inequality for representations}
\label{sec:coll-ineq-repr}

In this section we collect the material of
Sections~\ref{sec:collar-inequality} and~\ref{sec:posit-repr-infin} in order
to produce Collar Lemmas for $\Theta$-positive representations.

Theorem~\ref{thm:collarineq}
applies to 
$\Theta$-positive representations: 
\begin{corollary}[{\sc Collar Inequality}]\label{cor:col-ineq}
	Let $\eta$ be a $\Theta$-compatible dominant form and let~$\theta$ be in~$\Theta$.  Assume that
	$\braket{\theta,\eta}>0$.
	Let~$\Sigma$ be a connected oriented (not necessarily of finite type)
	surface whose fundamental group contains a free group.
	Then given a positive representation~$\rho$ of
	$\pi_1(\Sigma)$, two loops $\gamma_0$ and $\gamma_1$ geometrically intersecting, 
	we have
	\begin{equation*} 
		\left(\frac{1}
		{\pp^\eta\left(\rho(\gamma_1)\right)}\right)^
		{{1}/{\dual{ h_\theta}{\eta}}}
		+\frac{1}{\chi_\theta
			\left(\rho(\gamma_0)\right)}< 1\ .
	\end{equation*}	
\end{corollary}

\begin{proof}
	Let~$S$ be a finite type surface enclosing~$\gamma_0$ and~$\gamma_1$.
	
	Let~$x$ be a point of intersection of~$\gamma_0$ and~$\gamma_1$, we choose
	(and denote them the same way) representatives~$\gamma_0$ and~$\gamma_1$ in
	$\pi_1(S,x)$. Let us denote by $\gamma_{i}^{\pm}$ the attracting/repelling
	fixed points of~$\gamma_i$ in $\partial_\infty \pi_1(S)$.  The intersection
	hypothesis implies that, up to exchanging $\gamma_1$ and $\gamma_1^{-1}$,
	the sextuple
	\begin{equation*}
		(\gamma_1^-,\ \gamma_0^-,\ \gamma_1^+,  \gamma_1(\gamma_0^+),\
		\gamma_0(\gamma_1^+),\ \gamma_0^+)\ ,
	\end{equation*}
	is a positive configuration in $\partial_\infty \pi_1(S)$ (see for instance
	\cite[Lemma 2.2]{Lee-Zhang:2017}).
	
	We denote by $a^+, a^-, b^+$ and $b^-$ the images of respectively $\gamma_{0}^{+}, \gamma_{0}^{-}, \gamma_{1}^{+}$ 
	and $\gamma_{1}^{-}$ under the limit map. We also write
	\[ A\defeq\rho(\gamma_0)\ , \ \ B\defeq\rho(\gamma_1)\ .\]
	By Proposition~\ref{prop:non-peri_to_loxo}, $a^+$ and~$a^-$ are the attracting
	and repelling fixed points of~$A$, and $b^+$ and~$b^-$ are those of~$B$.
	By positivity, it follows that
	\begin{equation*}
		(b^-,a^-,b^+,B(a^+), A(b^+), a^+)
	\end{equation*} 
	is also a positive configuration (see figure~\ref{fig:sext}). Then the corollary follows from Theorem~\ref{thm:collarineq}.
\end{proof}

Choosing the $\Theta$-compatible dominant form~$\eta$ in
Corollary~\ref{cor:col-ineq} to be equal to a fundamental
weight~$\omega_{\theta}$ we immediately get the following corollary which is
Corollary~\ref{coro:B} of the introduction.

\begin{corollary}\label{cor:weight-root-collar}
	Let~$\rho$ be a $\Theta$-positive representation of a surface group
	$\pi_1(\Sigma)$. Let~$\gamma_0$ and~$\gamma_1$ be two geometrically intersecting loops and~$\theta$ in~$\Theta$, then
	\begin{equation*} 
		\frac{1}{\pp^{\omega_{\theta}}(\rho(\gamma_0))}
		+\frac{1}{\chi_{\theta}
			(\rho(\gamma_1))}\ <\ 1\ .
	\end{equation*}
\end{corollary}

The inequality of this last corollary can be reformulated as
$$\left(\pp^{\omega_{\theta}}(\rho(\gamma_0))-1\right) \left( \chi_{\theta}
(\rho(\gamma_1))-1\right) > 1. $$

\begin{remark}\label{rem.positiveweights}	The previous corollary ---or Corollary~\ref{coro:B}--- is
	actually equivalent to Corollary~\ref{cor:col-ineq}, by Corollary~\ref{coro:pos-line-fund}. 
\end{remark}

\subsection{Comparison with other Collar Lemmas}\label{sec:Collar-discuss}
Collar Lemmas originate from the work of Keen in hyperbolic geometry. 
For the holonomy $\rho\colon \pi_1(\Sigma)\to \mathsf{PSL}_2(\mathbb{R})$ of  a
hyperbolic structure, denoting by $\ell(\rho(\gamma))$ the length of the
geodesic representative of $\rho(\gamma)$ in the hyperbolic surface, building
on the results she
proved in  \cite{Keen} the following  sharp inequality was deduced  (see
\cite[Section~6]{Matelski:1976},  \cite[Corollary 4.1.2]{Buser:2010}): for~$\gamma_0$ and~$\gamma_1$ geometrically intersecting 

\begin{equation}\label{ineq:classical-collarlemma}
	\sinh\left(
	\frac{1}{2}\ell(\rho(\gamma_0)) \right) \sinh\left(
	\frac{1}{2}\ell(\rho(\gamma_1)) \right)>1.
\end{equation}

Moving to the higher rank setting there are several possible generalizations of this result, as many possible quantities can be understood as length of an element with respect to a representation. One possible direction is to replace the length with a suitable Finsler translation length; results in this direction are discussed in Section~\ref{s.domination}. Our collar lemma, as well as its predecessors discussed in Section~\ref{s.rootweight} is a non-symmetric generalization, as it compares the character of a root with respect to that of a weight. 

This asymmetry between roots and weights is key: on the one hand only by controlling the root we can deduce closedness in the space of representations, on the other hand 
it is proven in \cite[Theorem 7.1]{Beyrer:2021aa} that for Hitchin representations in $\mathsf{PSL}_3(\mathbb R)$ no collar lemma comparing the roots of two  elements that intersect geometrically can exist, and in this respect our result, as well as the results discussed in Section~\ref{s.rootweight}, are optimal.

\subsubsection{Collar Lemmas comparing roots and weights for $\Theta$-positive representations}\label{s.rootweight}
Many instances of Collar Lemmas comparing roots and weights for special classes of $\Theta$-positive representations already appeared in the literature. None of these results are sharp, as the proofs always involve a crude minoration.

In the case of Hitchin representations into $\mathsf{PSL}_n(\mathbb{R})$  Lee and Zhang prove the following inequalities, for $k$ in $\{1,\ldots, n-1\}$  \cite[Proposition 2.12(1)]{Lee-Zhang:2017}

\[	\left(\pp^{\omega_{1}}(\rho(\gamma_0))-1\right)\left(\chi_{\alpha_k}
(\rho(\gamma_1))-1\right)>1\ \]
here we denote by $\{\alpha_1,\ldots,\alpha_{n-1}\}$
the simple roots 
of $\mathsf{PSL}_n(\mathbb{R})$
in the standard numeration (i.e.\ $\alpha_i$ is connected to $\alpha_{i\pm 1}$
in the Dynkin diagram) and denote $\omega_i\coloneqq \omega_{\alpha_i}$.

In the case of maximal representations into $\mathsf{Sp}_{2n}(\mathbb{R})$ Burger and Pozzetti  obtain \cite[Theorem 3.3(2)]{Burger-Pozzetti:2017}  
\[\pp^{\omega_{1}}(\rho(\gamma_0))^n\left(\chi_{\alpha_n}(\rho(\gamma_1))-1\right)>1\ .\]

In the case of $\Theta$-positive representations into $\mathsf{SO}(p,q)$ for $p\leq q$ Beyrer and Pozzetti  proved Corollary  \ref{cor:weight-root-collar} \cite[Theorem B]{BeyrerPozzetti}: for every $1\leq k\leq p-1$
\[	\left(\pp^{\omega_k}(\rho(\gamma_0))-1\right)\left(\chi_{\alpha_k}
(\rho(\gamma_1))-1\right)>1\ . \]

While all these  results, as well as ours,  rely on the positivity of the
sextuple $(a^+,b^-,a^-,b^+, B(a^+), A(b^+))$ the strategy of proofs is
different in the three cases. Our proof follows the approach outlined in
\cite{BeyrerPozzetti} with the important new contribution of the introduction
of photons which allows treating all roots simultaneously regardless of the
dimension of the associated root space, without relying on specific linear
realizations of the Lie groups.
As such our proof is uniform for all $\Theta$-positive representations, independent of the zoology of the group involved.

\subsubsection{Collar Lemmas through domination}\label{s.domination}	
For a Hitchin representation into $\mathsf{PSL}_3(\mathbb{R})$, or
a maximal representation in $\mathsf{SO}_0(2,n)$, Tholozan \cite[Corollary 4]{Tholozan:2017} and
Collier--Tholozan--Toulisse \cite[Corollary 6]{Collier:2019aa} used domination to deduce
a Collar Lemma: for a Hitchin representation~$\rho$ in $\mathsf{PSL}_3(\mathbb{R})$ Tholozan finds a Fuchsian representation whose spectrum dominates $\ell_1(A)\coloneqq \log \pp^{\omega_{1}}(A)$  and he deduces from the hyperbolic Collar Lemma 
\[\sinh\left(\frac{1}{4}\ell_1(\rho(\gamma_0))\right)\sinh\left(\frac{1}{4}\ell_1(\rho(\gamma_1))\right)>1.\]
For a maximal representation~$\rho$ in $\mathsf{SO}_0(2,n)$ Collier, Tholozan, and Toulisse find a Fuchsian representation whose length spectrum dominates $\ell_1=\log\pp^{\omega_1}$,  and deduce similarly
\[\sinh\left(\frac{1}{2}\ell_1(\rho(\gamma_0))\right)\sinh\left(\frac{1}{2}\ell_1(\rho(\gamma_1))\right)>1\ .\]

These Collar Lemmas are sharp, but since they don't control the root character, they do not guarantee that the limit of a converging sequence contains loxodromic elements in its image.

In the case of maximal representations into $\mathsf{Sp}_{2n}(\mathbb{R})$  Fanoni and Pozzetti study appropriate generalizations of ortho-geodesics in the symmetric space. Generalizing Basmajian's inequality they obtain \cite[Corollary 1.3]{Fanoni-Pozzetti:2020} that if $\gamma$ represents a simple closed curve, then the tubular neighborhood of width $w(\gamma)$ of a specific parallel set invariant by the element $\rho(\gamma)$ is embedded in the symmetric space $\rho(\Gamma)\backslash\mathsf{Sp}_{2n}(\mathbb{R})/\mathsf{U}_{n}$. In this result $w(\gamma)$ can be taken as $\sqrt n{\rm arccoth}(\exp(\ell^R(\rho(\gamma)/2\sqrt n)))$. This result has a compelling analogy with hyperbolic geometry; while it can be used to obtain bounds of the translation length of $\rho(\gamma_0)$ in terms of the translation length of $\rho(\gamma_1)$, these, similarly to those in the results mentioned above, do not help in the study of limits.

\subsubsection{Other Collar Lemmas}
Beyrer and Pozzetti show in \cite[Theorem 1.1]{Beyrer:2021aa} that
Collar Lemmas are not specific to $\Theta$-positive representations and
define other classes of representations in $\mathsf{PSL}_d(\mathbb R)$
that satisfy the inequality
\[ \left(\pp^{\omega_{k}}(\rho(\gamma_0))-1\right)\left(\chi_{\alpha_k}
(\rho(\gamma_1))-1\right)>1\ .\]

They exhibit in particular the class of $(k+2)$-positive representations
for which this Collar Lemma holds (see \cite[Corollary
6.20]{Beyrer:2021b}).  In particular $(k+2)$-positive representations
form open subsets of the representation variety, but never connected
components, outside the Hitchin component.

This highlights that, even though we use the Collar Lemma in Section~\ref{ssec:closed}
to prove the closedness of the space of positive representations, this
is not a mere consequence of the Collar Lemma, but really of the
combination of the structure of limits of positive representations
established in \cite[Proposition~6.4]{Guichard:2021aa} ---see also \cite[Theorem~B]{Beyrer:2021b}--- and the Collar
Lemma.

 \subsection{Closedness of the space of positive representations}\label{ssec:closed}
As a consequence of the Collar Lemma together with results of
\cite{Guichard:2021aa} we obtain:
\begin{corollary}\label{cor:closed}
	The space of positive representations is closed in the space of
	representations.
\end{corollary}
\begin{proof}
	Let $\seq{\rho}$ be a sequence of positive representations converging to a
	representation $\rho_\infty$. Let $\gamma_0$ and $\gamma_1$ be loops
	intersecting at least once. Let $\theta$ be an element of $\Theta$.  Since
	for every $\gamma$ in $\pi_1(\Sigma)$, the sequence
	$\{\pp^{\omega_\theta}(\rho_m(\gamma))\}_{m\in\mathbb N}$ is bounded by a
	constant $K(\gamma)$, it follows from the collar inequality
	(Corollary~\ref{cor:weight-root-collar}) that for all~$m$,
	\[
	\frac{1}{\chi_\theta(\rho_m(\gamma_1))}\leq
	1-\left(\frac{1}{K(\gamma_0)}\right)\ .
	\]
	As a consequence there is a positive~$\epsilon$, such that for all~$m$ and
	all~$\theta$ in~$\Theta$, we have
	\[
	\chi_\theta\left(\rho_m(\gamma_1)\right)\geq 1+\epsilon\ .
	\]
	Since the Jordan projection (and hence $\chi_\theta$) is continuous, it
	follows that
	\[
	\chi_\theta\left(\rho_\infty(\gamma_1)\right)\geq 1+\epsilon\ .
	\]
	Recall that $h$ is $\Theta$-loxodromic
	if and only if $\chi_\theta(h)>1$ for all~$\theta$ in~$\Theta$ (Proposition \ref{pro:loxo-alg}). 
	In particular $\rho_\infty(\gamma_1)$ is loxodromic. 
	Let  $\seq{x}$ and $\seq{y}$ be the repelling and attracting fixed points of $\rho_m(\gamma_1)$, then $\seq{x}$ and $\seq{y}$
	converge to, respectively,  the attracting and the repelling fixed
	points $x_\infty$ and $y_\infty$ of $\rho_\infty(\gamma_1)$ which are
	transverse. By \cite[Proposition 6.4]{Guichard:2021aa},  $\rho_\infty$ is positive. This concludes the proof.
\end{proof}
\begin{appendix}
\section{Extension to real closed fields}\label{app:real}
In this appendix, we explain how to extend the results obtained previously to all real closed fields by using the quantifier elimination Theorem of Tarski and Seidenberg. 
The importance of real closed field in the theory of surface group representations is outlined in the work of Brumfiel \cite{Brumfiel:1987} and more recently Burger, Iozzi, Parreau and Pozzetti \cite{Burger:2023}. 
We will not address here challenges on the structure of the character variety of positive representations itself but only focus on the extension of our main results to positive representations defined over real closed fields.

\subsection{Real closed fields} 
A {\em real closed field} is a totally ordered  field in which every positive element is a square and every odd degree polynomial has a root. Obviously $\mathbb R$ is a real closed field.

The {\em Tarski--Seidenberg quantifier elimination theorem} loosely says that any semi-algebraic statement holding over $\mathbb R$ holds for any real closed field $\mathbb F$. 
Recall that a {\em semi-algebraic set} over an ordered field $\mathbb K$ is a subset of $\mathbb K^n$ which can be defined by finitely many algebraic equalities and inequalities with coefficients in $\mathbb K$. We will use two important consequences of  Tarski--Seidenberg:  the \emph{Projection Theorem}, stating that the image of a semi-algebraic set by a polynomial map is also semi-algebraic \cite[Proposition 2.2.7]{Bochnak:1998}, and the \emph{transfer principle} stating that  a semi-algebraic set defined over $\mathbb R$ is empty if and only if for  any real closed extension $\mathbb F$ of $\mathbb R$ the subset of $\mathbb F^n$ defined by the same equalities and inequalities is empty \cite[Proposition 5.3.5]{Bochnak:1998}.

Our goals are  now the following
\begin{itemize}
	\item Define positive representations with values in a semisimple algebraic
	group defined over a real closed field $\mathbb F$ that is an extension of~\(\mathbb{R}\),
	\item Show that Theorems \ref{theo:A} and \ref{theo:Bprime} can be rephrased
	in terms of semi-algebraic subset, and thus hold over arbitrary real closed
	fields containing~\(\mathbb{R}\).
\end{itemize}

\subsection{Positive representations over real closed fields}
Let $\mG$ be a semisimple real algebraic group equipped with a positive
structure relative to~$\Theta$ and $\cF_\Theta$ the   flag manifold associated to $\mP_\Theta$. Our first result is: 
\begin{proposition}
	The set of positive $n$-tuples is a semi-algebraic subset of $\cF^{n}_{\Theta}$.
\end{proposition}
\begin{proof}
	The parameterization theorem of Guichard--Wienhard \cite[Theorem
	10.1]{GuichardWienhard_pos} parametrizes a positive diamond as the image by a
	polynomial map ---indeed the exponential map defined on a unipotent subalgebra is actually polynomial--- of a product of cones in $\mk u_\theta$ that are semi-algebraic and in fact defined by finitely
	many explicit inequalities \cite[Section~5]{GuichardWienhard_pos}. Thus, by the Projection Theorem a positive diamond is a semi-algebraic set.  
\end{proof}
Any semi-algebraic set defined over~$\mathbb{K}$ admits a natural $\mathbb
F$-extension for any real closed field $\mathbb F$ containing~$\mathbb{K}$, which amounts to considering the set defined by the same polynomial equalities and inequalities in $\mathbb F^n$.
Given a real closed field~$\mathbb F$ we denote by $\cF_\Theta(\mathbb F)$ the $\mathbb F$-extension of the flag manifold $\cF_\Theta$ on which $\mG(\mathbb F)$ acts and say that a triple in $\cF_\Theta(\mathbb F)^3$ is positive if it belongs to the $\mathbb F$-extension of the set of positive triples in $\cF^{3}_{\Theta}$ (cf.\ \cite[Example 6.17 (c)]{Burger:2023}). 

We can now extend verbatim the definition of positive representations into algebraic groups with coefficients in real closed fields since it only involves the notion of positive triples and quadruples.

\subsection{The main results for real closed fields}

We now state Theorem \ref{theo:A} and Theorem \ref{theo:Bprime} for real closed fields and prove the corresponding statements. 

\begin{theorem}[\sc Positivity of the cross-ratio]\label{theorem:Aprime}
	Let $\ms G$ be a semisimple algebraic group admitting
	a positive structure relative to~$\Theta$,  $\cF_\Theta$ be the  flag manifold
	associated to $\Theta$, $\lambda$ a $\Theta$-compatible proximal
	dominant weight, 
	$\bb^\lambda$  the associated cross-ratio. Then for any
	real closed field~$\mathbb F$ containing~\(\mathbb{R}\) and every  positive quadruple
	$(x,y,X,Y)$ in $ \cF_\Theta(\mathbb F)^4$ it holds
	\[\bb^\lambda (x,y,X,Y)>1\ . 
	\]
\end{theorem}
\begin{proof}
	It      follows from Equations~\eqref{e.projcr} and~\eqref{e.crgen} in Section~\ref{sec:posit-cross-ratio} that the cross-ratio $\bb^\lambda$ defines an algebraic function from the semi-algebraic subset $O$ of $\cF^{4}_{\Theta}$. 
	As a result the set 
	$$\{(x,y,X,Y)\in \cF^{4}_{\Theta} \mid (x,y,X,Y) \text{ is positive, }\bb^\lambda (x,y,X,Y)\leq 1 \}$$
	is a semi-algebraic subset defined over $\mathbb R$. Since by Theorem  \ref{theorem:Aprime} such set is empty over $\mathbb R$, it is empty over every real closed field $\mathbb F$ extending $\mathbb R$, which proves the desired statement.
\end{proof}	
\begin{remark}
In order to extend this result to any real closed field $\mathbb F$ containing the field $\mathbb K$ of definition of $\ms G$ it is necessary to show that also the cross-ratio is defined over  $\mathbb K$. This is true and is verified in \cite{burger:2026}.
\end{remark}	

We now turn to Theorem \ref{theo:Bprime}. Given a real closed field $\mathbb F$ and an element~$A$ in $ \mG(\mathbb F)$ we say that a point~$a^+$ in $\cF_\Theta(\mathbb F)$ is an attracting fixed point for~$A$ (respectively $a^-$~is a repelling fixed point) if it is a fixed point for the $A$ action on $\cF_\Theta(\mathbb F)$ and the action of $\Ad(g)$ on the Lie algebra of the unipotent radical of the stabilizer of $a^+$ has all eigenvalues of modulus in $\mathbb F$ strictly larger than one (respectively strictly smaller than one). Then the subset of $\mG(\mathbb F)\times\cF_\Theta(\mathbb F)$ consisting of pairs $(A,a)$ such that $a$ is an attracting fixed point for $A$ (respectively repelling) is semi-algebraic, and corresponds to the $\mathbb F$-extension of the subset of $\mG\times\cF_\Theta$ consisting of pairs satisfying the same property  \cite[Section~4.4]{Burger:2023}.

\begin{theorem}
	Let~$\mathbb F$ be a real closed field containing~\(\mathbb{R}\).  For every pair of $\Theta$-loxodromic elements~$A$ and~$B$ in $\mG(\mathbb F)$ with attracting and repelling  fixed points $(a^+, a^-)$
	and $(b^+, b^-)$ such that the sextuple
	$(a^+, b^-, a^-, b^+, B(a^+), A(b^+))$ in $\cF_\Theta(\mathbb F)^6$ is positive and any $\theta$ in $\Theta$ it holds
	\begin{equation*}
		\frac{1}
		{\pp^{\omega_\theta}\left(A\right)}
		+\frac{1}{\chi_{\theta}
			\left(B\right)}< 1\ .
	\end{equation*}
\end{theorem}

\begin{proof}
	We need, again 	to show that the set $P$ of pairs $(A,B)$ in $\mG^2$ admitting attracting and repelling fixed points such that the sextuple 
	$$(a^+, b^-, a^-, b^+, B(a^+), A(b^+))$$ is positive and for which the Equation 	\begin{equation}\label{e.collarF}
		\frac{1}
		{\pp^{\omega_\theta}\left(A\right)}
		+\frac{1}{\chi_{\theta}
			\left(B\right)}\geq  1
	\end{equation}  
	holds
	is a semi-algebraic set.
	
	On the one hand, let $P_1$ be the set of pairs $(A,B)$ in $\mG^2$ admitting
	attracting and repelling fixed points such that 	the sextuple $(a^+,
	b^-, a^-, b^+, B(a^+), A(b^+))$ is positive, then $P_1$~is a semi-algebraic set. Indeed, this follows from the discussion above concerning attracting fixed points, from  the action of $\mG$ on $\cF_\Theta$ being algebraic, and the set of positive $6$-tuples being semi-algebraic. 
	
	On the other hand, the set $P_2$ of pairs $(A,B)$ in $\mG^2$ such that Equation
	\eqref{e.collarF} holds is a semi-algebraic set, since both the period
	$\pp^{\omega_\theta}$ and the $\theta$-character $\chi_\theta$ are
	semi-algebraic functions. Indeed, the periods are defined by the  expressions $\bb^{\omega_\theta}(b^+,b^-,a^+,B(a^+)),$
	which are semi-algebraic since the cross-ratio $\bb^{\omega_\theta}$ and the function associating to $B$ 
	its attracting and repelling fixed points is semi-algebraic; that the character $\chi_\theta$ is semi-algebraic follows from  {\cite[Proposition 4.7]{Burger:2023}}.
	
	Now the set $P=P_1\cap P_2$ is semi-algebraic. Theorem \ref{theo:Bprime} guarantees that it is empty over $\mathbb R$, and it thus follows from the transfer principle that it is empty over any real closed field $\mathbb F$ extending $\mathbb R$, which concludes the proof.  	
\end{proof}

\section{Proximal representation associated to fundamental weight}
\label{app2_prox_rep}

In this appendix we prove the following assertion stated in Section~\ref{sec:line-forms-posit}.

\begin{theorem}
	\label{app2_th_rep_prox}
	If the dimension of \(\mathfrak{g}_{\theta}\) is~\(1\), then the fundamental weight \(\omega_\theta\) is a proximal dominant weight.
\end{theorem}

We will use for~\(\mathfrak{g}_0\) the notation introduced in Section~\ref{sec:prelim} except that the set of restricted roots will be denoted by \(\Sigma(\mathfrak{g}_0, \mathfrak{a})\).

Denote by \(\mathfrak{g}_{\mathbb{C}}\) the complexification of~\(\mathfrak{g}_0\).  It is a complex semisimple Lie algebra equipped with an antilinear automorphism \(\tau\colon \mathfrak{g}_{\mathbb{C}} \to \mathfrak{g}_{\mathbb{C}}\) whose fixed point set is equal to~\(\mathfrak{g}_0\).

Let \(\mathfrak{t}\subset \mathfrak{z}_{\mathfrak{k}}(\mathfrak{a})\) be a maximal Abelian subalgebra.  Then the complexification \(\mathfrak{h}_{\mathbb{C}}\) of \(\mathfrak{a}\oplus \mathfrak{t}\) is a Cartan subalgebra of \(\mathfrak{g}_{\mathbb{C}}\).  We will denote by \(\Sigma( \mathfrak{g}_{\mathbb{C}}, \mathfrak{h}_{\mathbb{C}})\) the corresponding root system.  One has hence the equality
\[ \mathfrak{g}_{\mathbb{C}} =  \mathfrak{h}_{\mathbb{C}} \oplus \bigoplus_{ \alpha\in \Sigma( \mathfrak{g}_{\mathbb{C}}, \mathfrak{h}_{\mathbb{C}})}  \mathfrak{g}_{\mathbb{C},\alpha}.\]

The complex conjugation~\(\tau\) restricts to a complex conjugation on~\(\mathfrak{h}_{\mathbb{C}}\) and induces a complex conjugation on~\(\mathfrak{h}_{\mathbb{C}}^{*} = \hom(\mathfrak{h}_{\mathbb{C}}, \mathbb{C}) \);
 all these involutions will also be denoted by~\(\tau\).  For all~\(\alpha\) in \(\Sigma(\mathfrak{g}_{\mathbb{C}}, \mathfrak{h}_{\mathbb{C}})\), one has \(\tau(\alpha) \in \Sigma(\mathfrak{g}_{\mathbb{C}}, \mathfrak{h}_{\mathbb{C}})\) and
\[\mathfrak{g}_{\mathbb{C}, \tau(\alpha)} = \tau( \mathfrak{g}_{\mathbb{C}, \alpha}).\]

We will denote by~\(\psi\) the restriction map from \(\mathfrak{h}_{\mathbb{C}}\) to~\(\mathfrak{a}\):
\begin{align*}
	\psi\colon \Sigma( \mathfrak{g}_{\mathbb{C}}, \mathfrak{h}_{\mathbb{C}}) & \rightarrow \Sigma( \mathfrak{g}_0, \mathfrak{a}) \cup \{0\}\\
	\alpha &\mapsto \alpha|_{\mathfrak{a}}.
\end{align*}
The roots such that \(\psi(\alpha)=0\) are called compact.  One has also that \(\psi\circ \tau =\psi\): for every root~\(\alpha\) in \(\Sigma( \mathfrak{g}_{\mathbb{C}}, \mathfrak{h}_{\mathbb{C}})\), the equality \(\tau(\alpha)|_{\mathfrak{a}} = \alpha|_{\mathfrak{a}}\) holds.  For every root~\(\beta\) in \(\Sigma(\mathfrak{g}_0, \mathfrak{h})\), the space \(\mathfrak{g}_{\beta}\) can be recovered from the root spaces of the complexification:
\[\mathfrak{g}_{\beta} = \Bigl( \bigoplus_{{\alpha\in \psi^{-1}(\beta)}} \mathfrak{g}_{\mathbb{C}, \alpha}\Bigr) \cap \mathfrak{g}.\]

From this equality, we get the following result:
\begin{lemma}
	\label{app2_lemma_mult_one}
	Suppose that \(\dim_{\mathbb{R}} \mathfrak{g}_{\theta} =1\).  Then there is a unique root~\(\hat{\theta}\) in \(\psi^{-1}(\theta)\), and one has \(\tau( \hat{\theta}) = \hat{\theta}\).
\end{lemma}

Let us also denote by \(\Sigma^+( \mathfrak{g}_0, \mathfrak{a})\) and \(\Sigma^+( \mathfrak{g}_{\mathbb{C}}, \mathfrak{h}_{\mathbb{C}})\) choices of positive roots systems in those root systems.  More precisely those subsets of the root systems are the positive elements with respect to a lexicographic order induced by the choice of a basis of~\(\mathfrak{a}\) followed by the choice of a basis of \(\sqrt{-1} \mathfrak{t}\).  One gets that  \(\psi^{-1}(\Sigma^+( \mathfrak{g}_0, \mathfrak{a})) \subset \Sigma^+( \mathfrak{g}_{\mathbb{C}}, \mathfrak{h}_{\mathbb{C}})\) and \(\psi( \Sigma^+( \mathfrak{g}_{\mathbb{C}}, \mathfrak{h}_{\mathbb{C}})) \subset \Sigma^+( \mathfrak{g}_0, \mathfrak{a}) \cup \{0\}\).
The corresponding sets of simple roots are denoted \(\Delta( \mathfrak{g}_0, \mathfrak{a})\) and \(\Delta( \mathfrak{g}_{\mathbb{C}}, \mathfrak{h}_{\mathbb{C}})\).  Note that, for every simple root \(\alpha\) in \(\Delta( \mathfrak{g}_{\mathbb{C}}, \mathfrak{h}_{\mathbb{C}})\) such that \(\psi(\alpha)\neq 0\), then the conjugate root \(\tau(\alpha)\) also belongs to \(\Delta( \mathfrak{g}_{\mathbb{C}}, \mathfrak{h}_{\mathbb{C}})\).  More generally, for every~\(\alpha\) in \(\Sigma^{+}( \mathfrak{g}_{\mathbb{C}}, \mathfrak{h}_{\mathbb{C}})\), if \(\psi(\alpha)\neq 0\), then \(\tau(\alpha)\) belongs to \(\Sigma^{+}( \mathfrak{g}_{\mathbb{C}}, \mathfrak{h}_{\mathbb{C}})\).

\begin{lemma}
	\label{app2_lemma_simple_to_simple}
	For every~\(\alpha\) in \(\Delta( \mathfrak{g}_0, \mathfrak{a})\), there is \(\tilde{\alpha}\) in \(\Delta( \mathfrak{g}_{\mathbb{C}}, \mathfrak{h}_{\mathbb{C}})\) such that \(\psi( \tilde{\alpha}) = \alpha\).
\end{lemma}
\begin{proof}
	Let \(\tilde{\alpha}\) be the minimal element (with respect to the chosen lexicographic order) in \(\psi^{-1}(\alpha) \subset \Sigma^+( \mathfrak{g}_{\mathbb{C}}, \mathfrak{h}_{\mathbb{C}})\).  Assume by contradiction that \(\tilde{\alpha}\) is not simple.  There would then exist \(\gamma\) and \(\eta\) in \(\Sigma^+( \mathfrak{g}_{\mathbb{C}}, \mathfrak{h}_{\mathbb{C}})\) such that \(\tilde{\alpha}=\gamma+\eta\).  One has \(\alpha=\psi( \tilde{\alpha}) = \psi(\gamma) + \psi(\eta)\); since \(\alpha\)~is simple and since \(\psi( \Sigma^+( \mathfrak{g}_{\mathbb{C}}, \mathfrak{h}_{\mathbb{C}})) \subset \Sigma^+( \mathfrak{g}_0, \mathfrak{a}) \cup \{0\}\), one of the elements \(\psi(\gamma)\) or \(\psi(\eta)\) must be zero.  For definiteness, assume \(\psi(\eta)=0\).  We then have \(\psi(\gamma)= \alpha\) and \(\gamma\)~is strictly smaller than~\(\tilde{\alpha}\), contradicting the definition of~\(\tilde{\alpha}\).
\end{proof}

Recall the basis \((h_\alpha)_{\alpha \in \Delta( \mathfrak{g}_0, \mathfrak{a})}\) of~\(\mathfrak{a}\) (Section~\ref{sec:skd-triples}): For every~\(\alpha\) in \(\Delta( \mathfrak{g}_0, \mathfrak{a})\), \(h_\alpha\) belongs to an \(\mathfrak{sl}_2\)-triple \((x_\alpha, x_{-\alpha}, h_\alpha)\) with \(x_\alpha\in \mathfrak{g}_\alpha\) and \(x_{-\alpha}\in \mathfrak{g}_{-\alpha}\).  We will always assume that \(x_{-\alpha}= -\sigma(x_\alpha)\) where $\sigma$ is the Cartan involution (Section~\ref{sec:roots});
it is a classical fact that \(h_\alpha\) does not depend on the choice of \(x_\alpha\) and that \(x_\alpha\) can be freely chosen in a sphere inside \(\mathfrak{g}_\alpha\) \cite[Prop.~6.52]{Knapp_LieGrp}.  Similarly, we will denote by \((\tilde{h}_{\alpha})_{\alpha \in \Delta( \mathfrak{g}_{\mathbb{C}}, \mathfrak{h}_{C})}\) with respect to the complex Lie algebra.  Note that for any simple root \(\alpha\) in \(\Delta( \mathfrak{g}_{\mathbb{C}}, \mathfrak{h}_{\mathbb{C}})\), if \(\psi(\alpha)\neq 0\), then \(\tau( \tilde{h}_{\alpha}) = \tilde{h}_{\tau(\alpha)}\).

\begin{lemma}
	\label{app2_coroots_in_span}
	For every~\(\alpha\) in \(\Delta( \mathfrak{g}_0, \mathfrak{a})\), let \(\tilde{\alpha}\) be a simple root such that \(\psi( \tilde{\alpha}) =\alpha\) (Lemma~\ref{app2_lemma_simple_to_simple}). Then the element \(h_\alpha\) belongs to the complex vector space spanned by \(\tilde{h}_{\tilde{\alpha}}\) and  \(\tilde{h}_{\tau(\tilde{\alpha})}\).
\end{lemma}
\begin{proof}
  If \(\tau( \tilde{\alpha}) = \tilde{\alpha}\), then the space \(\mathfrak{g} \cap \mathfrak{g}_{\mathbb{C}, \tilde{\alpha}}\) is nontrivial and contained in \(\mathfrak{g}_\alpha\).  Choosing \(x_\alpha\) in this intersection, we get the equality \(h_\alpha = \tilde{h}_{\tilde{\alpha}}\).

	If \(\tau( \tilde{\alpha}) \neq \tilde{\alpha}\), then \(\tau(\tilde{\alpha})\) is a simple root distinct from~\(\tilde{\alpha}\).  This implies in particular that \([ \mathfrak{g}_{\tilde{\alpha}}, \mathfrak{g}_{ -\tau( \tilde{\alpha})}] =0\) and \([ \mathfrak{g}_{\tau(\tilde{\alpha})}, \mathfrak{g}_{ - \tilde{\alpha}}] =0\).  The intersection \(\mathfrak{g} \cap ( \mathfrak{g}_{\mathbb{C}, \tilde{\alpha}} \oplus \mathfrak{g}_{\mathbb{C}, \tau( \tilde{\alpha})} )\) is nontrivial, this means that we can choose \(x_\alpha\) in this intersection, i.e.\ a linear combination of \(x_{\tilde{\alpha}}\) and \(x_{\tau(\tilde{\alpha})}\).  One then gets that \(h_\alpha\) is a linear combination of \(\tilde{h}_{\tilde{\alpha}}\) and \(\tilde{h}_{\tau(\tilde{\alpha})}\).
\end{proof}

The fundamental weights \((\omega_\alpha)_{\alpha\in \Delta( \mathfrak{g}_0, \mathfrak{a})}\) are the elements of the basis dual to the basis \((h_\alpha)_{\alpha \in \Delta( \mathfrak{g}_0, \mathfrak{a})}\): \(\omega_\alpha( h_{\alpha'}) = \delta_{\alpha, \alpha'}\) for every~\(\alpha\), \(\alpha'\) in \(\Delta( \mathfrak{g}_0, \mathfrak{a})\).  Similarly, we define fundamental weights \(( \tilde{\omega}_\alpha)_{\alpha \in \Delta( \mathfrak{g}_{\mathbb{C}}, \mathfrak{h}_{\mathbb{C}})}\).
One has \(\tau( \tilde{\omega}_{\alpha}) =\tilde{\omega}_{\tau( \alpha)}\).

\begin{lemma}
	\label{app2_lemma_fond_weight_descend}
	Suppose that \(\dim_{\mathbb{R}} \mathfrak{g}_{\theta} =1\) and let~\(\hat{\theta}\) be the unique root in \(\psi^{-1}(\theta)\).  Then the restriction \(\tilde{\omega}_{ \hat{\theta}}|_{\mathfrak{a}}\) is equal to~\(\omega_\theta\).
\end{lemma}
\begin{proof}
	One has \(\tilde{\omega}_{ \hat{\theta}} ( h_\theta) = \tilde{\omega}_{ \hat{\theta}} ( \tilde{h}_{\hat{\theta}}) = 1\) (since \(\tilde{h}_{\hat{\theta}} = h_\theta\)).  If \(\alpha\) belongs to \(\Delta( \mathfrak{g}_0, \mathfrak{a}) \smallsetminus \{\theta\}\), one gets \(\tilde{\omega}_{ \hat{\theta}} ( h_\alpha) = 0 \) using Lemma~\ref{app2_coroots_in_span} and 
	 \(\tilde{\omega}_{ \hat{\theta}} (\tilde{h}_{\tilde{\alpha}})=0\) and \(\tilde{\omega}_{ \hat{\theta}} (\tilde{h}_{\tau(\tilde{\alpha})})=0\).
\end{proof}

This implies that the irreducible representation of \(\mathfrak{g}_{\mathbb{C}}\) whose highest weight is \(\tilde{\omega}_{\hat{\theta}}\) is the complexification of the irreducible representation of \(\mathfrak{g}_0\) whose highest weight is~\(\omega_{\theta}\).

\begin{proof}[Proof of Theorem~\ref{app2_th_rep_prox}]
	Call~\(V_{\mathbb{C}}\) the underlying space of the irreducible representation \(\rho_{\mathbb{C}}\colon \mathfrak{g}_{\mathbb{C}}\to \mathfrak{gl}(V_{\mathbb{C}})\) with highest weight~\(\tilde{\omega}_{\hat{\theta}}\).  Since \(\tau( \tilde{\omega}_{\hat{\theta}}) = \tilde{\omega}_{\tau(\hat{\theta})} = \tilde{\omega}_{\hat{\theta}}\), there is a real form~\(V\) of~\(V_{\mathbb{C}}\) that is invariant by \(\rho_{\mathbb{C}}( \mathfrak{g}_0)\).  Furthermore, the representation \(\rho\colon \mathfrak{g}_0 \to \mathfrak{gl}(V)\) induced by restriction is the irreducible representation with highest weight~\(\omega_\theta\).

	It is known that the weight space \(V_{\mathbb{C}}^{\tilde{\omega}_{\hat{\theta}}}\) is \(1\)-dimensional and that the other weights in \(V_{\mathbb{C}}\) are of the form \(\tilde{\omega}_{\hat{\theta}}-\gamma\) where \(\gamma\)~is a sum of simple roots (possibly with repetition) \cite[Theorem 5.5]{Knapp_LieGrp}.  Furthermore for every simple root \(\alpha\), \(\tilde{\omega}_{\hat{\theta}}-\alpha\) is a weight in \(V_{\mathbb{C}}\) if and only if \(\tilde{\omega}_{\hat{\theta}}( \tilde{h}_\alpha)>0\) and this happens if and only if \(\alpha=\hat{\theta}\).  Thus, the weights in \(V_{\mathbb{C}}\) are \(\tilde{\omega}_{\hat{\theta}}\) (with multiplicity~\(1\)) \(\tilde{\omega}_{\hat{\theta}} -\hat{\theta}\), and \(\tilde{\omega}_{\hat{\theta}} -\hat{\theta}-\gamma\) where \(\gamma\)~is a sum of simple roots.

	From this we obtain that, for every element~\(h\) in the open Weyl chamber \(\mathring{\mathfrak{a}}{\vphantom{\mathfrak{a}}}^{+}\), the element \(\rho(h)\) has eigenvalue \(\tilde{\omega}_{\hat{\theta}}(h)\) with multiplicity~\(1\) and all its other eigenvalues are \(\leq \tilde{\omega}_{\hat{\theta}}(h) - \hat{\theta}(h)< \tilde{\omega}_{\hat{\theta}}(h)\).  This means that \(\rho(h)\) is proximal.
\end{proof}

\end{appendix}
\bibliographystyle{amsplain}
\bibliography{collar.bib}

\end{document}